\documentclass[11pt]{article} 

\usepackage[utf8]{inputenc} 
\usepackage[a4paper,margin=1in]{geometry} 
\usepackage{amsfonts} 
\usepackage{amsthm} 
\usepackage{amsmath} 
\usepackage{amssymb} 
\usepackage{enumerate} 
\usepackage{natbib} 
\usepackage{booktabs} 
\usepackage{lscape} 
\usepackage{multirow} 
\usepackage{graphicx} 
\usepackage[section]{placeins} 
\usepackage{commath} 
\usepackage{bm} 
\usepackage[bookmarks]{hyperref} 

\usepackage{amsthm}
\makeatletter
\def\th@plain{%
  \thm@notefont{}
  \itshape 
}
\def\th@definition{%
  \thm@notefont{}
  \normalfont 
}
\makeatother

\newtheorem{corollary}{Corollary}
\newtheorem{lemma}{Lemma}
\newtheorem{thm}{Theorem}
\newtheorem{proposition}{Proposition}
\newtheorem{assumption}{Assumption}

\numberwithin{equation}{section} 

\DeclareMathOperator*{\argmin}{argmin}

\begin{document}
\title{Uniform Inference in High-Dimensional Dynamic Panel Data Models}
\author{\textsc{Anders Bredahl Kock\thanks{%
Aarhus University and CREATES. Email: akock@creates.au.dk. Financial support from the Center for Research in the Econometric Analysis of Time Series (grant DNRF78) is gratefully acknowledged.
}} \and \textsc{Haihan Tang\thanks{ %
Cambridge University, Faculty of Economics. Email: hht23@cam.ac.uk.}}}
\date{\today} 
\maketitle


\begin{abstract}
We establish oracle inequalities for a version of the Lasso in high-dimensional fixed effects dynamic panel data models. The inequalities are valid for the coefficients of the dynamic and exogenous regressors. Separate oracle inequalities are derived for the fixed effects. Next, we show how one can conduct uniformly valid simultaneous inference on the parameters of the model and construct a uniformly valid estimator of the asymptotic covariance matrix which is robust to conditional heteroskedasticity in the error terms. Allowing for conditional heteroskedasticity is important in dynamic models as the conditional error variance may be non-constant over time and depend on the covariates. Furthermore, our procedure allows for inference on high-dimensional subsets of the parameter vector of an increasing cardinality. We show that the confidence bands resulting from our procedure are asymptotically honest and contract at the optimal rate. This rate is different for the fixed effects than for the remaining parts of the parameter vector. 
\end{abstract}

\section{Introduction}

Dynamic panel data models are widely used in economics and social sciences. They are extremely popular as workers, firms, and countries often differ due to unobserved factors. Furthermore, these units are often sampled repeatedly over time in many modern applications thus allowing to model the dynamic development of these. However, so far no work has been done on how to conduct inference in the high-dimensional dynamic fixed effects model
\begin{align}
\label{eqn dgp}
y_{i,t}=\sum_{l=1}^{L}\alpha_{l}y_{i,t-l}+x_{i,t}'\beta+\eta_{i}+\varepsilon_{i,t},\ i=1,...,N,\ \text{ and }\ t=1,...,T 
\end{align}
where the presence of $L$ lags of $y_{i,t}$ allows for autoregressive dependence of $y_{i,t}$ on its own past. $x_{i,t}$ is a $p_{x}\times 1$ vector of exogenous variables and $\eta_i, i=1,....,N$ are the $N$ individual specific fixed effects while $\varepsilon_{i,t}$ are idiosyncratic error terms.  Applications of panel data are widespread: ranging from wage regressions where one seeks to explain worker's salary, to models of economic growth determining the factors that impact growth over time of a panel of countries as in \cite{islam1995growth}. 

Recent years have witnessed a surge in availability of big data sets including many explanatory variables. For example, \cite{de2012genes} have considered the effect of genes on happiness/life satisfaction. Controlling for many genes simultaneously clearly results in a vast set of explanatory variables, hence calling for techniques which can handle such a setting. High-dimensionality may also arise out of a desire to control for flexible functional forms by including various transformations, such as cross products, of the available explanatory variables. In the specific context of panel data models \cite{andersonbentzendalgaardselaya2012} investigated the causal effect of lightning density on economic growth using a US panel data set. These authors had access to a big set of control variables compared to the sample size. For this reason, they decided to investigate the effect of lightning using several subsets of control variables instead of including all control variables simultaneously as one would ideally do. In this paper we show how one can achieve this ideal by proposing an inferential procedure for high-dimensional dynamic panel data models.

Much progress has also been made on the methodological side in the last decade. Among the most popular procedures is the Lasso of \cite{tibshirani1996} which sparked a lot of research on its properties. However, until recently, not much work had been done on inference in high-dimensional models for Lasso-type estimators as these possess a rather complicated distribution even in the low dimensional case, see \cite{knight2000}. This problem has been cleverly approached by unpenalized estimation after double selection by \cite{bellonichenchernozhukovhansen2012,belloni2014inference} or by desparsification in \cite{zhang2014confidence, vandegeerbuhlmannritovdezeure2014, javanmard2013confidence, caner2014asymptotically}. 

The focus in the above mentioned work has been almost exclusively on independent data and often on the plain linear regression model while high-dimensional panel data has not been treated. Exceptions are \cite{kock2013} and \cite{belloni2014inference} who have established oracle inequalities and asymptotically valid inference for a low-dimensional parameter in \textit{static} panel data models, respectively. \cite{caner2014adaptive} have studied the properties of penalized GMM, which can be used to estimate dynamic panel data models, in the case of fewer parameters than observations. To the best of our knowledge, no research has been conducted on inference in high-dimensional dynamic panel data models. Note that high-dimensionality may arise from three sources in the dynamic panel data model (\ref{eqn dgp}). These sources are the coefficients pertaining to the lagged left hand side variables ($\alpha_l$), the exogenous variables ($\beta$), as well as the fixed effects ($\eta_i$). In particular, we shall see that (joint) inference involving an $\eta_i$ behaves markedly different from inference only involving $\alpha_l$'s and $\beta$. Furthermore, panel data differ from the classic linear regression model in that one does not have independence across $t=1,...,T$ for any $i$ as consecutive observations in time can be highly correlated for any given individual. Ignoring this dependence may lead to gravely misleading inference even in low-dimensional panel data models. For that reason we shall make \textit{no} assumptions on this dependence structure across $t=1,...,T$ for the $x_{i,t}$. Static panel data models are a special case of (\ref{eqn dgp}) corresponding to $\alpha_l=0,\ l=1,...,L$.

Traditional approaches to inference in low-dimensional static panel data models have considered the $N$ fixed effects $\eta_i$ as nuisance parameters which have been removed by taking either first differences or demeaning the data over time for each individual $i$, see e.g., \cite{wooldridge2010econometric, arellano2003panel, baltagi2008econometric}. In this paper we take the stand that the fixed effects may be of intrinsic interest. Thus we do not remove them by first differencing or demeaning. This allows us to test hypothesis simultaneously involving $\alpha, \beta$ and $\eta$.

The two most common assumptions on the unobserved heterogeneities, $\eta_i$, are the random and fixed effects frameworks. In the former, the $\eta_i$ are required to be uncorrelated with the remaining explanatory variables while the latter does not impose any restrictions. Ruling out any correlation between the $\eta_i$ and the other covariates often unreasonable. In this paper we strike a middle ground between the random and fixed effects setting: we do not require zero correlation between the unobserved heterogeneities and the other covariates, however we shall impose that $(\eta_1,....,\eta_N)$ is weakly sparse in a sense to be made precise in Section \ref{weaksparsity}. We still refer to the $\eta_i$ as fixed effects as we treat them as parameters to be estimated as is common in fixed effects settings. However, the reader should keep in mind that our setting is actually intermediate between the random and fixed effects setting. 


In an interesting recent paper dealing with the the low-dimensional case, \cite{bonhomme2012grouped} have assumed a different type of structure, namely grouping, on the fixed effects. However, in the high-dimensional setting we are considering, weak sparsity works well as just explained.

Our inferential procedure is closest in spirit to the one in \cite{vandegeerbuhlmannritovdezeure2014}, which in turn builds on \cite{zhang2014confidence}, who cleverly used nodewise regressions to \textit{desparsify} the Lasso and to construct an approximate inverse of the non-invertible sample Gram matrix in the context of the linear regression model. In particular, we show how nodewise regressions can be used to construct one of the blocks of the approximate inverse of the empirical Gram matrix in dynamic panel data models. As opposed to \cite{vandegeerbuhlmannritovdezeure2014}, we do not require the inverse covariance matrix of the covariates to be exactly sparse. It suffices that the rows of the inverse covariance matrix are weakly sparse. Thus, none of its entries needs to be zero.

We contribute by first establishing an oracle inequality for a version of the Lasso in dynamic panel data models for all groups of parameters. As can be expected, the fixed effects turn out to behave differently than the remaining parameters. Next, we show how joint asymptotically gaussian inference may be conducted on the three types of parameters in (\ref{eqn dgp}). In particular, we show that hypotheses involving an increasing number of parameters can be tested and provide a uniformly consistent estimator of the asymptotic covariance matrix which is robust to conditional heteroskedasticity. Thus, we introduce a feasible procedure for inference in high-dimensional heteroskedastic dynamic panel data models. Allowing for conditional heteroskedasticity is important in dynamic models like the one considered here as the conditional variance is known to often depend on the current state of the process, see e.g. \cite{engle1982autoregressive}. Thus, assuming the error terms to be independent of the covariates with a constant variance is not reasonable. Next, we show that confidence bands constructed by our procedure are asymptotically honest (uniform) in the sense of \cite{li1989honest} over a certain subset of the parameter space. Finally, we show that the confidence bands have uniformly the optimal rate of contraction for all types of parameters. Thus, the honesty is not bought at the price of wide confidence bands as is the case for sparse estimators, c.f. \cite{potscher2009confidence}. Simulations reveal that our procedure performs well in terms of size, power, and coverage rate of the constructed intervals.

The rest of the paper is organized as follows. Section \ref{model} introduces the estimator and provides an oracle inequality for all types of parameters. Next, Section \ref{sec inference} shows how limiting gaussian inference may be be conducted and provides a feasible estimator of the covariance matrix which is robust to heteroskedasticity even in the case where the number of parameter estimates we seek the limiting distribution for diverges with the sample size. Section \ref{unif} shows that confidence intervals constructed by our procedure are honest and contract at the optimal rate for all types of parameters. Section \ref{monte} studies our estimator in Monte Carlo experiments while Section \ref{conclusion} concludes. All the proofs of our results are deferred to Appendix A; Appendix B contains further auxiliary lemmas needed in Appendix A.

\section{The Model}\label{model}

\subsection{Notation}

For $x \in \mathbb{R}^n$, let $\|x\|_0=\sum_{i=1}^{n}1(x_i\neq 0)$, $\|x\|=\sqrt{\sum_{i=1}^{n}x_i^2}$, $\|x\|_1=\sum_{i=1}^n| x_{i}|$ and $\|x\|_{\infty}=\max_{1\leq i\leq n}| x_{i}|$ denote the $\ell_0$, $\ell_2$, $\ell_1$ and $\ell_\infty$ norms, respectively. Let $e_m$ denote the unit column vector with $m$th entry being 1 in some Euclidean space whose dimension depends on the context. If the argument of $\|\cdot\|_{\infty}$ is a matrix, then $\|\cdot\|_{\infty}$ denotes the maximal absolute element of the matrix. For some generic set $R\subseteq \{1,\ldots, n\}$, let $x_R\in \mathbb{R}^{|R|}$ denote the vector obtained by extracting the elements of $x\in \mathbb{R}^n$ whose indices are in $R$, where $|R|$ denotes the cardinality of $R$; $R^c= \{1,\ldots, n\} \setminus R$. For an $n\times n$ matrix $A$, $A_R$ denotes the submatrix consisting of the rows and columns indexed by $R$. $\otimes$ is the Kronecker product. Let $a \vee b$ and $a \wedge b$ denote $\max(a,b)$ and $\min(a,b)$, respectively. For two real sequences $(a_n)$ and $(b_n)$, $a_n\lesssim b_n$ means that $a_n\leq Cb_n$ for some fixed, finite and positive constant $C$ for all $n\geq 1$. For two deterministic sequences $a_n$ and $b_n$ we write $a_n\asymp b_n$ if there exist constants $0< a_1\leq a_2$ such that $a_1b_n\leq a_n\leq a_2b_n$ for all $n\geq 1$. $\text{sgn}(\cdot)$ is the sign function. $\text{maxeval}(\cdot)$ and $\text{mineval}(\cdot)$ are the maximal and minimal eigenvalues of the argument, respectively. For some vector $x\in \mathbb{R}^n$, $\text{diag}(x)$ gives a $n\times n$ diagonal matrix with $x$ supplying the diagonal entries.

The model in (\ref{eqn dgp}) can be rewritten as
\begin{equation}
\label{eqn dgp2}
y_{i,t}= z_{i,t}'\alpha+\eta_{i}+\varepsilon_{i,t},\quad  i=1,...,N,\ t=1,...,T
\end{equation}
where $z_{i,t}:=(y_{i,t-1},\ldots, y_{i,t-L}, x_{i,t}')'$ and $\alpha:=(\alpha_1, \ldots, \alpha_{L}, \beta')'$ are $p\times 1$ vectors ($p=p_{x}+L$). Note that the dimensions of $L$, $p_{x}$ and $p$ can vary with dimensions $N$ and $T$ but in general we suppress this dependence where no confusion arises. We assume that initial observations $y_{i,0}, y_{i,-1}, \ldots, y_{i,1-L}$ are available for $i =1,...,N$.

The three sources of high-dimensionality in (\ref{eqn dgp2}) are $p_x$, $L$ and $N$ as all of these can be increasing sequences. Sometimes one thinks of the number of lags, $L$, as being fixed and in that case only two sources remain.  
Next, (\ref{eqn dgp2}) may be written more compactly as
\[y_i=Z_{i}'\alpha+\eta_{i}\iota+\varepsilon_{i},\]
where $Z_i:=(z_{i,1}, \ldots, z_{i,T})$ is a $p\times T$ matrix, $y_{i}:=(y_{i,1},\ldots,y_{i,T})'$, $\varepsilon_{i}:=(\varepsilon_{i,1},\ldots,\varepsilon_{i,T})'$, and $\iota$ is a $T\times 1$ vector of ones. Then, one can write
\[y=(Z\quad D)\left( \begin{array}{c}
\alpha \\
\eta
\end{array}\right) +\varepsilon 
=\Pi\gamma+\varepsilon, \]
where $Z:=(Z_1,\ldots,Z_N)'$, $y:=(y_1', \ldots, y_N')'$ and $\varepsilon:=(\varepsilon_{1}',\ldots,\varepsilon_{N}')'$. $\eta:=(\eta_1,\ldots, \eta_N)'$ contains the fixed effects, $D:=I_N\otimes \iota$, and $\Pi:=(Z, D)$. Finally, $\gamma:=(\alpha',\eta')'$ contains all $p+N$ parameters of the model. Thus the dynamic panel model (\ref{eqn dgp}) can be written more compactly as something resembling a linear regression model. There are several differences, however. First, blocks of rows in the data matrix $\Pi$ may be heavily dependent. Second, we shall  see that $\alpha$ and $\eta$ have markedly different properties as a result of the fact that the probabilistic properties of the blocks of a properly scaled version of the Gram matrix pertaining to $\Pi$ are very different. Third, imposing weak sparsity only on $\eta$ implies that the oracle inequalities which we use as a stepping stone towards inference do not follow directly from the technique in, e.g., \cite{bickelritovtsybakov2009}. In fact, we do not get explicit expressions for the upper bounds but instead characterize them as solutions to certain quadratic equations in two variables. 

\subsection{Weak Sparsity and the Panel Lasso}\label{weaksparsity}

Let $J_1=\cbr[0]{j: \alpha_j\neq 0, j=1,\ldots,p}$ denote the active set of lagged left hand side variables and $x_{i,t}$ with $1\leq s_1=|J_1|\leq p$. $\alpha$ is said to be (exactly or $\ell_0$) sparse when $s_1$ is small compared to $p$. Exact sparsity is by now a standard assumption in high-dimensional econometrics and statistics. The unobserved heterogeneity, $\eta$, is usually modeled as either random or fixed effects. The former rules out correlation between $\eta$ and the remaining covariates. This is often too restrictive. In the fixed effects approach no restrictions are imposed on correlation between $\eta$ and the covariates. As explained in the introduction, our fixed effects approach is in fact a middle ground between pure random and fixed effects approaches. We choose to call it a fixed effects approach as $\eta$ is treated as a parameter to be estimated. However, $\eta$ is not entirely unrestricted and assumed to be \text{weakly} sparse\footnote{The term weakly sparse is borrowed from \cite{negahban2012unified}.} in the sense
\[\sum_{i=1}^{N}|\eta_i|^{\nu}\leq E_N\]
for some $0<\nu<1$ and $E_N=E>0$. Weak sparsity does not require any of the fixed effects to be zero but instead restricts the "sum", $E$, of all the fixed effects. $E$ can be large in the sense that it tends to infinity but the smaller it is, the sharper will our results be. It is appropriate to stress that the fixed effects can not be entirely unrestricted -- that is why our setting is middle ground between random and fixed effects. Thus, our framework also excludes many models of interest. We believe, however, that our results provide a useful first step towards uniform inference in high-dimensional dynamic panel data models and we certainly allow for more correlation between $\eta$ and the covariates than the random effects assumption does.

Note that he presence of many control variables in a high-dimensional model leaves less variation to be explained by the unobserved heterogeneities and these are therefore likely to be small in magnitude making the weak sparsity assumption reasonable. Thus, weak sparsity actually becomes more reasonable the larger the number of control variables is.

Weak sparsity is a strict generalization of exact sparsity in the sense that if only $s_2$ elements of $\eta_i$ are non-zero and none of these exceeds a constant $K$, then $\sum_{i=1}^N|\eta_i|^\nu\leq s_2K^\nu$. Thus, $E=s_2K^\nu$ works. Alternatively, exact sparsity of $\eta$ can be handled as the boundary case $\nu=0$ upon defining $0^0=0$ such that $E$ will equal the number of non-zero entries of $\eta$.


\subsection{The Objective Function and Assumptions}

Our starting point for inference is the minimiser $\hat{\gamma}=(\hat{\alpha}', \hat{\eta}')'$ of the following panel Lasso objective function
\begin{align}
L(\gamma)= \left\|y-\Pi\gamma\right\|^{2}+2\lambda_{N}\|\alpha\|_1+2\frac{\lambda_N}{\sqrt{N}}\|\eta\|_1.\label{objfunc}
\end{align}
As usual $\lambda_{N}$ is a positive regularization sequence. Note that we penalize $\alpha$ and $\eta$ differently to reflect the fact that we have $NT$ observations to estimate $\alpha_j$ for $j=1,..., p$ while only $T$ observations are available to estimate each $\eta_i$. Penalizing the fixed effects is not new and was already done in \cite{koenker2004quantile} and \cite{galvao2010penalized} in a low dimensional panel-quantile model. Furthermore, the penalization fits well with the weak sparsity assumption on the fixed effects and may increase efficiency of $\hat{\alpha}$ as found in \cite{galvao2010penalized}.

For practical implementation it is very convenient that we only have one penalty parameter $\lambda_N$ instead of having separate penalty parameters for $\alpha$ and $\eta$. The minimization problem can be solved easily as it simply corresponds to a weighted Lasso with known weights. However, the probabilistic analysis of the properly scaled Gram matrix is different from the one for the standard Lasso as it must be broken into several steps. We now turn to the assumptions needed for our inferential procedure.

\begin{assumption}
\label{assu panel data}
\item $\{(x_{i,1}',\ldots, x_{i,T}',\varepsilon_i')\}_{i=1}^N$ is an independent sequence  and
\begin{align*}
\mathbb{E}[\varepsilon_{i,t}| y_{i,t-1},...,y_{i,1-L}, x_{i,t},...,x_{i,1}]=0 \qquad\text{ for }\qquad i=1,..., N,\ t=1,...,T.
\end{align*}
\end{assumption}
Assumption \ref{assu panel data} imposes independence across $i=1,...,N$ which is standard in the panel data literature, see e.g. \cite{wooldridge2010econometric} or \cite{arellano2003panel}. Note however, that we do not assume the data to be identically distributed across $i=1,...,N$. Assumption \ref{assu panel data} also implies, by iterated expectations, that the error terms form a martingale difference sequence with respect to the filtration generated by the variables in the above conditioning set and thus restricts the degree of dependence in the error terms across $t$ (in particular, they are uncorrelated).\footnote{It can also be verified that $\cbr[0]{\varepsilon_{i,t}}_{t=1}^T$ forms a martingale difference sequence with respect to the natural filtration for all $i=1,\ldots,N$. This is because the $\varepsilon_{i,t}$ are (linear) functions of the variables in the conditioning set in Assumption \ref{assu panel data}.} However, it still allows for considerable dependence over time, as higher moments than the first are not restricted.  Furthermore, the error terms need not be identically distributed over time for any individual. Note that the increasing number of lags of $y_{i,t}$ also whiten the error terms. We also note that Assumption \ref{assu panel data} does not rule out that the error terms are conditionally heteroskedastic. In particular, they may be $\textit{autoregressively}$ conditionally heteroskedastic (ARCH). In panel data terminology, both lags of $y_{i,t}$ and $x_{i,t}$ are called \textit{predetermined} or \textit{weakly exogenous}. Finally, one can of course also include lags of the $x_{i,t}$ as these are also weakly exogenous.

In order to introduce the next assumption define the scaled empirical Gram matrix
\begin{align*}
\Psi_{N}=S^{-1}\Pi'\Pi S^{-1}=\left( \begin{array}{cc}
\frac{1}{NT}Z'Z & \frac{1}{T\sqrt{N}}Z'D\\
\frac{1}{T\sqrt{N}}D'Z & \text{I}_N
\end{array}\right)\  \text{ where }\ S=\left( \begin{array}{cc}
\sqrt{NT}\text{I}_p & 0\\
0 & \sqrt{T}\text{I}_N
\end{array}\right)  
\end{align*}
When $p+N>NT$, $\Psi_{N}$ is singular. However, to conduct inference it suffices that a compatibility type condition tailored to the panel data structure is satisfied. To be precise, define for integers $r_1\in \{1,\ldots,p\}$ and $r_2\in \{1,\ldots, N\}$
\begin{equation*}
\kappa^2(A, r_1, r_2):=\min_{\substack{R_1 \subseteq \{1,\ldots, p\}, |R_1|\leq r_1\\R_2 \subseteq \{1,\ldots,N\}, |R_2|\leq r_2\\R:=R_1\cup R_2}} \min_{\substack{\delta\in \mathbb{R}^{p+N}\setminus \{0\}\\\|\delta_{R^c}\|_1\leq 4\|\delta_{R}\|_1}}\frac{\delta'A\delta}{\frac{1}{r_1+r_2}\|\delta_R\|_1^2}
\end{equation*}
\footnote{Here $R_1\cup R_2$ is understood as $R_1\cup (R_2+p)$ where the addition is elementwise.}which is reminiscent of the restricted eigenvalue condition in \cite{bickelritovtsybakov2009}. We will need $\kappa^2(\Psi_N, r_1, r_2)$ to be bounded away from zero for $r_1=s_1$ and $r_2$ being a sequence made precise in the Appendix A depending on the degree of weak sparsity of the fixed effects. To bound  $\kappa^2(\Psi_N, s_1, r_2)$ away from zero consider $\kappa^2(\Psi, s_1, r_2)$ where\footnote{$\Psi$ actually also depends on $N$ and $T$ but for brevity we are silent about this.}
\[\Psi=\left( \begin{array}{cc}
\Psi_Z & 0\\
0 & \text{I}_N
\end{array}\right) :=\left( \begin{array}{cc}
\frac{1}{NT}\sum_{i=1}^{N}\sum_{t=1}^{T}\mathbb{E}[z_{i,t}z_{i,t}'] & 0\\
0 & \text{I}_N
\end{array}\right). \]
 We will see that in order for $\kappa^2(\Psi_N, s_1, r_2)$ to be bounded away from zero it suffices that $\kappa^2(\Psi, s_1, r_2)$ is bounded away from zero and $\Psi_N$ being close to $\Psi$ in an appropriate sense. Writing $\delta=(\delta_1',\delta_2')'$, where $\delta_1\in \mathbb{R}^p$ and $\delta_2\in \mathbb{R}^N$, note that by the block diagonal structure of $\Psi$
\begin{align*}
\frac{\delta'\Psi\delta}{\frac{1}{r_1+r_2}\|\delta_R\|_1^2}
\geq
\frac{\delta_1'\Psi_Z\delta_1+\delta_{2,R_2}'\delta_{2,R_2}}{\delta_{1,R_1}'\delta_{1,R_1}+\delta_{2,R_2}'\delta_{2,R_2}}
\geq 
1\wedge \frac{\delta_{1}'\Psi_Z\delta_{1}}{\delta_{1,R_1}'\delta_{1,R_1}}.
\end{align*}
The above estimates are useful as they show that we really only have to consider minimization over the upper left submatrix $\Psi_Z$ in the definition of $\kappa^2(\Psi_N, s_1, r_2)$. To be precise,
\begin{align}
\min_{r_2\in\cbr[0]{1,...,N}}\kappa^2(\Psi, s_1, r_2)\geq 1\wedge \min_{\substack{R_1 \subseteq \{1,\ldots, p\}, |R_1|\leq s_1}}\frac{\delta_{1}'\Psi_Z\delta_{1	}}{\delta_{1,R_1}'\delta_{1,R_1}}=:\kappa_2^2(\Psi_Z,s_1)\label{re2}.
\end{align}
Thus, $\kappa_2^2(\Psi_Z,s_1)$ is a uniform lower bound for $\kappa^2(\Psi, s_1, r_2)$ and in order for $\kappa^2(\Psi, s_1, r_2)$ to be bounded away from zero it suffices to assume that

\begin{assumption}
\label{assu smallest eigenvalue of PsiZ}
$\kappa_2^2=\kappa^2\del[0]{\Psi_Z, s_1}$ is uniformly bounded away from zero.
\end{assumption}

Assumption \ref{assu smallest eigenvalue of PsiZ} is rather innocent as it is trivially satisfied when the $\Psi_Z$ is positive definite. Since $\Psi_Z$ is the \textit{population} second moment matrix of $z_{i,t}$ this is a rather innocent assumption which is typically imposed. Compatibility type conditions are standard in the literature and various versions and their interrelationship have been investigated in \cite{van2009conditions}.

\begin{assumption}
\label{assu subgaussian}
There exist positive constants $C$ and $K$ such that
\begin{enumerate}[(a)]
\item $\varepsilon_{i,t}$ are uniformly subgaussian; that is, $\mathbb{P}(|\varepsilon_{i,t}|\geq \epsilon)\leq \frac{1}{2}Ke^{-C\epsilon^2}$ for every $\epsilon\geq 0$, $i=1,\ldots, N$ and $t=1,\ldots,T$. 
\item $z_{i,t,l}$ are uniformly subgaussian; that is, $\mathbb{P}(|z_{i,t,l}|\geq \epsilon)\leq \frac{1}{2}Ke^{-C\epsilon^2}$ for every $\epsilon\geq 0$, $i=1,\ldots, N$, $t=1,\ldots,T$ and $l=1,\ldots,p$.
\end{enumerate}
\end{assumption} 
In the context of the plain static regression model it is common practice to assume the error terms as well as the covariates to be subgaussian. However, this assumption is not as innocent in the context of the dynamic panel data model (\ref{eqn dgp}) as $y_{i,t}$ is generated by the model and its properties are thus completely determined by those of $x_{i,t}, \varepsilon_{i,t}$ as well as the parameters of the model. Lemma \ref{lemma y inherit subgaussian} in Appendix A shows that $y_{i,t}$ is subgaussian if $x_{i,t}$ and $\varepsilon_{i,t}$ satisfy this property and the parameters are well-behaved. In particular, a wide class of (causal) stationary processes are included. Note also, that Assumption \ref{assu subgaussian} imposes subgaussianity of the initial values $y_{i,0},...,y_{i,1-L}$ for all $i=1,...,N$. \cite{caner2014asymptotically} have derived results similar to ours in a cross sectional setting without the sub-gaussianity assumption. However, the dimension of their model can not increase as fast as here.

\subsection{The Oracle Inequalities}
%
%

With the above assumptions in place we are ready to state our first result. Defining $\mathcal{F}(s_1,\nu,E):=\cbr[1]{ \alpha \in \mathbb{R}^{p}: \enVert[0]{\alpha}_{0} \leq s_1} \times \cbr[1]{\eta \in \mathbb{R}^{N}: \sum_{i=1}^{N}|\eta_i|^{\nu}\leq E}$, one has
\begin{thm}[Oracle inequalities]
\label{thm probabilistic oracle inequality}
Let Assumptions \ref{assu panel data} - \ref{assu subgaussian} hold. Then, choosing $\lambda_{N}=\sqrt{4MNT(\log (p\vee N))^3}$ for some $M>0$, the following inequalities are valid with probability at least
\begin{align*}
1-Ap^{1-BM^{1/3}}-AN^{1-BM^{1/3}}-A(p^2+pN)\exp\del[4]{{-B\cbr[3]{ N/\sbr[2]{s_1+E\del[2]{\frac{\lambda_N}{\sqrt{N}T}}^{-\nu}}^2} ^{1/3}}}
\end{align*}
for positive constants $A$ and $B$ and $s_1+E\del[2]{\frac{\lambda_N}{\sqrt{N}T}}^{-\nu}\lesssim \sqrt{N}$,
\begin{align*}
&\frac{1}{NT}\left\|\Pi(\hat{\gamma}-\gamma) \right\|^2 
\leq \frac{120\lambda^2_{N}s_1}{\kappa_2^2(NT)^2}+\del[2]{\frac{120}{\kappa_2^2}+20} \frac{\lambda_N}{\sqrt{N}NT}E\del[3]{\frac{\lambda_N}{\sqrt{N}T}} ^{1-\nu}
\\
&\|\hat{\alpha}-\alpha\|_1
\leq\frac{120\lambda_{N}s_1}{\kappa_2^2NT}+\del[2]{\frac{120}{\kappa_2^2}+20}\frac{1}{\sqrt{N}}E\del[3]{\frac{\lambda_N}{\sqrt{N}T}}^{1-\nu}\\
&\|\hat{\eta}-\eta\|_1 
\leq\frac{120\lambda_{N}s_1}{\kappa_2^2\sqrt{N}T}+\del[2]{\frac{120}{\kappa_2^2}+20}E\del[3]{\frac{\lambda_N}{\sqrt{N}T}}^{1-\nu}.
\end{align*}
Moreover, the above bounds are valid uniformly over
$\mathcal{F}(s_1,\nu,E)$. 
\end{thm}

Theorem \ref{thm probabilistic oracle inequality} provides oracle inequalities for the prediction error as well as the estimation error of the parameter vectors. While these bounds are of independent interest, we primarily use them as means towards our ultimate end of conducting (joint) inference on $\alpha$ and $\eta$. We stress that the bounds in Theorem \ref{thm probabilistic oracle inequality} are finite sample bounds; they hold for any fixed values of $N$ and $T$. The novel feature of our oracle inequalities is that $E$, the "size" of $\eta$, is allowed to grow even when we want the upper bound of $\|\hat{\eta}-\eta\|_1$ go to zero. The special case of exact sparsity of $\eta$ corresponds to $\nu=0$ and $E$ being the sparsity index, say $s_2$, of $\eta$. 


We also note that the oracle inequalities are not obtained in an entirely standard manner as the mixture of exact and weak sparsity in dynamic panel data models calls for a different proof technique which yields the upper bounds as solutions to certain quadratic equations. Furthermore, we remark that in analogy to oracle inequalities in the plain linear regression model the number of covariates in $x_{i,t}$ ($p_x$) may increase at an exponential rate in $NT$ without hindering the right hand sides of the oracle inequalities in being small. Finally, we do not assume independence across $t=1,...,T$ for any individual thus altering the standard probabilistic analysis as well. Instead we use concentration inequalities for  martingales to obtain bounds almost as sharp as in the completely independent case. If one restricts the dependence structure of $\cbr[0]{x_{i,t}}_{t=1}^T$ for every $i=1,...,N$ to be, e.g., strongly mixing then one can use concentration inequalities for mixing processes such as in \cite{merlevede2011bernstein}. Restricting the dependence structure this way will allow $s_1$ and $E$ to increase faster. The focus on the $\ell_1$-norm in the oracle inequalities for $\alpha$ and $\eta$ is due to the fact that an upper bound in this norm will be particularly useful when developing our uniformly valid inference procedure in the following sections.   

\section{Inference}
\label{sec inference}

In this section we show how to conduct inference on $\gamma$ and first discuss how desparsification as proposed in \cite{vandegeerbuhlmannritovdezeure2014} works in our context.

\subsection{The Desparsified Lasso Estimator $\tilde{\gamma}$}
\label{sec desparsified lasso}

First, observe that $L(\gamma)$ in (\ref{objfunc}) is convex in $\gamma$ and in order for $\hat{\gamma}$ to be a minimiser of $L$, 0 must belong to the subdifferential of $L(\gamma)$ at $\hat{\gamma}$, i.e.
\[0\in\partial L(\hat{\gamma})=\left( \begin{array}{c}
-2Z'(y-\Pi\hat{\gamma})+2\lambda_{N}\hat{\kappa}_1\\
-2D'(y-\Pi\hat{\gamma})+2\frac{\lambda_N}{\sqrt{N}}\hat{\kappa}_2
\end{array}\right) \]
where $\hat{\kappa}_1$ and $\hat{\kappa}_2$ are $p\times 1$ and $N\times 1$ vectors, respectively, such that $\hat{\kappa}_{1j}\in [-1,1]$ with $\hat{\kappa}_{1j}=\text{sgn}(\hat{\alpha}_j)$ if $\hat{\alpha}_j\neq 0$ for $j=1,\ldots,p$. Similarly, $\hat{\kappa}_{2i}\in [-1,1]$ with $\hat{\kappa}_{2i}=\text{sgn}(\hat{\eta}_i)$ if $\hat{\eta}_i\neq 0$ for $i=1,\ldots,N$. Hence,
\begin{equation}
\label{eqn KKT condition}
-\Pi'\left( y-\Pi\hat{\gamma}\right)+\left( \begin{array}{c}
\lambda_{N}\hat{\kappa}_1\\
\frac{\lambda_N}{\sqrt{N}}\hat{\kappa}_2
\end{array}\right)=0.  
\end{equation}
Using that $y=\Pi\gamma+\varepsilon$ and multiplying by $S^{-1}$ from the left yields
\[\Psi_{N} S\left( \hat{\gamma}-\gamma\right) +S^{-1}\left( \begin{array}{c}
\lambda_{N}\hat{\kappa}_1\\
\frac{\lambda_N}{\sqrt{N}}\hat{\kappa}_2
\end{array}\right)=S^{-1}\Pi'\varepsilon.  \]
In order to derive the limiting distribution of $S(\hat{\gamma}-\gamma)$ one would usually proceed by isolating $S(\hat{\gamma}-\gamma)$ which implies inverting $\Psi_N$. However, when $p+N>NT$, $\Psi_N$ is not invertible. The idea of \cite{vandegeerbuhlmannritovdezeure2014} and \cite{javanmard2013confidence}  is to circumvent this problem by using an approximate inverse of $\Psi_{N}$ and controlling the asymptotic approximation error. Suppose that a matrix $\hat{\Theta}$ is a reasonable approximation to the inverse of $\Psi_N$. We shall explicitly construct $\hat{\Theta}$ in the next section. Then we may write
\begin{align*}
\hat{\gamma}=\gamma-S^{-1}\hat{\Theta}S^{-1}\left( \begin{array}{c}
\lambda_{N}\hat{\kappa}_1\\
\frac{\lambda_N}{\sqrt{N}}\hat{\kappa}_2
\end{array}\right)+S^{-1}\hat{\Theta}S^{-1}\Pi'\varepsilon-S^{-1}\Delta
\end{align*}
where $\Delta:=\del[1]{ \hat{\Theta}\Psi_{N}-\text{I}}S\del[0]{ \hat{\gamma}-\gamma}$ is the error resulting from using an approximate inverse $\hat{\Theta}$ of $\Psi_N$ as opposed to an exact inverse. The term $S^{-1}\hat{\Theta}S^{-1}\left( \begin{array}{c}
\lambda_{N}\hat{\kappa}_1\\
\frac{\lambda_N}{\sqrt{N}}\hat{\kappa}_2
\end{array}\right)$ in the above display is the bias incurred by $\hat{\gamma}$ due to shrinkage of the parameters in (\ref{objfunc}). As this bias term is known one may add it back to $\hat{\gamma}$ in order to define the debiased estimator 
\begin{align*}
\tilde{\gamma}=\hat{\gamma}+S^{-1}\hat{\Theta}S^{-1}\left( \begin{array}{c}
\lambda_{N}\hat{\kappa}_1\\
\frac{\lambda_N}{\sqrt{N}}\hat{\kappa}_2
\end{array}\right)
=\gamma+S^{-1}\hat{\Theta}S^{-1}\Pi'\varepsilon-S^{-1}\Delta
\end{align*}
The new estimator $\tilde{\gamma}$ is no longer sparse as it has added a bias correction terms to the sparse Lasso estimator $\hat{\gamma}$. Therefore, we will also refer to it as the \textit{desparsified} Lasso estimator in the dynamic panel context.
\begin{equation}
\label{eqn multivariate inference equnation}
S\left( \tilde{\gamma}-\gamma\right)=\hat{\Theta}S^{-1}\Pi'\varepsilon-\Delta,  
\end{equation}
For any $(p+N)\times 1$ vector $\rho$ with $\|\rho\|=1$ we shall study the asymptotic behaviour of 
\begin{equation}
\label{eqn univariate inference equnation}
\rho'S\left( \tilde{\gamma}-\gamma\right)
=
\rho'\hat{\Theta}S^{-1}\Pi'\varepsilon-\rho'\Delta. 
\end{equation}
A central limit theorem for $\rho'\hat{\Theta}S^{-1}\Pi'\varepsilon$ as well as asymptotic negligibility of $\rho'\Delta$ will yield asymptotically gaussian inference. Furthermore, we shall provide a uniformly consistent estimator of the asymptotic variance of $\rho'\hat{\Theta}S^{-1}\Pi'\varepsilon$ even in the presence of conditional heteroskedasicity. A leading special case of (\ref{eqn univariate inference equnation}) is when one is only interested in the asymptotic distribution of $\tilde{\gamma}_j$ corresponding to $\rho=e_j$ being the $j$th basis vector of $\mathbb{R}^{p+N}$. In general, we will be interested in the asymptotic distribution of a subset $H\subseteq\cbr[0]{1,...,p+N}$ of the indices of $\gamma$ with cardinality $h$ and shall show that asymptotically honest (uniformly valid) gaussian inference is possible in the presence of heteroskedasticity even for $h\to\infty$ and $H$ simultaneously involving elements of $\alpha$ and $\eta$.

\subsection{Construction of $\hat{\Theta}$}

\label{sec nodewise regression}
As is clear from the discussion above we need a good choice for $\hat{\Theta}$. In particular we shall show that
\[\hat{\Theta}=\left( \begin{array}{cc}
\hat{\Theta}_Z & 0\\
0 & \text{I}_N
\end{array}\right) \]
works well. Here $\hat{\Theta}_Z$ will be constructed using nodewise regressions as in \cite{vandegeerbuhlmannritovdezeure2014} and we show that this is possible even when the rows of $Z$ are not independent and identically distributed. The construction of $\hat{\Theta}_Z$ parallels the one in \cite{vandegeerbuhlmannritovdezeure2014} to a high extent but importantly for our context we do not need the rows of $\Psi_Z^{-1}$ to be sparse for the nodewise regressions to work well. We will discuss the importance of this, once we have properly constructed $\hat{\Theta}_Z$. First, define
\begin{equation}
\label{eqn nodewise regression lasso objective function}
\hat{\phi}_j= \argmin_{\delta \in \mathbb{R}^{p-1}}\left\lbrace \frac{1}{NT}\|z_j-Z_{-j}\delta\|^2+2\lambda_{node}\|\delta\|_1\right\rbrace, \qquad j=1,...,p,
\end{equation}
where $z_j$ is the $j$th column of $Z$, $Z_{-j}$ is the $NT\times (p-1)$ submatrix of $Z$ with $Z$'s $j$th column removed, and the $(p-1)\times 1$ vector $\hat{\phi}_j=\{\hat{\phi}_{j,l}: l=1,\ldots,p, l\neq j\}$. Thus, $\hat{\phi}_j$ is the Lasso estimator resulting from regressing $z_j$ on $Z_{-j}$. Next, define 
\[\hat{C}=\left( \begin{array}{cccc}
1 & -\hat{\phi}_{1,2} &\cdots  &-\hat{\phi}_{1,p}\\
-\hat{\phi}_{2,1} & 1  &\cdots  &-\hat{\phi}_{2,p}\\
\vdots&\vdots&\ddots&\vdots\\
-\hat{\phi}_{p,1} & -\hat{\phi}_{p,2} &\cdots  &1  \\
\end{array}\right) \]
and $\hat{\tau}_j^2= \frac{1}{NT}\|z_j-Z_{-j}\hat{\phi}_j\|^2+\lambda_{node}\|\hat{\phi}_j\|_1$ as well as $\hat{T}^2=\text{diag}(\hat{\tau}_1^2,\ldots,\hat{\tau}_p^2)$. Finally, we set $\hat{\Theta}_Z=\hat{T}^{-2}\hat{C}$. Let $\hat{C}_j$ denote the $j$th row of $\hat{C}$ and let $\hat{\Theta}_{Z,j}$ denote the $j$th row of $\hat{\Theta}_{Z}$ but both written as a $p\times 1$ vectors. Then, $\hat{\Theta}_{Z,j}=\hat{C}_j/\hat{\tau}_j^2$. For any $j=1,...,p$, the KKT condition for a minimum in (\ref{eqn nodewise regression lasso objective function}) are
\begin{align}
-\frac{1}{NT}Z_{-j}'(z_j-Z_{-j}\hat{\phi}_j)+\lambda_{node}w_j=0, \label{nodeKKT}
\end{align}
where $w_j$ is the subdifferential of $\|x\|_1$ evaluated at $\hat{\phi}_{j}$. Using this, the definition of $\hat{\tau}_j$, and $\hat{\phi}_j'w_j=\enVert[0]{\hat{\phi}_j}_1$ yields
\begin{align}
\hat{\tau}_j^2
=
\frac{1}{NT}(z_j-Z_{-j}\hat{\phi}_j)'(z_j-Z_{-j}\hat{\phi}_j)+\lambda_{node}\enVert[0]{\hat{\phi}_j}_1
=\frac{1}{NT}(z_j-Z_{-j}\hat{\phi}_j)'z_j.\label{align tauj2 obtained via KKT}
\end{align}
Thus, by the definition of $\hat{\Theta}_{Z,j}$, and as $\hat{\tau}_j^2$ is bounded away from zero (we shall later argue rigorously for this)
\begin{equation}
\label{eqn diagonal of empirical Gram times approximate inverse}
\frac{1}{NT}z_j'Z\hat{\Theta}_{Z,j}
=
1.
\end{equation}
Furthermore, the KKT conditions (\ref{nodeKKT}) can also be written as
\begin{equation}
\label{eqn nodewise regression KKT conditions}
\frac{1}{NT}Z_{-j}'(z_j-Z_{-j}\hat{\phi}_j)=\lambda_{node}w_j,
\end{equation}
which implies $\frac{1}{NT}Z_{-j}'Z\hat{\Theta}_{Z,j}=\lambda_{node}w_j/\hat{\tau}_j^2$. Combining with (\ref{eqn diagonal of empirical Gram times approximate inverse}) yields
\begin{equation}
\label{eqn extended KKT conditions}
\left\| \frac{1}{NT}Z'Z\hat{\Theta}_{Z,j}-e_j\right\|_{\infty}\leq \frac{\lambda_{node}}{\hat{\tau}_j^2},
\end{equation}
which together with an oracle inequality for $\enVert[0]{\hat{\gamma}-\gamma}_1$ provides an upper bound on the $j$th entry of $\Delta$ in (\ref{eqn univariate inference equnation}). In other words, (\ref{eqn extended KKT conditions}) will be used to show the required asymptotic negligibility of $\rho'\Delta$ in (\ref{eqn univariate inference equnation}) by arguments made rigorous in the appendix.

\subsection{Asymptotic Properties of the Approximate Inverse}

In order to show that $\rho'\hat{\Theta}S^{-1}\Pi'\varepsilon$ is asymptotically gaussian one needs to understand the limiting behaviour of $\hat{\Theta}$ constructed above. We show that $\hat{\Theta}$ is close to 
\[\Theta=\left( \begin{array}{cc}
\Theta_Z & 0\\
0 & \text{I}_N
\end{array}\right):=\left( \begin{array}{cc}
\Psi_Z^{-1} & 0\\
0 & \text{I}_N
\end{array}\right)\]
in an appropriate sense. To this end, note that by \cite{yuan2010}
\begin{equation}
\label{align inverse of the partitioned matrix}
\Theta_{Z,j,j}=\left[ \Psi_{Z,j,j}-\Psi_{Z,j,-j}\Psi_{Z,-j,-j}^{-1}\Psi_{Z,-j,j}\right]^{-1}\  \text{ and }\ 
\Theta_{Z,j,-j}=-\Theta_{Z,j,j}\Psi_{Z,j,-j}\Psi_{Z,-j,-j}^{-1},
\end{equation}
where $\Theta_{Z,j,j}$ is the $j$th diagonal entry of $\Theta_{Z}$, $\Theta_{Z,j,-j}$ is the $1\times (p-1)$ vector obtained by removing the $j$th entry of the $j$th row of $\Theta_Z$, $\Psi_{Z, -j,-j}$ is the submatrix of $\Psi_Z$ with the $j$th row and column removed, $\Psi_{Z,j,-j}$ is the $j$th row of $\Psi_Z$ with its $j$th entry removed, $\Psi_{Z,-j,j}$ is the $j$th column of $\Psi_Z$ with its $j$th entry removed. Next, let $z_{i,t,j}$ be the $j$th element of $z_{i,t}$ and $z_{i,t,-j}$ be all elements except the $j$th. Define the $ (p-1)\times 1$ vector
\[\phi_j:=\argmin_{\delta}\frac{1}{NT}\sum_{i=1}^{N}\sum_{t=1}^{T}\mathbb{E}[z_{i,t,j}-z_{i,t,-j}'\delta]^2\]
such that 
\begin{equation}
\label{eqn definition of phij}
\phi_j=\left( \frac{1}{NT}\sum_{i=1}^{N}\sum_{t=1}^{T}\mathbb{E}[z_{i,t,-j}z_{i,t,-j}']\right)^{-1}\left( \frac{1}{NT}\sum_{i=1}^{N}\sum_{t=1}^{T}\mathbb{E}[z_{i,t,-j}z_{i,t,j}]\right)=\Psi_{Z,-j,-j}^{-1}\Psi_{Z,-j,j}.
\end{equation}
Therefore, $\Theta_{Z,j,-j}=-\Theta_{Z,j,j}\phi_j'$ showing that $\Theta_{Z,j,-j}$ and $\phi_j'$ only differ by a multiplicative constant. In particular, $j$th row of $\Theta_Z$ is exactly sparse if and only if $\phi_j$ is exactly sparse. More generally, we shall exploit below that weak sparsity of one implies weak sparsity of the other. Furthermore, defining $\zeta_{j,i,t}:=z_{i,t,j}-z_{i,t,-j}'\phi_j$ we may write
\[z_{i,t,j}=z_{i,t,-j}'\phi_j+\zeta_{j,i,t}, \qquad \text{for }i=1,...,N,\quad  t=1,...,T.\]
where by the definition of $\phi_j$
\begin{equation}
\label{eqn regression error by construction}
\frac{1}{NT}\sum_{i=1}^{N}\sum_{t=1}^{T}\mathbb{E}[z_{i,t,-j}\zeta_{j,i,t}]=0.
\end{equation}
Thus, in light of Theorem \ref{thm probabilistic oracle inequality}, it is sensible that the Lasso estimator $\hat{\phi}_j$ defined in (\ref{eqn nodewise regression lasso objective function}) is close to the population regression coefficients $\phi_j$ (we shall make this more formal in Appendix A). Next, defining
\begin{align*}
\tau^2_j
:=
\mathbb{E}\sbr[2]{\frac{1}{NT}\sum_{i=1}^{N}\sum_{t=1}^{T}(z_{i,t,j}-z_{i,t,-j}'\phi_j)^2}=\Psi_{Z,j,j}-\Psi_{Z,j,-j}\Psi_{Z,-j,-j}^{-1}\Psi_{Z,-j,j}
=
\frac{1}{\Theta_{Z,j,j}}
\end{align*}
observe $\Theta_{Z,j,-j}=-\phi_j'/\tau_j^2$. Thus, we can write $\Theta_Z=T^{-1}C$ where $T=\text{diag}(\tau_1^2,...,\tau_p^2)$ and $C$ is defined similarly to $\hat{C}$ but with $\phi_j$ replacing $\hat{\phi}_j$ for $j=1,...,p$. Finally, let $\Theta_{Z,j}$ denote the $j$th row of $\Theta_Z$ written as a column vector. In Lemma \ref{lemma step stone to final theorem} below we will see that $\hat{\phi}_j$ and $\hat{\tau}_j^2$ are close to $\phi_j$ and $\tau_j^2$, respectively such that $\hat{\Theta}_{Z,j}$ is close to $\Theta_{Z,j}$ which is the desired control of $\hat{\Theta}_{Z,j}$. Write $\rho=(\rho_1',\rho_2')'$ with $\|\rho\|=1$, where $\rho_1\in \mathbb{R}^p$ and $\rho_2\in \mathbb{R}^N$. Hence define
\[H=H_1\cup \del[1]{H_2+p}:= \{j: \rho_{1j} \neq 0\}\cup \del[1]{\{i: \rho_{2i} \neq 0\}+p},\]
with $|H_1|=h_{1,N}=h_1$, $|H_2|=h_{2,N}=h_2$ and $|H|=h=h_1+h_2$. In dynamic panel data models it may not be reasonable to assume that the rows of the inverse second moment matrix $\Psi_Z^{-1}=\Theta_Z$, i.e. $\Theta_{Z,j}$ are sparse. Paralleling Section \ref{weaksparsity} we shall instead assume that the $\Theta_{Z,j}$ are weakly sparse and assume that
\begin{equation}
\sum_{k=1}^{p-1}|\phi_{j,k}|^{\vartheta}=\tau_j^{2\vartheta}\sum_{l\neq j}^{p}|\Theta_{Z,j,l}|^{\vartheta}\leq G_j \label{wsinv}
\end{equation}
for some $0<\vartheta<1$ and $G_j>0$. Define $\bar{G}:=\max_{j\in H_1}G_j$. 

\begin{assumption}
\label{assu more on eigen values}
\item \begin{enumerate}[(a)]
\item $\text{mineval}(\Psi_Z)$ is uniformly bounded away from zero and $\text{maxeval}(\Psi_Z)$ is uniformly bounded from above.
\item $\bar{G}\lambda_{node}^{1-\vartheta}=O(1)$. 
\item There exist positive constants $C$ and $K$ such that $\zeta_{j,i,t}$ are uniformly subgaussian; that is, $\mathbb{P}(|\zeta_{j,i,t}|\geq \epsilon)\leq \frac{1}{2}Ke^{-C\epsilon^2}$ for every $\epsilon>0$, $i=1,\ldots,N$, $t=1,\ldots,T$ and $j=1,\ldots,p$.
\end{enumerate}
\end{assumption}

Assumption \ref{assu more on eigen values}(a) is standard and strengthens Assumption \ref{assu smallest eigenvalue of PsiZ} slightly. Recall that the population matrix $\Psi_Z$ can can have full rank even when the empirical counterpart $\Psi_N$ has rank zero -- which it has when $p+N>NT$. Note that Assumption \ref{assu more on eigen values}(a) implies that $\tau_j^2$ is uniformly bounded away from zero as $\tau_j^2 = 1/\Theta_{Z,j,j}\geq 1/\text{maxeval}(\Theta_Z)=\text{mineval}(\Psi_Z)$. Similarly, $\tau_j^2\leq \text{maxeval}(\Psi_Z)$ implying that $\tau_j^2$ is bounded in (\ref{wsinv}). Therefore, weak sparsity of $\phi_{j,k}$ translates into weak sparsity of the rows of $\Theta$. Notice that we generalize the cross sectional results of \cite{vandegeerbuhlmannritovdezeure2014} by not imposing the inverse covariance (second moment matrix) of $z_{i,t}$ to have sparse rows. When $z_{i,t}$ is gaussian
exact sparsity of $\Psi_Z^{-1}$ is related to the notion of conditional independence: the $(j,k)$th entry of $\Psi_Z^{-1}$ being zero is equivalent to $z_{i,t,j}$ being independent of $z_{i,t,k}$ conditional on the remaining variables in $z_{i,t}$. This is hard to justify in dynamic panel data models. First, it does not sound reasonable for $x_{i,t}$'s to be mostly conditionally independent given the lagged variables. Second, adjacent lagged variables $y_{i,t-l}$ and $y_{i,t-l-1}$ $(l=1,\ldots, L+1)$ are not independent even after conditioning on all the other variables in $z_{i,t}$. In conclusion, it is important to relax the exact sparsity assumption on the rows of $\Theta_Z$ in the context of dynamic panel data models.

Part (b) restricts the rate of growth of $\bar{G}$. As we shall choose $\lambda_{node}\asymp \sqrt{\frac{\log^3(p)}{N}}$ it implies in particular that $\bar{G}=O\del[1]{(N/\log^3(p))^{\frac{1-\vartheta}{2}}}$.  Part (c) imposes subgaussianity on the error terms from the nodewise regressions.

\begin{lemma}
\label{lemma step stone to final theorem}
Let Assumptions \ref{assu panel data}, \ref{assu subgaussian} and \ref{assu more on eigen values} hold. Define $\lambda_{node}=\sqrt{16M(\log p)^3/N}$ for some $M>0$. Then, for $M$ sufficiently large,
\begin{align}
\max_{j\in H_1}|\hat{\tau}_j^2-\tau_j^2|& =O_p\del[3]{\bar{G}^{1/2}\sbr[3]{\frac{(\log p)^3}{N}}^{\frac{2-\vartheta}{4}}} \label{align lemma 2 1}\\
\max_{j\in H_1}\frac{1}{\hat{\tau}_j^2}&=O_p(1)\label{align lemma 2 2}\\
\max_{j\in H_1}\envert[3]{\frac{1}{\hat{\tau}_j^2}-\frac{1}{\tau_j^2}}&=O_p\del[3]{\bar{G}^{1/2}\sbr[3]{\frac{(\log p)^3}{N}}^{\frac{2-\vartheta}{4}}}\label{align lemma 2 3}\\
\max_{j\in H_1}\left\|  \hat{\Theta}_{Z,j}-\Theta_{Z,j}\right\|_1&=O_p\del[3]{\bar{G}\sbr[3]{\frac{(\log p)^3}{N}}^{\frac{1-\vartheta}{2}}}\label{align lemma 2 4}\\
\max_{j\in H_1}\enVert[1]{ \hat{\Theta}_{Z,j}-\Theta_{Z,j}}
&=O_p\del[3]{\bar{G}^{1/2}\sbr[3]{\frac{(\log p)^3}{N}}^{\frac{2-\vartheta}{4}}}\label{align lemma 2 5}\\
\max_{j\in H_1}\enVert[1]{ \hat{\Theta}_{Z,j}}_1&=O_p\del[3]{\bar{G}^{1/2}\sbr[3]{\frac{(\log p)^3}{N}}^{-\frac{\vartheta}{4}}}\label{align lemma 2 6}
\end{align}
\end{lemma}

Lemma \ref{lemma step stone to final theorem} is used as a stepping stone towards the establishing asymptotically gaussian inference as provides the rate at which $\hat{\Theta}_{Z}$ approaches $\Theta_Z$ uniformly over $H_1$. Note that for $H_1=\cbr[0]{1,...,p}$, (\ref{align lemma 2 4}) provides an upper bound on the induced $\ell_\infty$-distance between $\hat{\Theta}_{Z}$ and $\Theta_Z$. However, we only need to control this distance for those indices corresponding to the parameters we seek the joint limiting distribution of. On the other hand, it should be stressed that the uniformity over $H_1$ of the above results is crucial in establishing the limiting gaussian inference and providing a feasible estimator of the covariance matrix of the parameter estimates. In case one is only interested in one entry of $\gamma$, $H_1$ reduces to a singleton if this entry is in $\alpha$. If this entry is in $\eta$, Lemma \ref{lemma step stone to final theorem} is actually superfluous as the lower right hand corners of $\hat{\Theta}$ and $\Theta$ are identical. 

\subsection{The Asymptotic Distribution of $\tilde{\gamma}$}

In this section we formalise the discussion in Section \ref{sec desparsified lasso} as Theorem \ref{thm Delta is small op 1}. To this end, define
\begin{align}
\Sigma_{\Pi\varepsilon}
=
\mathbb{E}(S^{-1}\Pi'\varepsilon\varepsilon'\Pi S^{-1})
=
\left( \begin{array}{cc}
\mathbb{E}\sbr[1]{ Z'\varepsilon\varepsilon'Z/\del[0]{NT}}& \mathbb{E}\sbr[1]{ Z'\varepsilon\varepsilon'D/\del[0]{ \sqrt{N}T}} \\
\mathbb{E}\sbr[1]{ D'\varepsilon\varepsilon'Z/\del[0]{ \sqrt{N}T} } & \mathbb{E}\sbr[1]{ D'\varepsilon\varepsilon'D/T}
\end{array}\right)
=
 \left( \begin{array}{cc}
\Sigma_{1,N}& \Sigma_{2,N}\\
\Sigma_{2,N}'& \Sigma_{3,N}
\end{array}\right)\nonumber.
\end{align}
and note that  
\[\Sigma_{1,N}=\mathbb{E}\left[ \frac{1}{NT}\sum_{i=1}^{N}\sum_{j=1}^{N}Z_i\varepsilon_i\varepsilon_j'Z_j'\right]= \frac{1}{NT}\sum_{i=1}^{N}\mathbb{E}\left[ Z_i\varepsilon_i\varepsilon_i'Z_i'\right]=\frac{1}{NT}\sum_{i=1}^{N}\sum_{t=1}^{T}\mathbb{E}\left[\varepsilon_{i,t}^2z_{i,t}z_{i,t}' \right], \]
where the second and third equality both follow from Assumption \ref{assu panel data}. Likewise, $\Sigma_{3,N}= \frac{1}{T}\sum_{i=1}^{N}\mathbb{E}\left[ d_i\varepsilon_i\varepsilon_i'd_i'\right]=\frac{1}{T}\sum_{i=1}^{N}\sum_{t=1}^{T}\mathbb{E}\left[\varepsilon_{i,t}^2d_{i,t}d_{i,t}' \right]=\text{diag}(\frac{1}{T}\sum_{t=1}^{T}\mathbb{E}[\varepsilon_{1,t}^2],...,\frac{1}{T}\sum_{t=1}^{T}\mathbb{E}[\varepsilon_{N,t}^2])$, where $d_i'$ is the $i$th $T\times N$ block of $D$, and $d_{i,t}$ is a $N\times 1$ zero vector with the $i$th entry replaced by $1$. In the same manner, $\Sigma_{2,N}= \frac{1}{\sqrt{N}T}\sum_{i=1}^{N}\mathbb{E}\left[ z_i\varepsilon_i\varepsilon_i'd_i'\right]=\frac{1}{\sqrt{N}T}\sum_{i=1}^{N}\sum_{t=1}^{T}\mathbb{E}\left[ \varepsilon_{i,t}^2z_{i,t}d_{i,t}'\right]$. In words, $\Sigma_{2,N}$ is a $p\times N$ matrix with its $i$th column being $\frac{1}{\sqrt{N}T}\sum_{t=1}^{T}\mathbb{E}[z_{i,t}\varepsilon_{i,t}^2]$. Finally, motivated by the above, define the feasible sample counterpart of $\Sigma_{\Pi\varepsilon}$ as
\[\hat{\Sigma}_{\Pi\varepsilon}=\left( \begin{array}{cc}
\hat{\Sigma}_{1,N}& \hat{\Sigma}_{2,N}\\
\hat{\Sigma}_{2,N}'& \hat{\Sigma}_{3,N}
\end{array}\right):=\left( \begin{array}{cc}
\frac{1}{NT}\sum_{i=1}^{N}\sum_{t=1}^{T}\hat{\varepsilon}_{i,t}^2z_{i,t}z_{i,t}'  & \frac{1}{\sqrt{N}T}\sum_{i=1}^{N}\sum_{t=1}^{T}\hat{\varepsilon}_{i,t}^2z_{i,t}d_{i,t}'\\
\frac{1}{\sqrt{N}T}\sum_{i=1}^{N}\sum_{t=1}^{T}\hat{\varepsilon}_{i,t}^2d_{i,t}z_{i,t}'&
 \frac{1}{T}\sum_{i=1}^{N}\sum_{t=1}^{T}\hat{\varepsilon}_{i,t}^2d_{i,t}d_{i,t}'\end{array}\right), \]
where $\hat{\varepsilon}_{i,t}:=y_{i,t}-z_{i,t}'\hat{\alpha}-\hat{\eta}_i$. One could also consider constructing $\hat{\varepsilon}_{i,t}$ based on the desparsified estimates. However, this would require running the nodewise regressions for all variables and not only those pertaining to the coefficients in the hypothesis being tested resulting in a much more computationally demanding procedure. The following assumptions are needed to establish the validity of asymptotically gaussian inference of our procedure.

\begin{assumption}
\label{assu rates sufficient}
Let $\tilde{p}:=p\vee N\vee T$ and assume
\begin{enumerate}[(a)]
\item 
\[\frac{(h_1\vee h_21\{h_1\neq 0\})^2\bar{G}^2\sbr[2]{\frac{(\log p)^3}{N}}^{-\vartheta}(\log\tilde{p})^7}{N}=o(1),\qquad   \frac{(\log (N\vee T))^31\{h_2\neq 0\}}{T}=o(1).\]
\item 
\[\frac{\del[2]{h_1^2\bar{G}^2\sbr[2]{\frac{(\log p)^3}{N}}^{-\vartheta}\vee Nh_2^2}  \left[ s_1\vee E\del[2]{\frac{(\log (p\vee N))^3}{T}}^{-\nu/2}\right] (\log\tilde{p})^5}{NT}=o(1).\]
\item 
\begin{align*}
& \frac{(h_1\vee h_2)\sbr[2]{\del[2]{\bar{G}\sbr[2]{\frac{(\log p)^3}{N}}^{-\vartheta/2}\vee (\log \tilde{p})^2}1\{h_1\neq 0\}\vee 1\{h_2\neq 0\}}\sbr[2]{s_1^2\vee E^2\del[2]{\frac{(\log (p\vee N))^3}{T}}^{-\nu}}(\log\tilde{p})^4}{N}=o(1).
\end{align*}
\item $\text{mineval}(\Sigma_{\Pi\varepsilon})$ is uniformly bounded away from zero and $\text{maxeval}(\Sigma_{1,N})$ is uniformly bounded from above.
\end{enumerate}
\end{assumption}

Assumption \ref{assu rates sufficient} is slightly stronger than what we actually need in order to prove Theorem \ref{thm Delta is small op 1} but it is less cluttered in terms of notation. Assumption \ref{assu rates sufficient} restricts the rate at which $p$, $T$, $s_1$, $E$, $\bar{G}$, $h_1$ and $h_2$ are allowed to increase as none of these are assumed to be bounded. First, note that $p=L+p_x$ only enters through its logarithm. Thus, we can allow for very high-dimensional models. Furthermore, $h_1$ as well as $h_2$ are allowed to increase with the sample size such that hypotheses of an increasing dimension involving $\alpha$ and $\eta$ \textit{simultaneously} can be tested. In the classical setting where one is only interested in testing hypotheses on $\alpha$ one has that $h_2=0$ such that Assumption \ref{assu rates sufficient} simplifies. The case of hypotheses only involving the fixed effects $\eta$ corresponds to $h_1=0$ and again the assumptions simplify. We also note that Assumption \ref{assu rates sufficient} requires $\bar{G}$ and $h_1$ necessarily to be $o(N^{\frac{1-\vartheta}{2}})$, $s_1$ necessarily to be $o(N^{1/2})$, $E$ necessarily to be $o(\sqrt{N(\log (p\vee N))^{3\nu}/T^{\nu}})$, and $h_2$ necessarily be $o(T^{1/2})$. The restrictions on $h_1$ and $h_2$, i.e. the number of common coefficients and fixed effects involved in the hypothesis (no fixed effects need to be zero), thus clearly encompass the classical setting where one tests only a fixed number of parameters ($h_1$ and $h_2$ fixed). The assumptions of \ref{assu rates sufficient} are satisfied if, for example, $p=N, T=N^{1/2},\nu=\vartheta=0.5, s_1=N^{1/4}, E=N^{1/6}$ and $\bar{G}=N^{1/7}$. Thus, while we allow these quantities to diverge, the rate at which they do so must be under control.


\begin{thm}
\label{thm Delta is small op 1}
Let Assumptions \ref{assu panel data}, \ref{assu subgaussian}, \ref{assu more on eigen values}, and \ref{assu rates sufficient} be satisfied. If, furthermore, $\{\varepsilon_{i,t}\}_{t=1}^T$ is an independent sequence for all $i=1,...,N$, then 
\begin{align}
\frac{\rho'S\left( \tilde{\gamma}-\gamma\right)}{\sqrt{\rho'\hat{\Theta}\hat{\Sigma}_{\Pi\varepsilon}\hat{\Theta}'\rho}}\xrightarrow{d}N(0,1),\label{asygauss}
\end{align}
where $\rho=(\rho_1',\rho_2')'$ is a $(p+N)\times 1$ vector, with $\|\rho\|=1$, $\rho_1\in \mathbb{R}^p$ and $\rho_2\in \mathbb{R}^N$. Moreover, 
\begin{equation}
\label{eqn uniform denominator}
\sup_{\gamma\in \mathcal{F}(s_1,\nu,E)}|\rho'\hat{\Theta}\hat{\Sigma}_{\Pi\varepsilon}\hat{\Theta}'\rho-\rho' \Theta \Sigma_{\Pi\varepsilon} \Theta'\rho|=o_p(1).
\end{equation}
Finally, for every fixed set $H\subseteq\cbr[0]{1,...,N+p}$ with cardinality $h$, we have
\begin{align}
[ S_H(\tilde{\gamma}_H-\gamma_H)]'\del[1]{ \hat{\Theta}\hat{\Sigma}_{\Pi\varepsilon}\hat{\Theta}'}^{-1}_H [ S_H(\tilde{\gamma}_H-\gamma_H)] \xrightarrow{d}\chi_h^2.\label{chi2}
\end{align}
\end{thm}

Theorem \ref{thm Delta is small op 1} provides sufficient conditions under which our procedure allows for asymptotically gaussian inference. We stress again that  hypotheses involving an increasing number of parameters can be tested and that the total number of parameters in the model may be much larger than the sample size. Furthermore, the error terms are allowed to be conditionally heteroskedastic and we provide a consistent estimator of the asymptotic covariance matrix even for the case of hypotheses involving an increasing number of parameters. Indeed, this estimator converges uniformly over $\mathcal{F}(s_1,\nu,E)$ even for high-dimensional covariance matrices which we use in Theorem \ref{thm uniform convergence} to establish the honesty (uniform validity) over this set of confidence intervals based on (\ref{asygauss}). \cite{vandegeerbuhlmannritovdezeure2014} have derived similar results in the setting of the homoskedastic linear cross sectional model for the case of inference on a low-dimensional parameter. Thus, our results can be seen as an extension to dynamic panel data models. We stress again that we relax their assumption of the the inverse covariance matrix $\Theta_Z$ being exactly sparse which is important in dynamic models like ours. Furthermore, relaxing the homoskedasticity assumption is important as volatility is known to vary over time in dynamic models, see e.g. \cite{engle1982autoregressive}, and the conditional volatility often depends on the state of the process. Theorem \ref{thm Delta is small op 1} is also related to \cite{belloni2014inference} who consider inference in static panel data models for a low-dimensional parameter of interest.  

The classical setup where one is only interested in inference on $\alpha$ corresponds to $\rho_2=0$ such that $\sqrt{NT}\rho_1'\left( \tilde{\alpha}-\alpha\right)$ is asymptotically gaussian with variance equal to the limit of $\rho_1' \Theta_Z \Sigma_{1,N} \Theta_Z'\rho_1$ (assumed to exist for illustration). If, furthermore, $\varepsilon_{i,t}$ is homoskedastic with variance $\sigma^2$ and independent of $z_{i,t}$ for all $i=1,...,N$ and $t=1,...,T$, it follows from the definition of $\Sigma_{1,N}$ that this variance equals the limit of $\sigma^2\rho_1'\Theta_Z\rho_1=\sigma^2\rho_1'\Psi_Z^{-1}\rho_1$. The leading special case where one is interested in testing a hypothesis on the $j$'th entry of $\alpha$ corresponds to $\rho_1=e_j$. Similar reasoning shows that in the case where one is testing hypotheses involving fixed effects only, corresponding to $\rho_1=0$, one has that $\rho_2'\sqrt{T}\left( \tilde{\eta}-\eta\right)$ is asymptotically gaussian with variance $\sigma^2$. This simple form of the variance follows from the asymptotic independence of the components of $\tilde{\eta}$. Note that the different rates of convergence for $\tilde{\alpha}$ and $\tilde{\eta}$ are in accordance with Theorem \ref{thm probabilistic oracle inequality}.

(\ref{chi2}) is a straightforward consequence of (\ref{asygauss}) and reveals that classical $\chi^2$ inference can be carried out in the usual manner. Thus, asymptotically valid $\chi^2$-inference can be performed in order to test a hypothesis on $h$ parameters simultaneously. Wald tests of general restrictions of the type $H_0: g(\gamma)=0$ (where $g:\mathbb{R}^{p+N} \to\mathbb{R}^h$ is differentiable in an open neighborhood around $\gamma$ and has derivative matrix of rank $h$) can now also be constructed in the usual manner, see e.g. \cite{davidson00} Chapter 12, even when $p+N>NT$ which has hitherto been impossible.

Finally, the independence assumption on $\varepsilon_{i,t}$ across $t$ is needed only if one tests hypotheses involving $\{\eta_i\}_{i=1}^N$ ($h_2\neq 0$). Weaker assumptions on the error terms, such as strong mixing, are possible at the expense of more involved expressions but will not be pursued here.

\section{Honest Confidence Intervals}\label{unif}

In this section we show that the confidence bands based on (\ref{asygauss}) are honest (uniformly valid) and contract at the optimal rate. The precise result is contained in the following theorem. 

\begin{thm}
\label{thm uniform convergence}
Let Assumptions \ref{assu panel data}, \ref{assu subgaussian}, \ref{assu more on eigen values}, and \ref{assu rates sufficient} be satisfied. Then, for all $\rho\in \mathbb{R}^{p+N}$ with $\|\rho\|=1$,
\begin{equation}
\label{eqn major uniform convergence}
\sup_{t\in \mathbb{R}}\sup_{\gamma\in \mathcal{F}(s_1,\nu,E)}\envert[4]{\mathbb{P}\del[4]{ \frac{\rho'S\left( \tilde{\gamma}-\gamma\right)}{\sqrt{\rho'\hat{\Theta}\hat{\Sigma}_{\Pi\varepsilon}\hat{\Theta}'\rho}}\leq t} - \Phi(t) }=o(1), 
\end{equation}
where $\Phi(\cdot)$ is the CDF of the standard normal distribution. Furthermore, define $\tilde{\sigma}_{\alpha,j}:=\sqrt{[ \hat{\Theta}_{Z}\hat{\Sigma}_{1,N}\hat{\Theta}_{Z}] _{jj}}$ and $\tilde{\sigma}_{\eta,i}:=\sqrt{[ \hat{\Sigma}_{3,N}]_{ii}}$ for $j=1,...,p$ and $i=1,...,N$, respectively. Then,
\begin{align}
&\liminf_{N\rightarrow\infty}\inf_{\gamma\in \mathcal{F}(s_1,\nu,E)}\mathbb{P}\del[2]{\alpha_j \in \sbr[2]{\tilde{\alpha}_j-z_{1-\delta/2}\frac{\tilde{\sigma}_{\alpha,j}}{\sqrt{NT}}, \tilde{\alpha}_j+z_{1-\delta/2}\frac{\tilde{\sigma}_{\alpha,j}}{\sqrt{NT}} } }\geq 1-\delta,\label{align honest confidence interval alpha}\\
&\liminf_{N\rightarrow\infty}\inf_{\gamma\in \mathcal{F}(s_1,\nu,E)}\mathbb{P}\del[2]{\eta_i \in \sbr[2]{\tilde{\eta}_i-z_{1-\delta/2}\frac{\tilde{\sigma}_{\eta,i}}{\sqrt{T}}, \tilde{\eta}_i+z_{1-\delta/2}\frac{\tilde{\sigma}_{\eta,i}}{\sqrt{T}} } }\geq 1-\delta,\label{align honest confidence interval eta}
\end{align}
for $j=1,...,p$ and $i=1,...,N$, respectively, where $z_{1-\delta/2}$ is the $1-\delta/2$ percentile of the standard normal distribution. Finally, letting $\text{diam}([a,b] )=b-a$ be the length (which coincides with the Lebesgue measure of $[a,b]$) of an interval $[a,b]$ in the real line, we have
\begin{align}
&\sup_{\gamma \in\mathcal{F}(s_1,\nu,E)}\text{diam}\del[2]{ \sbr[2]{\tilde{\alpha}_j-z_{1-\delta/2}\frac{\tilde{\sigma}_{\alpha,j}}{\sqrt{NT}}, \tilde{\alpha}_j+z_{1-\delta/2}\frac{\tilde{\sigma}_{\alpha,j}}{\sqrt{NT}} } }=O_p\del[2]{ \frac{1}{\sqrt{NT}}} ,\label{align confidence interval contract at right rate alpha}\\
&\sup_{\gamma \in\mathcal{F}(s_1,\nu,E)}\text{diam}\del[2]{ \sbr[2]{\tilde{\eta}_i-z_{1-\delta/2}\frac{\tilde{\sigma}_{\eta,i}}{\sqrt{T}}, \tilde{\eta}_i+z_{1-\delta/2}\frac{\tilde{\sigma}_{\eta,i}}{\sqrt{T}} } }=O_p\del[2]{ \frac{1}{\sqrt{T}}},\label{align confidence interval contract at right rate eta}
\end{align}
for $j=1,...,p$ and $i=1,...,N$, respectively.
\end{thm}

(\ref{eqn major uniform convergence}) reveals that the convergence to the normal distribution in Theorem \ref{thm Delta is small op 1} is uniform over $\mathcal{F}(s_1,\nu,E)$. Since the desparsified Lasso is not a sparse estimator this uniform convergence does not contradict the work of \cite{leebpotscher2005}. Next, (\ref{align honest confidence interval alpha}) is a direct consequence of (\ref{eqn major uniform convergence}) and reveals that the desparsified Lasso produces confidence bands which are \textit{honest} (uniform) over $\mathcal{F}(s_1,\nu,E)$. Honest confidence bands are important in practical applications of dynamic panel data models as they guarantee the existence of an $N_0$, not depending on $\gamma\in\mathcal{F}(s_1,\nu,E)$, such that $\sbr[1]{\tilde{\alpha}_j-z_{1-\delta/2}\frac{\tilde{\sigma}_{\alpha,j}}{\sqrt{NT}}, \tilde{\alpha}_j+z_{1-\delta/2}\frac{\tilde{\sigma}_{\alpha,j}}{\sqrt{NT}}}$ covers $\alpha_j$ with probability not much smaller than $1-\delta$. Here the important point is that one and the same $N_0$ guarantees this coverage, irrespective of the true value of $\gamma\in\mathcal{F}(s_1,\nu,E)$. On the other hand, pointwise consistent confidence bands only guarantee that 
\begin{align*}
\inf_{\gamma\in\mathcal{F}(s_1,\nu,E)}\liminf_{N\rightarrow\infty}\mathbb{P}\del[2]{\alpha_j \in \sbr[2]{\tilde{\alpha}_j-z_{1-\delta/2}\frac{\tilde{\sigma}_{\alpha,j}}{\sqrt{NT}}, \tilde{\alpha}_j+z_{1-\delta/2}\frac{\tilde{\sigma}_{\alpha,j}}{\sqrt{NT}} } }\geq 1-\delta,
\end{align*}
implying that the value of $N$ needed in order to guarantee a coverage of close to $1-\delta$ may depend on the \textit{unknown} true parameter. Thus, for some parameter values one may have to sample more data points to achieve the desired coverage than for others which is unfortunate as one does not know for which parameters this is the case. An honest confidence set $S_N$ for $\alpha_j$ can of course trivially be obtained by setting $S_N=\mathbb{R}$. However, this is clearly not very informative and therefore (\ref{align confidence interval contract at right rate alpha}) is reassuring as it guarantees that the length of the honest confidence interval contracts at the optimal rate. In particular, the confidence bands are uniformly narrow over $\mathcal{F}(s_1,\nu,E)$ in the sense that for any $\epsilon>0$ there exists an $M>0$ such that $\text{diam}\del[2]{ \sbr[2]{\tilde{\alpha}_j-z_{1-\delta/2}\frac{\tilde{\sigma}_{\alpha,j}}{\sqrt{NT}}, \tilde{\alpha}_j+z_{1-\delta/2}\frac{\tilde{\sigma}_{\alpha,j}}{\sqrt{NT}} } }\leq \frac{M}{\sqrt{NT}}$ for all $\gamma\in \mathcal{F}(s_1,\nu,E)$ with probability at least $1-\epsilon$. Therefore, our confidence bands are not only honest, they are also very \textit{informative} as they contract as fast as possible and this contraction is uniform over $\mathcal{F}(s_1,\nu,E)$. Since the desparsified  Lasso is not a sparse estimator, this fast contraction does not contradict inequality 6 in Theorem 2 of \cite{potscher2009confidence} who shows that honest confidence bands based on sparse estimators must be large. 

Similarly to the confidence bands pertaining to $\alpha$, the ones for the fixed effects are also honest and contract at the optimal rate. Note that this rate is again slower than the one for $\alpha$. It is also worth remarking that the above inference results are valid without any sort of lower bound on the non-zero coefficients as inference is not conducted after model selection.

\section{Monte Carlo}\label{monte}

In this section we investigate the finite sample properties of our estimator by means of simulations. All calculations are carried out in R using the \texttt{glmnet} package and $\lambda_N$ and $\lambda_{node}$ are chosen via BIC by the formula given in (9.4.9) in \cite{davidson00}. Cross validation was also considered, but this did not alter the results while being considerably slower. We leave it for future work to establish theoretical performance guarantees on these procedures in the setting of high-dimensional dynamic panel data models. 

The data generating process is (\ref{eqn dgp}) and in all experiments $(\alpha_1,\alpha_2,\alpha_3,\alpha_4)=(0.9, 0, 0, -0.3)$ such that the roots of the corresponding lag polynomial lie outside the unit disk. In practice, one might not know the true lag length and usually specifies a reasonably large number of lags (to test downwards). To reflect this in our simulations, we always included 5 lags but also experimented with more than 5 lags. The results were not sensitive to this.

For each $i=1,...,N$, the $x_{i,t}$ are generated according to the autoregressive structure
\[x_{i,t}=a_x x_{i,t-1}+e_{distur,i,t},\]
where the $e_{distur,i,t}$ are $p_x\times 1$ random disturbance vectors independent across $i$ and $t$. $a_x$ is an autoregressive scalar which controls the temporal dependence of $x_{i,t}$. For simplicity, we restrict $a_x$ to be the same across $i$. When $a_x=0$, we have temporal independence across $t$ for $x_{i,t}$. Since Assumption \ref{assu panel data} does not restrict any temporal dependence of $x_{i,t}$, we set $a_x=0.5$. Our simulation results are reasonably robust to the choice of $a_x$. The covariance matrix of $e_{distur,i,t}$ is chosen to have a Toeplitz structure with the $(i, j)$th entry equal to $\rho^{|i-j|}$ with $\rho=0.75$. We also experimented with other choices of $\rho$ which did not change the results dramatically. Furthermore, we also tried to let the covariance matrix of $e_{distur,i,t}$ be block-diagonal. Again, this did not alter our results.

We allow the fixed effect $\eta_i$ to depend on the initial observation of $x_i$:
\[\eta_i=x_{i,1}'b_{\eta}/\sqrt{\log p_x}\qquad i=1,\ldots,N,\]
where $b_{\eta}$ is a $p_x\times 1$ vector whose entries are drawn from standard normal and normalized to have unit $\ell_1$-norm. Note that $|\eta_i|\leq \|x_{i,1}/\sqrt{\log p_x}\|_{\infty}\|b_\eta\|_1= \|x_{i,1}/\sqrt{\log p_x}\|_{\infty}$. If $x_{i,1}$ is multivariate normal, then $\|x_{i,1}\|_{\infty}=O_p(\sqrt{\log p_x})$. In this sense, $\eta_i$ is bounded. However, $\eta$ is not necessarily weakly sparse and thus we also investigate how robust our results are to violations of this assumption. Of course our estimator performed much better in the truly weakly sparse setting than the setting we present here (results available upon request).
 
As our theory allows for heteroskedasticity, we also investigate the effect of this. To
be precise, we consider error terms of the form $\varepsilon_{i,t}=u_{i,t}\del[1]{x_{i,t,1}/\sqrt{2}+b_x x_{i,t,2}}$ where $u_{i,t}$ is independent of $y_{i,t-1},...,y_{i,1-L}$ and $x_{i,t},...,x_{i,1}$. $b_x$ is chosen such that the unconditional variance of $\varepsilon_{i,t}$ is the same as the one of $u_{i,t}$ which in turns equals the one from the homoskedastic case. A simple calculation reveals that $b_x=\del[1]{-\sqrt{2}\rho+\sqrt{2\rho^2+2-4a_x^2}}/2$. Note that $\varepsilon_{i,t}$ constructed this way satisfies Assumption \ref{assu panel data}. The reason we ensure that the unconditional variance is the same as in the homoskedastic case is that we do not want any findings in the heteroskedastic case to be driven by a plain change in the unconditional variance.

Our estimator is compared to the least squares oracle which only includes variables with non-zero coefficients on top of those variables we wish to test hypothesis about. Thus, it is an oracle which knows the relevant control variables. When sample size allows it, that is when $p+N\leq NT$, we also implement naive least squares including all variables. This estimator is numerically equivalent to the often used within estimator. Finally, we implemented the desparsified conservative Lasso of \cite{caner2014asymptotically}. However, this only improved the results slightly and so we do not report these results here. The number of Monte Carlo replications is 1,000 for all setups and we consider the performance of our estimator along the following dimensions:

\begin{enumerate}
\item Estimation error: We compute the root mean square errors (RMSE) of all procedures averaged over the Monte Carlo replications.
\item Coverage rate: We calculate the coverage rate of a gaussian confidence interval constructed as in Theorem \ref{thm uniform convergence}. This is done for three coefficients of regressors in $x_{i,t}$.
\item Length of confidence interval: We calculate the length of the three confidence intervals considered in point 2 above.
\item Size: We evaluate the size of the $\chi^2$-test in Theorem \ref{thm Delta is small op 1} for a hypothesis involving the same three parameters we construct confidence intervals for in point 2 above.
\item Power: We evaluate the power of the $\chi^2$-test in point 4 above. 
\end{enumerate}
All tests are carried out at the $5\%$ level of significance and all confidence intervals have a nominal coverage of $95\%$. Furthermore, as our results regarding estimation error are for the plain Lasso, the root mean square errors are reported for this instead of the desparsified Lasso. As our models are dynamic, we allow for a burn-in period of 1,000 observations when generating the data.

The following experiments were carried out

\begin{itemize}
\item Experiment 1: (moderate-dimensional setting): $N=20$ and $T=10$. $\beta$ is $100\times 1$ with five equidistant non-zero entries equaling one. Thus, $p=105$ and $s_1=7$. In total, $\gamma=(\alpha', \eta')'$ is $125\times 1$. The disturbances of $x_{i,t}$, $e_{distur,i,t}$, are gaussian and $\varepsilon_{i,t}$ are standard gaussian. We test the true hypothesis 
\begin{align*}
H_{0}:(\gamma_7, \gamma_{27}, \gamma_{47})=(0,0,0)
\end{align*}
by the $\chi^2_3$ test described in Theorem \ref{thm Delta is small op 1} in order to gauge the size of the test. The power is investigated by the hypothesis
\begin{align*}
H_{0}:(\gamma_7, \gamma_{27}, \gamma_{47})=(0.4,0,0).
\end{align*}
The following variations of this setting are considered
\begin{enumerate}[(a)]
\item The baseline case described so far.
\item Same as (a) but with heteroskedastic errors.
\item Same as (b) but $e_{distur,i,t}$ and $\varepsilon_{i,t}$ are $t$-distributed with 3 degrees of freedom. In this case, even $\eta_i$ may not be $O_p(1)$.
\end{enumerate}
\item Experiment 2: (high-dimensional setting). $N=20$ and $T=10$. $\beta$ is $400\times 1$ with five equidistant non-zero entries equaling one. Thus, $p=405$ and $s_1=7$. In total, $\gamma=(\alpha', \eta')'$ is $425\times 1$. The disturbances of $x_{i,t}$, $e_{distur,i,t}$, are gaussian and $\varepsilon_{i,t}$ are standard gaussian. We test the true hypothesis 
\begin{align*}
H_{0}:(\gamma_7, \gamma_{87}, \gamma_{167})=(0,0,0)
\end{align*}
by the $\chi^2_3$ test described in Theorem \ref{thm Delta is small op 1} in order to gauge the size of the test. The power is investigated by the hypothesis
\begin{align*}
H_{0}:(\gamma_7, \gamma_{87}, \gamma_{167})=(0.4,0,0).
\end{align*}
The following variations of this setting are considered
\begin{enumerate}[(a)]
\item The baseline case described so far.
\item Same as (a) but with heteroskedastic errors.
\item Same as (b) but $e_{distur,i,t}$ and $\varepsilon_{i,t}$ are $t$-distributed with 3 degrees of freedom. In this case, even $\eta_i$ may not be $O_p(1)$.
\end{enumerate}
\item Experiment 3: (increase $T$): As Experiment 2 but with $T=40$.
\item Experiment 4: (increase $N$): As Experiment 2 but with $N=40$. 

\item Experiment 5: (high-dimensional setting 2). $N=20$ and $T=40$. $\beta$ is $1005\times 1$ with 15 equidistant non-zero entries equaling one. Thus, $p=1010$ and $s_1=17$. In total, $\gamma=(\alpha', \eta')'$ is $1030\times 1$. The disturbances of $x_{i,t}$, $e_{distur,i,t}$, are gaussian and $\varepsilon_{i,t}$ are standard gaussian. We test the true hypothesis 
\begin{align*}
H_{0}:(\gamma_7, \gamma_{74}, \gamma_{141})=(0,0,0)
\end{align*}
by the $\chi^2_3$ test described in Theorem \ref{thm Delta is small op 1} in order to gauge the size of the test. The power is investigated by the hypothesis
\begin{align*}
H_{0}:(\gamma_7, \gamma_{74}, \gamma_{141})=(0.4,0,0).
\end{align*}
The following variations of this setting are considered
\begin{enumerate}[(a)]
\item The baseline case described so far.
\item Same as (a) but with heteroskedastic errors.
\item Same as (b) but $e_{distur,i,t}$ and $\varepsilon_{i,t}$ are $t$-distributed with 3 degrees of freedom. In this case, even $\eta_i$ may not be $O_p(1)$.
\end{enumerate}
\end{itemize}

\begin{table}[h] 
\centering 
\begin{tabular}{cccccccccccc}
\toprule
& &\multicolumn{2}{c}{ RMSE} & \multicolumn{3}{c}{Coverage}& \multicolumn{3}{c}{Length}&&\\
\cmidrule(lr){3-4} 
\cmidrule(lr){5-7}
\cmidrule(lr){8-10}
& & $\alpha$& $\eta$ & $\gamma_7$ & $\gamma_{27}$ & $\gamma_{47}$ &  $\gamma_7$ & $\gamma_{27}$ & $\gamma_{47}$ & Size & Power \\ 
\midrule
\multirow{3}{*}{\small1(a)}
& LS & 23.744 & 16.382 & 0.762 & 0.786 & 0.766 & 0.552 & 0.549 & 0.553 & 0.412 & 0.741 \\ 
&   DL & 3.061 & 8.528 & 0.892 & 0.918 & 0.884 & 0.395 & 0.396 & 0.396 & 0.150 & 0.852 \\ 
&   Ora & 0.523 & 7.226 & 0.933 & 0.939 & 0.919 & 0.402 & 0.403 & 0.404 & 0.074 & 0.907 \\ 
\midrule
\multirow{3}{*}{\small1(b)}
& LS & 23.796 & 16.444 & 0.732 & 0.760 & 0.747 & 0.548 & 0.543 & 0.547 & 0.453 & 0.747 \\ 
&  DL & 3.011 & 8.298 & 0.920 & 0.914 & 0.903 & 0.408 & 0.389 & 0.391 & 0.135 & 0.846 \\ 
&  Ora & 0.524 & 7.079 & 0.920 & 0.931 & 0.937 & 0.401 & 0.393 & 0.398 & 0.092 & 0.904 \\ 
\midrule
\multirow{3}{*}{\small1(c)}
& LS & 46.140 & 51.979 & 0.753 & 0.772 & 0.743 & 0.910 & 0.883 & 0.889 & 0.432 & 0.605 \\ 
&  DL & 4.747 & 23.939 & 0.912 & 0.901 & 0.883 & 0.619 & 0.544 & 0.567 & 0.159 & 0.632 \\ 
&  Ora & 1.200 & 23.596 & 0.907 & 0.937 & 0.913 & 0.662 & 0.617 & 0.617 & 0.091 & 0.651 \\ 
\bottomrule
\end{tabular}
\caption{\small Experiment 1. LS, DL and Ora: least squares including all variables, desparsified Lasso and least squares oracle. RMSE: root mean square error. Coverage: the coverage rate of the asymptotic 95\% confidence intervals. Length: the average length of the asymptotic 95\% confidence intervals. Size: size of the correct hypothesis $H_0: (\gamma_{7}, \gamma_{27},\gamma_{47})=(0,0,0)$. Power: the probability to reject the false $H_0: (\gamma_{7}, \gamma_{27},\gamma_{47})=(0.4,0,0)$.}
\label{Exp1}
\end{table}

Table \ref{Exp1} contains the results of experiment 1. 
Setting 1(a) reveals that the RMSE of the Lasso are lower than those for least squares including all variables but higher than those of least squares only including the relevant variables. This is the case for $\alpha$ as well as the fixed effects. Next, it is very encouraging that the coverage probabilities for the desparsified Lasso are close to the ones based on the oracle. The length of the confidence intervals are also comparable for those two procedures while the ones based on the within estimator are considerably wider while still having a lower coverage. The oracle and the desparsfied Lasso both produce tests which are a bit oversized but they are still much better than the within estimator. The same is true when it comes to power.

Experiment 1(b) adds heteroskedasticity to the error terms and none of the procedures is affected by this. 

In Panel 1(c) the random variable have heavy tails. Overall, and as expected, all procedures suffer from this. However, it is worth mentioning that the coverage rate of the confidence intervals does not decrease. Instead, the length of these intervals increases to reflect the larger uncertainty. The size of the significance test is not affected either while the power suffers. 
 
\begin{table}[h] 
\centering 
\begin{tabular}{cccccccccccc}
\toprule
& &\multicolumn{2}{c}{ RMSE} & \multicolumn{3}{c}{Coverage}& \multicolumn{3}{c}{Length}&&\\
\cmidrule(lr){3-4} 
\cmidrule(lr){5-7}
\cmidrule(lr){8-10}
& & $\alpha$& $\eta$ & $\gamma_7$ & $\gamma_{87}$ & $\gamma_{167}$ &  $\gamma_7$ & $\gamma_{87}$ & $\gamma_{167}$ & Size & Power \\ 
\midrule
\multirow{3}{*}{\small2(a)}
& LS &  &  &  &  &  &  &  &  &  &  \\ 
&   DL & 4.209 & 8.333 & 0.875 & 0.893 & 0.881 & 0.386 & 0.385 & 0.386 & 0.189 & 0.841 \\ 
&   Ora & 0.513 & 7.103 & 0.919 & 0.918 & 0.924 & 0.402 & 0.403 & 0.403 & 0.110 & 0.922 \\ 
\midrule
\multirow{3}{*}{\small2(b)}
& LS &  &  &  &  &  &  &  &  &  &  \\ 
&  DL & 4.165 & 8.322 & 0.896 & 0.872 & 0.861 & 0.407 & 0.379 & 0.381 & 0.189 & 0.825 \\ 
&  Ora & 0.535 & 7.074 & 0.906 & 0.913 & 0.929 & 0.401 & 0.396 & 0.397 & 0.101 & 0.899 \\ 
\midrule
\multirow{3}{*}{\small2(c)}
& LS &  &  &  &  &  &  &  &  &  &  \\ 
&  DL & 7.602 & 22.895 & 0.916 & 0.868 & 0.882 & 0.622 & 0.543 & 0.551 & 0.193 & 0.602 \\ 
&  Ora & 1.074 & 21.724 & 0.922 & 0.944 & 0.944 & 0.657 & 0.619 & 0.632 & 0.076 & 0.674 \\ 
\bottomrule
\end{tabular}
\caption{\small Experiment 2. LS, DL and Ora: least squares including all variables, desparsified Lasso and least squares oracle. RMSE: root mean square error. Coverage: the coverage rate of the asymptotic 95\% confidence intervals. Length: the average length of the asymptotic 95\% confidence intervals. Size: size of the correct hypothesis $H_0: (\gamma_{7}, \gamma_{87},\gamma_{167})=(0,0,0)$. Power: the probability to reject the false $H_0: (\gamma_{7}, \gamma_{87},\gamma_{167})=(0.4,0,0)$. }
\label{Exp2}
\end{table}

Next, we turn to experiment 2(a) which is high-dimensional. The results can be found in Table \ref{Exp2}. As expected, the estimation error is higher for the Lasso than for the oracle. However, it is encouraging that the confidence intervals produced by the desparsified Lasso have a coverage which is as almost as good as the one for the the oracle. In fact, the coverage rate is close to identical to the one in the above moderate-dimensional simulation. The length of the confidence bands based on the desparsifed Lasso is actually slightly lower than the ones based on the oracle which explains their slightly worse performance. The siginificance test is again a bit oversized for the oracle as well as the desparsified Lasso but the size is not far from the one in Table \ref{Exp1}. Power is also virtually unaffected by the increase in dimension.

Experiment 2(b) adds heteroskedasticity and the results are not affected by this. Finally, the addition of heavy tails in Experiment 2(c) makes the estimators less precise. However, and as in the moderate-dimensional setting above, the coverage remains high since the confidence bands get wider. The size of the significance test is unaffected while the power goes down for the oracle as well as the depsarsified Lasso.

\begin{table}[h] 
\centering 
\begin{tabular}{cccccccccccc}
\toprule
& &\multicolumn{2}{c}{ RMSE} & \multicolumn{3}{c}{Coverage}& \multicolumn{3}{c}{Length}&&\\
\cmidrule(lr){3-4} 
\cmidrule(lr){5-7}
\cmidrule(lr){8-10}
& & $\alpha$& $\eta$ & $\gamma_7$ & $\gamma_{87}$ & $\gamma_{167}$ &  $\gamma_7$ & $\gamma_{87}$ & $\gamma_{167}$ & Size & Power \\ 
\midrule
\multirow{3}{*}{\small3(a)}
& LS & 40.341 & 7.367 & 0.815 & 0.827 & 0.796 & 0.266 & 0.266 & 0.268 & 0.325 & 0.993 \\ 
&  DL & 1.208 & 2.121 & 0.931 & 0.918 & 0.925 & 0.190 & 0.190 & 0.190 & 0.098 & 1.000 \\ 
&  Ora & 0.223 & 3.208 & 0.943 & 0.945 & 0.933 & 0.186 & 0.187 & 0.187 & 0.052 & 1.000 \\ 
\midrule
\multirow{3}{*}{\small3(b)}
& LS & 40.351 & 7.376 & 0.800 & 0.823 & 0.813 & 0.267 & 0.265 & 0.267 & 0.318 & 0.993 \\ 
&  DL & 1.169 & 2.139 & 0.923 & 0.915 & 0.924 & 0.200 & 0.189 & 0.189 & 0.104 & 1.000 \\ 
&  Ora & 0.233 & 3.235 & 0.955 & 0.940 & 0.953 & 0.189 & 0.185 & 0.186 & 0.060 & 1.000 \\ 
\midrule
\multirow{3}{*}{\small3(c)}
& LS & 88.649 & 29.738 & 0.766 & 0.813 & 0.837 & 0.460 & 0.452 & 0.448 & 0.315 & 0.821 \\ 
&  DL & 2.630 & 7.689 & 0.931 & 0.922 & 0.913 & 0.359 & 0.309 & 0.319 & 0.090 & 0.926 \\ 
&  Ora & 0.610 & 10.759 & 0.930 & 0.963 & 0.951 & 0.329 & 0.299 & 0.302 & 0.056 & 0.962 \\ 
\bottomrule
\end{tabular}
\caption{\small Experiment 3.  LS, DL and Ora: least squares including all variables, desparsified Lasso and least squares oracle. RMSE: root mean square error. Coverage: the coverage rate of the asymptotic 95\% confidence intervals. Length: the average length of the asymptotic 95\% confidence intervals. Size: size of the correct hypothesis $H_0: (\gamma_{7}, \gamma_{87},\gamma_{167})=(0,0,0)$. Power: the probability to reject the false $H_0: (\gamma_{7}, \gamma_{87},\gamma_{167})=(0.4,0,0)$. }
\label{Exp3}
\end{table}

In Table \ref{Exp3}, $T$ has been increased to $40$ compared to Table \ref{Exp2}. This results in lower estimation errors for the Lasso as well as oracle assisted least squares. The coverage rates of the confidence bands also improve and get closer to the nominal rate. At the same time, the bands also get narrower. The size of the significance test also improves and the power of the oracle and the desparsifed Lasso is 1. As above, adding heteroskedasticity does not alter the results. The consequence of heavy tails are also the same: higher estimation error, no change in coverage of confidence bands, wider bands, unchanged size, but lower power.

\begin{table}[h] 
\centering 
\begin{tabular}{cccccccccccc}
\toprule
& &\multicolumn{2}{c}{ RMSE} & \multicolumn{3}{c}{Coverage}& \multicolumn{3}{c}{Length}&&\\
\cmidrule(lr){3-4} 
\cmidrule(lr){5-7}
\cmidrule(lr){8-10}
& & $\alpha$& $\eta$ & $\gamma_7$ & $\gamma_{87}$ & $\gamma_{167}$ &  $\gamma_7$ & $\gamma_{87}$ & $\gamma_{167}$ & Size & Power \\ 
\midrule
\multirow{3}{*}{\small4(a)}
& LS &  &  &  &  &  &  &  &  &  &  \\ 
&  DL & 2.145 & 14.002 & 0.924 & 0.891 & 0.920 & 0.261 & 0.262 & 0.263 & 0.123 & 0.988 \\ 
&  Ora & 0.351 & 13.570 & 0.936 & 0.934 & 0.930 & 0.282 & 0.283 & 0.285 & 0.065 & 0.999 \\ 
\midrule
\multirow{3}{*}{\small4(b)}
& LS &  &  &  &  &  &  &  &  &  &  \\ 
&  DL & 2.145 & 14.073 & 0.933 & 0.904 & 0.911 & 0.275 & 0.260 & 0.263 & 0.114 & 0.979 \\ 
&  Ora & 0.368 & 13.469 & 0.926 & 0.931 & 0.940 & 0.283 & 0.281 & 0.282 & 0.076 & 1.000 \\ 
\midrule
\multirow{3}{*}{\small4(c)}
& LS &  &  &  &  &  &  &  &  &  &  \\ 
&  DL & 4.305 & 42.303 & 0.917 & 0.899 & 0.903 & 0.456 & 0.386 & 0.390 & 0.139 & 0.825 \\ 
&  Ora & 0.782 & 41.269 & 0.926 & 0.934 & 0.926 & 0.465 & 0.428 & 0.435 & 0.087 & 0.870 \\ 
\bottomrule
\end{tabular}
\caption{\small Experiment 4.  LS, DL and Ora: least squares including all variables, desparsified Lasso and least squares oracle. RMSE: root mean square error. Coverage: the coverage rate of the asymptotic 95\% confidence intervals. Length: the average length of the asymptotic 95\% confidence intervals. Size: size of the correct hypothesis $H_0: (\gamma_{7}, \gamma_{87},\gamma_{167})=(0,0,0)$. Power: the probability to reject the false $H_0: (\gamma_{7}, \gamma_{87},\gamma_{167})=(0.4,0,0)$. }
\label{Exp4}
\end{table}

Table \ref{Exp4} increases $N$ to 40 compared to Table \ref{Exp2}. This results in more fixed effects to be estimated. Thus, it is not surprising that the estimation error for $\alpha$ goes down while the one for $\eta$ increases. The coverage rates of the oracle as well as the desparsified Lasso improve compared to Table \ref{Exp2}. However, the length of the confidence bands does not decrease as much as when $T$ was increased in the previous experiment. The size of the significance test decreases when increasing $N$ while power is close to 1. Adding heteroskedasticity has no consequences while the presence of heavy tails has the usual effect.

\begin{table}[h] 
\centering 
\begin{tabular}{cccccccccccc}
\toprule
& &\multicolumn{2}{c}{ RMSE} & \multicolumn{3}{c}{Coverage}& \multicolumn{3}{c}{Length}&&\\
\cmidrule(lr){3-4} 
\cmidrule(lr){5-7}
\cmidrule(lr){8-10}
& & $\alpha$& $\eta$ & $\gamma_7$ & $\gamma_{74}$ & $\gamma_{141}$ &  $\gamma_7$ & $\gamma_{74}$ & $\gamma_{141}$ & Size & Power \\ 
\midrule
\multirow{3}{*}{\small5(a)}
& LS &  &  &  &  &  &  &  &  &  &  \\ 
&   DL & 3.342 & 2.463 & 0.922 & 0.903 & 0.912 & 0.209 & 0.209 & 0.208 & 0.124 & 0.997 \\ 
&   Ora & 0.546 & 3.305 & 0.956 & 0.942 & 0.949 & 0.187 & 0.187 & 0.187 & 0.062 & 1.000 \\ 
\midrule
\multirow{3}{*}{\small5(b)}
& LS &  &  &  &  &  &  &  &  &  &  \\ 
&   DL & 3.327 & 2.432 & 0.913 & 0.916 & 0.896 & 0.218 & 0.207 & 0.208 & 0.127 & 0.998 \\ 
&   Ora & 0.556 & 3.261 & 0.942 & 0.951 & 0.921 & 0.190 & 0.186 & 0.186 & 0.065 & 1.000 \\ 
\midrule
\multirow{3}{*}{\small5(c)}
& LS &  &  &  &  &  &  &  &  &  &  \\ 
&   DL & 7.294 & 7.273 & 0.936 & 0.910 & 0.916 & 0.363 & 0.307 & 0.304 & 0.101 & 0.920 \\ 
&   Ora & 1.072 & 9.693 & 0.952 & 0.936 & 0.951 & 0.326 & 0.297 & 0.293 & 0.061 & 0.955 \\ 
\bottomrule
\end{tabular}
\caption{\small Experiment 5.  LS, DL and Ora: least squares including all variables, desparsified Lasso and least squares oracle. RMSE: root mean square error. Coverage: the coverage rate of the asymptotic 95\% confidence intervals. Length: the average length of the asymptotic 95\% confidence intervals. Size: size of the correct hypothesis $H_0: (\gamma_{7}, \gamma_{74},\gamma_{141})=(0,0,0)$. Power: the probability to reject the false $H_0: (\gamma_{7}, \gamma_{74},\gamma_{141})=(0.4,0,0)$. }
\label{Exp5}
\end{table}

Table \ref{Exp5} contains a setting with more than 1000 variables. The main message of the previous tables prevails even in this setting: the coverage of the Lasso-based confidence intervals is almost as good as the ones based on the oracle. On the other hand, the bands of the former are now slightly wider than the ones of the latter. Both procedures have power close to one while the Lasso-based  test is a bit oversized compared to the oracle based test. Heteroskedasticity does not affect the results. The consequences of heavy tails are the same as in the previous experiments.

\section{Conclusion}\label{conclusion}

This paper has considered inference in high-dimensional dynamic panel data models with fixed effects. In particular we have shown how hypotheses involving an increasing number of variables can be tested. These hypotheses can involve parameters from all groups of variables in the model simultaneously. As a stepping stone towards this inference we constructed a uniformly valid estimator of the covariance matrix of the parameter estimates which is robust towards conditional heteroskedasticity. We also stress that our theory does not require the inverse covariance matrix of the covariates to be exactly sparse.

Next, we showed that confidence bands based on our procedure are asymptotically honest and contract at the optimal rate. This rate of contraction depends on which type of parameter is under consideration. Simulations revealed that our procedure works well in finite samples. Future work may include relaxing the sparsity assumption on the inverse covariance matrix $\Theta_Z$ as well as extending our results to non-linear panel data models.

\section{Appendix A}

\subsection{Sufficient Conditions for $y_{i,t}$ to be Subgaussian}
The following Lemma provides sufficient conditions for $y_{i,t}$ to inherit the subgaussianity from the covariates and the error terms. It allows for a wide range of models but rules out dynamic panel data models which are explosive or contain unit roots.
\begin{lemma}
\label{lemma y inherit subgaussian}
Let $x_{i,t,k}$ and $\varepsilon_{i,t}$ be uniformly subgaussian for $i=1,...,N$, $t=1,...,T$ and $k=1,...,p_x$ and assume that $\enVert[0]{\beta}_1\leq C$ for some $C>0$ for all $N$ and $T$. Furthermore, $\max_{1\leq i\leq N}|\eta_i|$ is bounded uniformly in $N$ and $T$. Then, if all roots of $1-\sum_{j=1}^L\alpha_jz^j$ ($\alpha_1,\ldots,\alpha_L$ fixed) are outside the unit disc, $y_{i,t}$ is uniformly subgaussian for $i=1,...,N$ and $t=1,...,T$.
\end{lemma}

\begin{proof}[Proof of Lemma \ref{lemma y inherit subgaussian}]
Let $y_t=\sum_{j=1}^L\alpha_jy_{t-j}+u_t$ be an AR($L$) process with roots outside the unit disc. Write the companion form as $\xi_t=F\xi_{t-1}+v_t$. Then, by the monotone convergence theorem for Orlicz norms, see \cite{vandervaartWellner1996} exercise 6, page 105, $\enVert[1]{\enVert[0]{\xi_t}}_{\psi_2}\leq\enVert[1]{\sum_{j=1}^\infty \|F^j\|_{\ell_2}\enVert[0]{v_{t-j}}}_{\psi_2}=\sum_{j=1}^\infty\|F^j\|_{\ell_2}\enVert[1]{\enVert[0]{v_{t-j}}}_{\psi_2}=\sum_{j=1}^\infty\|F^j\|_{\ell_2}\enVert[0]{u_{t-j}}_{\psi_2}$, where $\|\cdot\|_{\ell_2}$ is the $\ell_2$ induced norm, and the last equality used that $v_t$ is $L\times 1$ with only one non-zero entry equaling $u_t$. By Corollary 5.6.14 in \cite{horn2012matrix} there exists a $1>\delta>0$ such that $\|F^j\|_{\ell_2}\leq (1-\delta)^j$ for $j$ sufficiently large. Thus, if $\enVert[0]{u_{t}}_{\psi_2}$ is uniformly bounded we conclude $\enVert[0]{y_t}_{\psi_2}\leq \enVert[1]{\enVert[0]{\xi_t}}_{\psi_2}\leq K$ for some $K>0$. Thus, in our context it suffices to show that $\enVert[0]{x_{i,t}'\beta+\eta_i+\varepsilon_{i,t}}_{\psi_2}$ is uniformly bounded as $y_{i,t}=\sum_{j=1}^{L}\alpha_{j}y_{i,t-j}+x_{i,t}'\beta+\eta_{i}+\varepsilon_{i,t}=\sum_{j=1}^{L}\alpha_{j}y_{i,t-j}+u_{i,t}$ with $u_{i,t}=x_{i,t}'\beta+\eta_i+\varepsilon_{i,t}$. But $\enVert[0]{x_{i,t}'\beta+\eta_i+\varepsilon_{i,t}}_{\psi_2}\leq \sum_{j=1}^p|\beta_j|\enVert[0]{x_{i,t,k}}_{\psi_2}+\enVert[0]{\eta_i}_{\psi_2}+\enVert[0]{\varepsilon_{i,t}}_{\psi_2}$ which is bounded by the assumptions made.
\end{proof}
 
\subsection{Proof of Theorem \ref{thm probabilistic oracle inequality}}

Before we proceed to the proof of Theorem \ref{thm probabilistic oracle inequality}, we introduce $\eta_i^*=\eta_i1\cbr[0]{|\eta_i|\geq \Xi}$ for some $\Xi>0$. We shall choose $\Xi=\frac{\lambda}{\sqrt{N}T}$. Next, $J_2=\cbr[0]{i:\eta_i^*\neq 0,\ i=1,...,N}$ and $s_2=|J_2|$. Introduce the events
\[\mathcal{A}_{N}=\left\lbrace \|Z'\varepsilon\|_{\infty}\leq \frac{\lambda_{N}}{2},\quad \|D'\varepsilon\|_{\infty}\leq \frac{\lambda_{N}}{2\sqrt{N}}\right\rbrace, \quad \mathcal{B}_{N}=\left\lbrace \kappa^2(\Psi_{N}, s_1,s_2)\geq \frac{\kappa_2^2}{2}\right\rbrace. \]
\begin{lemma}
\label{lemma first step}
On the event $\mathcal{A}_{N}$, the following inequalities are valid
\begin{equation}
\label{eqn handy inequality lemma}
\|\Pi(\hat{\gamma}-\gamma)\|^2+\lambda_{N}\|\hat{\alpha}-\alpha\|_1+\frac{\lambda_N}{\sqrt{N}}\|\hat{\eta}-\eta\|_1\leq 4\lambda_{N}\|\hat{\alpha}_{J_1}-\alpha_{J_1}\|_1+4\frac{\lambda_N}{\sqrt{N}}\|\hat{\eta}_{J_2}-\eta_{J_2}\|_1+4\frac{\lambda_N}{\sqrt{N}}E\Xi^{1-\nu};
\end{equation}
\begin{equation}
\label{eqn precondition of restricted eigenvalue condition}
\|\hat{\alpha}_{J_1^c}-\alpha_{J_1^c}\|_1+\frac{1}{\sqrt{N}}\|\hat{\eta}_{J_2^c}-\eta_{J_2^c}\|_1\leq 3\|\hat{\alpha}_{J_1}-\alpha_{J_1}\|_1+3\frac{1}{\sqrt{N}}\|\hat{\eta}_{J_2}-\eta_{J_2}\|_1+4\frac{1}{\sqrt{N}}E\Xi^{1-\nu}.
\end{equation}
\end{lemma}

\begin{proof}
By the minimizing property of the Lasso,
\[\|  y-\Pi\hat{\gamma}\|^{2}+2\lambda_{N}\|\hat{\alpha}\|_1+2\frac{\lambda_N}{\sqrt{N}}\|\hat{\eta}\|_1\leq \left\|   y-\Pi\gamma\right\|^{2}+2\lambda_{N}\|\alpha\|_1+2\frac{\lambda_N}{\sqrt{N}}\|\eta\|_1\]
such that inserting $y=\Pi\gamma+\epsilon$ yields
\begin{align}
\|\Pi(\hat{\gamma}-\gamma)\|^2\leq 2\varepsilon'\Pi(\hat{\gamma}-\gamma)+2\lambda_{N}(\|\alpha\|_1-\|\hat{\alpha}\|_1)+2\frac{\lambda_N}{\sqrt{N}}(\|\eta\|_1-\|\hat{\eta}\|_1).\label{auxbound1}
\end{align}
Note that on $\mathcal{A}_N$
\begin{align*}
& 2\varepsilon'\Pi(\hat{\gamma}-\gamma) \leq 2\| \varepsilon'Z\|_{\infty}\|\hat{\alpha}-\alpha\|_1+2\| \varepsilon'D\|_{\infty}\|\hat{\eta}-\eta\|_1\leq \lambda_{N}\|\hat{\alpha}-\alpha\|_1+\frac{\lambda_N}{\sqrt{N}}\|\hat{\eta}-\eta\|_1.
\end{align*}
Using this and adding $\lambda_{N}\|\hat{\alpha}-\alpha\|_1+\frac{\lambda_N}{\sqrt{N}}\|\hat{\eta}-\eta\|_1$ to both sides of (\ref{auxbound1}) gives
\begin{align*}
&\|\Pi(\hat{\gamma}-\gamma)\|^2+\lambda_{N}\|\hat{\alpha}-\alpha\|_1+\frac{\lambda_N}{\sqrt{N}}\|\hat{\eta}-\eta\|_1\\ 
& \leq 2\lambda_{N}(\|\alpha\|_1-\|\hat{\alpha}\|_1+\|\hat{\alpha}-\alpha\|_1)+2\frac{\lambda_N}{\sqrt{N}}(\|\eta\|_1-\|\hat{\eta}\|_1+\|\hat{\eta}-\eta\|_1)\\
&\leq 2\lambda_{N}(\|\alpha_{J_1}\|_1-\|\hat{\alpha}_{J_1}\|_1+\|\hat{\alpha}_{J_1}-\alpha_{J_1}\|_1)+2\frac{\lambda_N}{\sqrt{N}}(\|\eta_{J_2}\|_1-\|\hat{\eta}_{J_2}\|_1+\|\hat{\eta}_{J_2}-\eta_{J_2}\|_1+2E\Xi^{1-\nu})\\
&\leq 4\lambda_{N}\|\hat{\alpha}_{J_1}-\alpha_{J_1}\|_1+4\frac{\lambda_N}{\sqrt{N}}\|\hat{\eta}_{J_2}-\eta_{J_2}\|_1+4\frac{\lambda_N}{\sqrt{N}}E\Xi^{1-\nu},
\end{align*}
where the second inequality is due to
\begin{align*}
&\|\eta_{J_2^c}\|_1-\|\hat{\eta}_{J_2^c}\|_1+\|\hat{\eta}_{J_2^c}-\eta_{J_2^c}\|_1\leq 2\|\eta_{J_2^c}\|_1=2\sum_{i=1}^{N}|\eta_i|1\{|\eta_i|<\Xi\}<2\Xi^{1-\nu}\sum_{i=1}^{N}|\eta_i|^{\nu}1\{|\eta_i|<\Xi\}\leq 2E\Xi^{1-\nu}.
\end{align*}
We proved (\ref{eqn handy inequality lemma}). (\ref{eqn precondition of restricted eigenvalue condition}) follows trivially from this. 
\end{proof}

\vspace{0.5cm}

\begin{lemma}[Deterministic oracle inequalities]
\label{lemma deterministic results}

Let Assumption \ref{assu smallest eigenvalue of PsiZ} hold. On the event $\mathcal{A}_{N}\cap \mathcal{B}_{N} $ one has for any positive constant $\lambda_{N}$,
\begin{align*}
&\left\|\Pi(\hat{\gamma}-\gamma) \right\|^2 
\leq \frac{120\lambda^2_{N}s_1}{\kappa_2^2NT}+\del[2]{\frac{120}{\kappa_2^2}+20} \frac{\lambda_N}{\sqrt{N}}E\del[3]{\frac{\lambda_N}{\sqrt{N}T}} ^{1-\nu}
\\
&\|\hat{\alpha}-\alpha\|_1
\leq\frac{120\lambda_{N}s_1}{\kappa_2^2NT}+\del[2]{\frac{120}{\kappa_2^2}+20}\frac{1}{\sqrt{N}}E\del[3]{\frac{\lambda_N}{\sqrt{N}T}}^{1-\nu}\\
&\|\hat{\eta}-\eta\|_1 
\leq\frac{120\lambda_{N}s_1}{\kappa_2^2\sqrt{N}T}+\del[2]{\frac{120}{\kappa_2^2}+20}E\del[3]{\frac{\lambda_N}{\sqrt{N}T}}^{1-\nu}.
\end{align*}
Moreover, the above bounds are valid uniformly over
$\mathcal{F}(s_1,\nu,E):=\cbr[1]{ \alpha \in \mathbb{R}^{p}: \enVert[0]{\alpha}_{0} \leq s_1} \times \cbr[1]{\eta \in \mathbb{R}^{N}: \sum_{i=1}^{N}|\eta_i|^{\nu}\leq E}$.
\end{lemma}

\begin{proof}
By (\ref{eqn handy inequality lemma}) of Lemma \ref{lemma first step}, which is valid on $\mathcal{A}_{N}$,
\begin{equation}
\label{eqn upper bound for prediction error}
\|\Pi(\hat{\gamma}-\gamma)\|^2\leq 4\lambda_{N}\|\hat{\alpha}_{J_1}-\alpha_{J_1}\|_1+4\frac{\lambda_N}{\sqrt{N}}\|\hat{\eta}_{J_2}-\eta_{J_2}\|_1+4\frac{\lambda_N}{\sqrt{N}}E\Xi^{1-\nu}.
\end{equation}
Consider the auxiliary event
\[\mathcal{C}_N:=\cbr[2]{ \frac{1}{\sqrt{N}}E\Xi^{1-\nu}\leq \frac{1}{4}\|\hat{\alpha}_{J_1}-\alpha_{J_1}\|_1+\frac{1}{4\sqrt{N}}\|\hat{\eta}_{J_2}-\eta_{J_2}\|_1} .\] 
On the event $\mathcal{A}_N\cap\mathcal{C}_N$, from (\ref{eqn precondition of restricted eigenvalue condition}) of Lemma \ref{lemma first step}, we have
\begin{equation}
\label{eqn precondition of restricted eigenvalue condition on event Cn}
\|\hat{\alpha}_{J_1^c}-\alpha_{J_1^c}\|_1+\frac{1}{\sqrt{N}}\|\hat{\eta}_{J_2^c}-\eta_{J_2^c}\|_1\leq 4\|\hat{\alpha}_{J_1}-\alpha_{J_1}\|_1+4\frac{1}{\sqrt{N}}\|\hat{\eta}_{J_2}-\eta_{J_2}\|_1.
\end{equation}
In order to apply the compatibility condition, re-parametrise the vector $\delta$ in the definition of the compatibility condition as follows. Let $b^1$ and $b^2$ be $p\times 1$ and $N\times 1$ vectors, respectively, with $b=(b^{1'},b^{2'})'$ defined as
\[\del[3]{ \begin{array}{c}
b^1\\ b^2
 \end{array}}:=\del[3]{ \begin{array}{cc}
\text{I}_{p} & 0\\
0 & \sqrt{N}\text{I}_{N}
 \end{array}}\del[3]{\begin{array}{c}
\delta^1\\ \delta^2
 \end{array}} .\]
Hence, that $\kappa^2(\Psi_N, r_1, r_2)$ is bounded away from zero for integers $r_1 \in\{1,\ldots, p\}$ and $ r_2\in \{1,\ldots, N\}$ is equivalent to
\begin{equation}
\label{eqn equivalent modified compatibility condition}
\kappa^2(\Psi_{N}, r_1, r_2):=\min_{\substack{R_1 \subseteq \{1,\ldots, p\}, |R_1|\leq r_1\\R_2 \subseteq \{1,\ldots,N\}, |R_2|\leq r_2\\R:=R_1\cup R_2}} \min_{\substack{b\in \mathbb{R}^{p+N}\setminus \{0\},\\\|b_{R_1^c}^1\|_1+\frac{1}{\sqrt{N}}\|b_{R_2^c}^2\|_1\\\leq  4\|b_{R_1}^1\|_1+\frac{4}{\sqrt{N}}\|b_{R_2}^2\|_1}}\frac{\|\Pi b\|^2}{\frac{NT}{r_1+r_2}\enVert[3]{\del[3]{ \begin{array}{c}
b_{R_1}^1\\
b_{R_2}^2/\sqrt{N}
\end{array}} }_1^2}>0.
\end{equation}
By (\ref{eqn precondition of restricted eigenvalue condition on event Cn}), our estimator satisfies the constraint of the just introduced version of the compatibility condition and so 
\begin{align*}
\|\Pi(\hat{\gamma}-\gamma)\|^2 &\geq \frac{\kappa^2(\Psi_{N}, s_1,s_2)NT}{s_1+s_2}\enVert[3]{ \del[3]{ \begin{array}{c}
\hat{\alpha}_{J_1}-\alpha_{J_1}\\
(\hat{\eta}_{J_2}-\eta_{J_2})/\sqrt{N}
\end{array}} }_1^2 \\
&\geq \frac{\kappa^2(\Psi_{N}, s_1,s_2)NT}{s_1+s_2}\del[2]{\|\hat{\alpha}_{J_1}-\alpha_{J_1}\|_1^2+\frac{1}{N}\|\hat{\eta}_{J_2}-\eta_{J_2}\|_1^2}\\
&\geq \frac{\kappa_2^2NT}{2(s_1+s_2)}\del[2]{\|\hat{\alpha}_{J_1}-\alpha_{J_1}\|_1^2+\frac{1}{N}\|\hat{\eta}_{J_2}-\eta_{J_2}\|_1^2},
\end{align*}
where the last inequality is valid on $\mathcal{B}_{N}$. Hence, on $\mathcal{A}_N\cap\mathcal{B}_N\cap\mathcal{C}_N$ upon combining with (\ref{eqn upper bound for prediction error}) one has,
\begin{align*}
\frac{\kappa_2^2NT}{2(s_1+s_2)}\del[2]{\|\hat{\alpha}_{J_1}-\alpha_{J_1}\|_1^2+\frac{1}{N}\|\hat{\eta}_{J_2}-\eta_{J_2}\|_1^2} 
&\leq
4\lambda_{N}\|\hat{\alpha}_{J_1}-\alpha_{J_1}\|_1+\frac{4\lambda_{N}}{\sqrt{N}}\|\hat{\eta}_{J_2}-\eta_{J_2}\|_1+\frac{4\lambda_{N}}{\sqrt{N}}E\Xi^{1-\nu}\\
&\leq
5\lambda_{N}\|\hat{\alpha}_{J_1}-\alpha_{J_1}\|_1+\frac{5\lambda_{N}}{\sqrt{N}}\|\hat{\eta}_{J_2}-\eta_{J_2}\|_1,
\end{align*}
which, since $\kappa_2^2>0$ by Assumption \ref{assu smallest eigenvalue of PsiZ}, is equivalent to  
\begin{align*}
\|\hat{\alpha}_{J_1}-\alpha_{J_1}\|_1^2-\frac{10\lambda_{N}(s_1+s_2)}{\kappa_2^2NT}\|\hat{\alpha}_{J_1}-\alpha_{J_1}\|_1+\frac{1}{N}\|\hat{\eta}_{J_2}-\eta_{J_2}\|_1^2-
\frac{10\lambda_{N}(s_1+s_2)}{\kappa_2^2N^{3/2}T}\|\hat{\eta}_{J_2}\|_1\leq 0.
\end{align*}
Let $x=\|\hat{\alpha}_{J_1}-\alpha_{J_1}\|_1$, $y=\|\hat{\eta}_{J_2}-\eta_{J_2}\|_1$, $a=\frac{10\lambda_{N}(s_1+s_2)}{\kappa_2^2NT}$, $b=\frac{1}{N}$ and $c=\frac{10\lambda_{N}(s_1+s_2)}{\kappa_2^2N^{3/2}T}$. Thus one has 
\begin{equation}
\label{quadineq}
x^2-ax+by^2-cy\leq 0.
\end{equation}
First bound $x=\enVert[0]{\hat{\alpha}_{J_1}-\alpha_{J_1}}_1$. For every $y$ the values of $x$ that satisfy the above quadratic inequality form an interval in $\mathbb{R}_+$. The right end point of this interval is the desired upper bound on $x$. Clearly, by the solution formula for the roots of a second degree polynomial, this right end point is a decreasing function in $by^2-cy$. Hence, we first minimize the polynomial $by^2-cy$ to find the largest possible value of $x$ which satisfies (\ref{quadineq}). This yields $y=c/2b$ and the corresponding value of $by^2-cy$ is $-c^2/(4b)$. Hence, our desired upper bound on $x$ is the largest solution of $x^2-ax-\frac{c^2}{4b}\leq 0$. By the standard solution formula for the roots of a quadratic polynomial this yields 
\begin{align}
\enVert[0]{\hat{\alpha}_{J_1}-\alpha_{J_1}}_1=x\leq \frac{a+\sqrt{a^2+c^2/b}}{2}
\leq 
a+\frac{c}{2\sqrt{b}}.
\label{align quadratic solution for x}
\end{align}
Switching the roles of $x$ and $y$, one gets a similar bound on $y=\enVert[0]{\hat{\eta}_{J_2}-\eta_{J_2}}_1$, namely
\begin{align}
\enVert[0]{\hat{\eta}_{J_2}-\eta_{J_2}}_1=y\leq \frac{c+\sqrt{c^2+ba^2}}{2b}
\leq
\frac{c}{b}+\frac{a}{2\sqrt{b}}.
\label{align quadratic solution for y}
\end{align}
Inserting the definitions of $a, b$ and $c$ into (\ref{align quadratic solution for x}) and (\ref{align quadratic solution for y}), we get
\begin{equation}
\label{eqn l1 norm on J1 alpha}
\|\hat{\alpha}_{J_1}-\alpha_{J_1}\|_1\leq \frac{15\lambda_{N}(s_1+s_2)}{\kappa_2^2NT} 
\end{equation}
\begin{equation}
\label{eqn l1 norm on J2 eta}
\|\hat{\eta}_{J_2}-\eta_{J_2}\|_1\leq \frac{15\lambda_{N}(s_1+s_2)}{\kappa_2^2N^{1/2}T} .
\end{equation}
Therefore, on $\mathcal{A}_N\cap\mathcal{B}_N\cap\mathcal{C}_N$, it follows from (\ref{eqn handy inequality lemma}) that
\begin{align*}
&\left\|\Pi(\hat{\gamma}-\gamma) \right\|^2 
\leq
4\lambda_N\|\hat{\alpha}_{J_1}-\alpha_{J_1}\|_1+\frac{4\lambda_N}{\sqrt{N}} \|\hat{\eta}_{J_2}-\eta_{J_2}\|_1+\frac{4\lambda_N}{\sqrt{N}}E\Xi^{1-\nu}
\leq 
\frac{120\lambda^2_{N}(s_1+s_2)}{\kappa_2^2NT}+\frac{4\lambda_N}{\sqrt{N}}E\Xi^{1-\nu}
\\
&\|\hat{\alpha}-\alpha\|_1
\leq 4\|\hat{\alpha}_{J_1}-\alpha_{J_1}\|_1+\frac{4}{\sqrt{N}} \|\hat{\eta}_{J_2}-\eta_{J_2}\|_1+\frac{4}{\sqrt{N}}E\Xi^{1-\nu}
\leq 
\frac{120\lambda_{N}(s_1+s_2)}{\kappa_2^2NT}+\frac{4}{\sqrt{N}}E\Xi^{1-\nu}\\
&\|\hat{\eta}-\eta\|_1 
\leq 4\sqrt{N}\|\hat{\alpha}_{J_1}-\alpha_{J_1}\|_1+4\|\hat{\eta}_{J_2}-\eta_{J_2}\|_1+4E\Xi^{1-\nu}
\leq
\frac{120\lambda_{N}(s_1+s_2)}{\kappa_2^2\sqrt{N}T}+4E\Xi^{1-\nu}.
\end{align*}

On $\mathcal{A}_N\cap\mathcal{C}_N^c$ one has trivial oracle inequalities via (\ref{eqn handy inequality lemma}) of Lemma \ref{lemma first step}. To be precise,
\[\|\Pi(\hat{\gamma}-\gamma)\|^2< 20\lambda_N\frac{E\Xi^{1-\nu}}{\sqrt{N}},\quad \|\hat{\alpha}-\alpha\|_1< 20\frac{E\Xi^{1-\nu}}{\sqrt{N}},\quad \|\hat{\eta}-\eta\|_1< 20E\Xi^{1-\nu}.\]
These inequalities are valid on event $\mathcal{A}_N\cap \mathcal{B}_N \cap \mathcal{C}_N^c$ too. Synchronising constants, using that $(\mathcal{A}_N\cap \mathcal{B}_N \cap \mathcal{C}_N^c)\cup(\mathcal{A}_N\cap \mathcal{B}_N \cap \mathcal{C}_N)=\mathcal{A}_N\cap \mathcal{B}_N$, and recognising that
\[s_2:=\sum_{i=1}^{N}1\{|\eta_i|\geq \Xi\}=\sum_{i=1}^{N}1\{|\eta_i|^{\nu}\geq \Xi^{\nu}\}\leq E\Xi^{-\nu},\]
We arrive at
\begin{align*}
&\left\|\Pi(\hat{\gamma}-\gamma) \right\|^2 
\leq 
\frac{120\lambda^2_{N}(s_1+E\Xi^{-\nu})}{\kappa_2^2NT}+\frac{20\lambda_N}{\sqrt{N}}E\Xi^{1-\nu}
\\
&\|\hat{\alpha}-\alpha\|_1
\leq 
\frac{120\lambda_{N}(s_1+E\Xi^{-\nu})}{\kappa_2^2NT}+\frac{20}{\sqrt{N}}E\Xi^{1-\nu}\\
&\|\hat{\eta}-\eta\|_1 
\leq
\frac{120\lambda_{N}(s_1+E\Xi^{-\nu})}{\kappa_2^2\sqrt{N}T}+20E\Xi^{1-\nu}.
\end{align*}
The deterministic oracle inequalities follow upon choosing $\Xi=\frac{\lambda_N}{\sqrt{N}T}$.

To see the uniformity $\mathcal{F}(s_1,\nu,E)$, note that only properties $s_1$, $\nu$ and $E$ characterizing $\alpha$ and $\eta$ enter the deterministic oracle inequalities. Hence, the deterministic oracle inequalities are uniform over the set $\mathcal{F}(s_1,\nu,E)$. 
\end{proof}

\vspace{0.5cm}

For the proof of Lemma \ref{lemma lower bound on Ant} below, we shall use Orlicz norms as defined in \cite{vandervaartWellner1996}: Let $\psi$ be a non-decreasing, convex function with $\psi(0)=0$. Then, the Orlicz norm of a random variable $X$ is given by
\begin{align*}
\enVert{X}_\psi=\inf\left\{C>0:\mathbb{E}\psi\left(|X|/C\right)\leq 1\right\},
\end{align*}
where, as usual, $\inf \emptyset =\infty$. We will use Orlicz norms for  $\psi(x)=\psi_b(x)=e^{x^b}-1$ for various values of $b$. The following Lemma provides a lower bound on the probability of $\mathcal{A}_N$.

\begin{lemma}
\label{lemma lower bound on Ant}
Let $\lambda_{N}=\sqrt{4MNT(\log (p\vee N))^3}$ for some $M>0$. By Assumptions \ref{assu panel data} and \ref{assu subgaussian}, we have 
\[\mathbb{P}(\mathcal{A}_{N})\geq 1-Ap^{1-BM^{1/3}}-AN^{1-BM^{1/3}},\]
for positive constants $A$ and $B$.
\end{lemma}

\begin{proof}
Consider the event $\{\|Z'\varepsilon\|_{\infty}> \lambda_{N}/2\}$ first. To this end, let $z_{j,l}$ denote the $j$th entry of the $l$th column of $Z$, i.e. the $j$th entry of $(z_{1,1,l}, z_{1,2,l},\ldots,z_{1,T,l},z_{2,1,l},\ldots,z_{N,T,l})'$. Similarly, we write $\varepsilon_j$ for the $j$th entry of $\varepsilon$. Now note that $j\mapsto(\lceil\frac{j}{T}\rceil,j-\lfloor\frac{j}{T}\rfloor T)$ is a bijection from $\cbr{1,...,NT}$ to $\cbr{1,...,N}\times \cbr{1,...,T}$ where $\lfloor x\rfloor$ denotes the greatest integer strictly less than $x$ and $\lceil x\rceil$ the smallest integer greater than or equal to $x\in \mathbb{R}$. In case the $l$th column of $Z$ corresponds one of the lags of the left hand side variable, assume for concreteness the $k$th lag, define $\mathcal{F}_{n}=\sigma\del[1]{y_{\lceil\frac{j}{T}\rceil,j-\lfloor\frac{j}{T}\rfloor T},...,y_{\lceil\frac{j}{T}\rceil,j-\lfloor\frac{j}{T}\rfloor T-L}, \varepsilon_{\lceil\frac{j}{T}\rceil,j-\lfloor\frac{j}{T}\rfloor T}, 1\leq j\leq n}$ and $S_{n,l}=\sum_{j=1}^nz_{j,l}\varepsilon_j=\sum_{j=1}^ny_{\lceil\frac{j}{T}\rceil,j-\lfloor\frac{j}{T}\rfloor T-k}\varepsilon_{\lceil\frac{j}{T}\rceil,j-\lfloor\frac{j}{T}\rfloor T}$. Thus,
\begin{align*}
\mathbb{E}[S_{n,l}|\mathcal{F}_{n-1}]
&=
\sum_{j=1}^{n-1}y_{\lceil\frac{j}{T}\rceil,j-\lfloor\frac{j}{T}\rfloor T-k}\varepsilon_{\lceil\frac{j}{T}\rceil,j-\lfloor\frac{j}{T}\rfloor T}+\mathbb{E}\sbr[1]{y_{\lceil\frac{n}{T}\rceil,j-\lfloor\frac{n}{T}\rfloor T-k}\varepsilon_{\lceil\frac{n}{T}\rceil,n-\lfloor\frac{n}{T}\rfloor T}|\mathcal{F}_{n-1}}\\
&=
S_{n-1,l}+y_{\lceil\frac{n}{T}\rceil,j-\lfloor\frac{n}{T}\rfloor T-k}\mathbb{E}\sbr[1]{\varepsilon_{\lceil\frac{n}{T}\rceil,n-\lfloor\frac{n}{T}\rfloor T}|\mathcal{F}_{n-1}}.
\end{align*}
Using that $\del[1]{\lceil\frac{n}{T}\rceil,n-\lfloor\frac{n}{T}\rfloor T}$ is a unique pair $(i,t)\in\cbr{1,\ldots,N}\times\cbr{1,\ldots,T}$ we have that 
\begin{align*}
\mathbb{E}\sbr[1]{\varepsilon_{\lceil\frac{n}{T}\rceil,n-\lfloor\frac{n}{T}\rfloor T}|\mathcal{F}_{n-1}}
=
\mathbb{E}[\varepsilon_{i,t}|\mathcal{F}_{n-1}]
= 
\mathbb{E}[\varepsilon_{i,t}|\sigma(y_{i,s},\ldots,y_{i,1-L}, \varepsilon_{i,s},\ldots,\varepsilon_{i,1},1\leq s\leq t-1)]
\end{align*}
\footnote{For $t=1$, the last expression in the above display is to be read as absence of conditioning on the error terms.}where the last equality follows from the assumption of independence across $1\leq i\leq N$ (Assumption \ref{assu panel data}). By Assumption \ref{assu panel data}, this conditional expectation equals zero as the $\varepsilon_{i,s}$ are linear functions of $y_{i,s},\ldots,y_{i,s-L}$ and $x_{i,s}$. Thus, $S_{n,l}$ is a martingale with mean zero (the increments are martingale differences by the above argument). A similar argument applies when the $l$th column of $Z$  equals $\cbr[0]{x_{1,1,k},..., x_{1,T,k},x_{2,1,k},...,x_{N,T,k}}'$ for some $1\leq k\leq p_x$ such that every row of $Z'\varepsilon$ is a zero mean martingale.

Next, note that by Assumption \ref{assu subgaussian}, for all $1\leq j\leq NT$, $1\leq l\leq p$ and $\epsilon>0$, one has
\[\mathbb{P}(|z_{j,l}\varepsilon_{j}|\geq\epsilon)\leq \mathbb{P}(|z_{j,l}|\geq \sqrt{\epsilon})+\mathbb{P}(|\varepsilon_{j}|\geq \sqrt{\epsilon})\leq Ke^{-C\epsilon}. \]
It follows from Lemma 2.2.1 in \cite{vandervaartWellner1996} that $\|z_{j,l}\varepsilon_{j}\|_{\psi_1}\leq (1+K)/C$. Then, by the definition of the Orlicz norm, $\mathbb{E}\sbr[1]{ e^{C/(1+K)|z_{j,l}\varepsilon_{j}|} }\leq 2$. Now use Proposition \ref{prop adapation of Fan} in Appendix B with $D=C/(1+K)$, $\alpha=1/3$ and $C_1=2$ to conclude
\begin{align*}
\mathbb{P}\del[2]{\|Z'\varepsilon\|_{\infty}> \frac{\lambda_{N}}{2}}
\leq 
 \sum_{l=1}^{p}\mathbb{P}\del[3]{ \envert[3]{\sum_{j=1}^{NT}z_{j,l}\varepsilon_{j} }> \frac{\lambda_{N}}{2NT}NT }
=
pAe^{-B\log (p\vee N)M^{1/3}}\leq Ap^{1-BM^{1/3}}.
\end{align*}
Note also  that the upper bound of the preceding probability becomes arbitrarily small for sufficiently large $N$ and $M$ such that we also conclude
\begin{equation}
\label{eqn rates for Zepsilon infinity}
\|Z'\varepsilon\|_{\infty}=O_p(\lambda_{N}).
\end{equation}

Next, consider the event $\{\|D'\varepsilon\|_{\infty}> \lambda_{N}/(2\sqrt{N})\}$. Using Assumption \ref{assu panel data} a small calculation shows that all entries of $D'\varepsilon$ are zero mean martingales with respect to the natural filtration. As above, Assumption \ref{assu subgaussian} and Lemma 2.2.1 in \cite{vandervaartWellner1996} yield $\|\varepsilon_{i,t}\|_{\psi_2}\leq \del[1]{ \frac{1+K/2}{C}}^{1/2}$ such that by the second to last inequality on page 95 in \cite{vandervaartWellner1996} one has $\|\varepsilon_{i,t}\|_{\psi_1}\leq \|\varepsilon_{i,t}\|_{\psi_2}(\log2)^{-1/2}\leq  \del[1]{ \frac{1+K/2}{C}} ^{1/2}(\log2)^{-1/2}$ for all $i$ and $t$. Then using the definition of the Orlicz norm, $\mathbb{E}\sbr[1]{ \exp  \del[1]{  \del[1]{ \frac{C}{1+K/2}} ^{1/2}(\log2)^{1/2}|\varepsilon_{i,t}| } }\leq 2$ and Proposition \ref{prop adapation of Fan} in Appendix B with $D=\del[1]{\frac{C}{1+K/2}}^{1/2}(\log2)^{1/2}$, $\alpha=1/3$ and $C_1=2$ implies
\begin{align*}
\mathbb{P}\del[2]{\|D'\varepsilon\|_{\infty}> \frac{\lambda_{N}}{2\sqrt{N}} }
\leq
\sum_{i=1}^{N}\mathbb{P}\del[2]{ \envert[2]{\sum_{t=1}^{T}\varepsilon_{i,t} } > \frac{\lambda_N}{2\sqrt{N}T}T } \leq ANe^{-B(\log (p\vee N)M^{1/3}}
\leq
 AN^{1-BM^{1/3}}.
\end{align*}
Note also that the upper bound of the preceding probability becomes arbitrarily small for sufficiently large $N$ and $M$, such that we may also conclude
\begin{equation}
\label{eqn rates for Depsilon infinity}
\|D'\varepsilon\|_{\infty}=O_p\left( \frac{\lambda_{N}}{\sqrt{N}}\right).
\end{equation}
\end{proof}

\vspace{0.5cm}

The following lemma shows that $\kappa^2(\Psi_{N}, s_1, s_2)$ and $\kappa^2(\Psi, s_1, s_2)$ are close if $\Psi_{N}$ and $\Psi$ are in some sense close.

\begin{lemma}
\label{lemma difference between two matrices}
Let $A$ and $B$ be two positive semidefinite $(p+N)\times(p+N)$ matrices and $\delta:=\max_{1\leq i,j \leq p+N}|A_{ij}-B_{ij}|$. For any integers $r_1 \in\{1,\ldots, p\}$ and $r_2\in \{1,\ldots, N\}$, one has
\[\kappa^2(B, r_1,r_2)\geq \kappa^2(A, r_1,r_2) -\delta 25(r_1+r_2).\]
\end{lemma}

\begin{proof}
Let $x$ be a $(p+N)\times 1$ non-zero vector, satisfying $\|x_{R^c}\|_1\leq 4\|x_{R}\|_1$ for $R=R_1\cup(R_2+p)$ where $R_1\subseteq\cbr[0]{1,...,p}$ with $|R_1|\leq r_1$, and $R_2\subseteq\cbr[0]{1,...,N}$ with $|R_2|\leq r_2$. Now,
\begin{align*}
&|x'Ax-x'Bx|=|x'(A-B)x|\leq \|x\|_1\|(A-B)x\|_{\infty}\leq \|x\|_1^2\delta =\delta\left( \|x_R\|_1+\|x_{R^c}\|_1\right) ^2\\
&\leq \delta \left( \|x_R\|_1+4\|x_R\|_1\right)^2\leq \delta 25\|x_R\|_1^2.
\end{align*}
Hence,
\[\frac{x'Bx}{\frac{1}{r_1+r_2}\|x_R\|_1^2}\geq \frac{x'Ax}{\frac{1}{r_1+r_2}\|x_R\|_1^2}-\delta 25(r_1+r_2)\geq\kappa^2(A,r_1,r_2) -\delta 25(r_1+r_2),\]
where the last inequality is true because of the definition of $\kappa^2(A,r_1,r_2)$. Minimising the left-hand side over non-zero $x$ satisfying $\|x_{R^c}\|_1\leq 4\|x_{R}\|_1$ yields the claim.
\end{proof}

\vspace{0.5cm}

Define
\[\tilde{\mathcal{B}}_{N}=\cbr[4]{ \max_{1\leq i,j\leq p+N}\left| \Psi_{N,ij}-\Psi_{ij}\right| \leq \frac{\kappa_2^2(\Psi_Z, s_1)}{50\sbr[2]{s_1+E\del[2]{\frac{\lambda_N}{\sqrt{N}T}}^{-\nu}}}}. \]
Setting $A=\Psi$, $B=\Psi_{N}$ it follows from Lemma \ref{lemma difference between two matrices} that $ \tilde{\mathcal{B}}_{N} \subseteq\mathcal{B}_{N}$ as $\kappa_2^2(\Psi_Z, s_1)\leq \kappa^2(\Psi, s_1, s_2)$ for all $s_2\in\cbr[0]{1,...,N}$ as argued prior to Assumption \ref{assu smallest eigenvalue of PsiZ}. Thus, we just need to find a lower bound on $\mathbb{P}(\tilde{\mathcal{B}}_{N})$ in order to prove Theorem \ref{thm probabilistic oracle inequality}.

\begin{lemma}
\label{lemma lower bound on tildeBnt} 
Let Assumptions \ref{assu panel data}, \ref{assu smallest eigenvalue of PsiZ} and \ref{assu subgaussian} hold. Assume that $s_1+E\del[2]{\frac{\lambda_N}{\sqrt{N}T}}^{-\nu}\lesssim \sqrt{N}$. Then, there exist positive constants $A, B$ such that 
\[\mathbb{P}(\mathcal{B}_N^c)\leq\mathbb{P}(\tilde{\mathcal{B}}_{N}^c)\leq A(p^2+pN)\exp\del[4]{{-B\cbr[3]{ N/\sbr[2]{s_1+E\del[2]{\frac{\lambda_N}{\sqrt{N}T}}^{-\nu}}^2} ^{1/3}}}.\]
\end{lemma}

\begin{proof}
Since the lower right $N\times N$ blocks of $\Psi_{N}$ and $\Psi$ are identical, it suffices to bound the entries of $\frac{1}{NT}Z'Z-\frac{1}{NT}\mathbb{E}[Z'Z]$ and $\frac{1}{T\sqrt{N}}Z'D$. A typical element of $\frac{1}{NT}Z'Z-\frac{1}{NT}\mathbb{E}[Z'Z]$ is of the form $\frac{1}{NT}\sum_{i=1}^{N}\sum_{t=1}^{T}( z_{i,t,l}z_{i,t,k}-\mathbb{E}[z_{i,t,l}z_{i,t,k}])$ for some $l,k\in \{1,\ldots,p\}$. By Assumption \ref{assu subgaussian} we have for every $\epsilon>0$
\[\mathbb{P}(|z_{i,t,l}z_{i,t,k}|\geq\epsilon)\leq \mathbb{P}(|z_{i,t,l}|\geq \sqrt{\epsilon})+\mathbb{P}(|z_{i,t,k}|\geq \sqrt{\epsilon})\leq Ke^{-C\epsilon}. \]
It follows from Lemma 2.2.1 in \cite{vandervaartWellner1996} that $\|z_{i,t,l}z_{i,t,k}\|_{\psi_1}\leq (1+K)/C$. Hence, by subadditivity of the Orlicz norm and Jensen's inequality
\begin{align*}
\enVert[4]{ \frac{1}{T}\sum_{t=1}^{T}\left( z_{i,t,l}z_{i,t,k}-\mathbb{E}[z_{i,t,l}z_{i,t,k}]\right)}_{\psi_1}
\leq
2\max_{1\leq t\leq T}\|z_{i,t,l}z_{i,t,k}\|_{\psi_1}\leq \frac{2(1+K)}{C}.
\end{align*}
Thus, by the definition of the Orlicz norm, $\mathbb{E} \exp \del[1]{\frac{C}{2(1+K)}\envert[1]{ \frac{1}{T}\sum_{t=1}^{T}( z_{i,t,l}z_{i,t,k}-\mathbb{E}[z_{i,t,l}z_{i,t,k}])}}\leq 2$. Using independence across $i$ (Assumption \ref{assu panel data}) to invoke Proposition \ref{prop adapation of Fan} in Appendix B with $D=\frac{C}{2(1+K)}$, $\alpha=1/3$ and $C_1=2$ such that for every $x\gtrsim \frac{1}{\sqrt{N}}$
\begin{equation}
\label{eqn block comparison 1}
\mathbb{P}\del[2]{  \envert[2]{ \sum_{i=1}^{N}\frac{1}{T}\sum_{t=1}^{T}( z_{i,t,l}z_{i,t,k}-\mathbb{E}[z_{i,t,l}z_{i,t,k}])}\geq Nx } \leq Ae^{-B(x^2N)^{1/3}},
\end{equation}
for positive constants $A$ and $B$.

Next, consider $\frac{1}{T\sqrt{N}}Z'D$. A typical element can be written as $\frac{1}{\sqrt{N}T}\sum_{t=1}^{T}z_{i,t,l}$ for some $i\in \{1,\ldots,N\}$ and $l\in \{1,\ldots,p\}$. By Assumption \ref{assu subgaussian}, we have $\mathbb{P}(|z_{i,t,l}|\geq \epsilon)\leq \frac{1}{2}Ke^{-C\epsilon^2}$ for all $\epsilon>0$ and it follows from Lemma 2.2.1 in \cite{vandervaartWellner1996} that $\|z_{i,t,l}\|_{\psi_2}\leq \del [1]{  \frac{1+K/2}{C}}^{1/2}$. Hence,
\begin{equation*}
\label{align M over squart root N}
\enVert[4]{ \frac{1}{\sqrt{N}T}\sum_{t=1}^{T}z_{i,t,l}}_{\psi_2}
\leq
\frac{1}{\sqrt{N}}\max_{1\leq t \leq T}\| z_{i,t,l}\|_{\psi_2}\leq \frac{1}{\sqrt{N}}\del[2]{ \frac{1+K/2}{C}} ^{1/2}=:\frac{C'}{\sqrt{N}}.
\end{equation*}
Thus, it follows by Markov's inequality, positivity and increasingness of $\psi_2(x)$, as well as $1\wedge \psi_2(x)^{-1}=1\wedge (e^{x^2}-1)^{-1}\leq 2e^{-x^2}$ that for any $x>0$
\begin{align}
\mathbb{P}\del [2]{ \envert[2]{ \frac{1}{\sqrt{N}T}\sum_{t=1}^{T}z_{i,t,l}} >x}
\leq 
1\wedge \frac{1}{e^{(x\sqrt{N}/C')^2}-1} 
 \leq
 2e^{-\frac{Nx^2}{C'^2}}\leq Ae^{-Bx^2N}\label{align block comparison 2},
\end{align}
where the last estimate follows by choosing $A$ and $B$ sufficiently large/small for (\ref{eqn block comparison 1}) and (\ref{align block comparison 2}) both to be valid. Setting $x=\frac{\kappa^2}{50\sbr[2]{s_1+E\del[2]{\frac{\lambda_N}{\sqrt{N}T}}^{-\nu}} }=\frac{\kappa^2}{50}\frac{1}{s_1+E\del[2]{\frac{\lambda_N}{\sqrt{N}T}}^{-\nu}}$, using that $\frac{1}{s_1+E\del[2]{\frac{\lambda_N}{\sqrt{N}T}}^{-\nu}}\gtrsim \frac{1}{\sqrt{N}}$ and $\kappa^2$ being bounded away from 0 (Assumption \ref{assu smallest eigenvalue of PsiZ}), we have
\begin{align*}
&\mathbb{P}(\mathcal{B}_{N}^c)
\leq
\mathbb{P}( \tilde{\mathcal{B}}_{N}^c)
=
\mathbb{P}\del[2]{\max_{1\leq i,j\leq p+N}| \Psi_{N,ij}-\Psi_{ij}| > x}\\
&\leq
 A(p^2+pN)\sbr[4]{\exp\del[4]{-B\cbr[4]{\sbr[4]{\frac{\kappa^2/50}{s_1+E\del[2]{\frac{\lambda_N}{\sqrt{N}T}}^{-\nu}}}^2N}^{1/3}}\vee  \exp\del[4] {-B{\sbr[4]{\frac{\kappa^2/50}{s_1+E\del[2]{\frac{\lambda_N}{\sqrt{N}T}}^{-\nu}}}^2N}}}\\
&\leq
A (p^2+pN)\exp\del[4]{{-B\cbr[3]{ N/\sbr[2]{s_1+E\del[2]{\frac{\lambda_N}{\sqrt{N}T}}^{-\nu}}^2} ^{1/3}}}
\end{align*}
where the last estimate has merged $(\kappa^2/50)^{2/3}$ into $B$.
\end{proof}

\vspace{0.5cm}

\begin{proof}[Proof of Theorem \ref{thm probabilistic oracle inequality}]
Theorem 1 follows by combining Lemmas \ref{lemma deterministic results}, \ref{lemma lower bound on Ant}, and  \ref{lemma lower bound on tildeBnt}.
\end{proof}

\vspace{0.5cm}

\begin{corollary}
\label{thm l1 norm consistency}
Let the conditions of Theorem \ref{thm probabilistic oracle inequality} hold. For large enough $M>0$ and assuming $\frac{(\log (p\vee N))^3\sbr[2]{s_1+E\del[2]{\frac{\lambda_N}{\sqrt{N}T}}^{-\nu}}^2}{N}=o(1)$, we have the following stochastic orders valid uniformly over
$\mathcal{F}(s_1,\nu,E)$. 
\[
\frac{1}{NT}\left\|\Pi(\hat{\gamma}-\gamma) \right\|^2 = O_p\del[3]{ s_1\del[3]{ \frac{\lambda_N}{NT}}^2}  + O_p\del[3]{ \frac{\lambda_N}{\sqrt{N}NT}E\del[2]{\frac{\lambda_N}{\sqrt{N}T}}^{1-\nu}},\]
\[\|\hat{\alpha}-\alpha\|_1=O_p\del[3]{  s_1\frac{\lambda_N}{NT}}+ O_p\del[3]{  \frac{1}{\sqrt{N}}E\del[2]{\frac{\lambda_N}{\sqrt{N}T}}^{1-\nu}},\]
\[\|\hat{\eta}-\eta\|_1 =O_p\del[3]{s_1 \frac{\lambda_N}{\sqrt{N}T}}  + O_p\del[3]{  E\del[2]{\frac{\lambda_N}{\sqrt{N}T}}^{1-\nu}}.\]
\end{corollary}

\begin{proof}[Proof of Corollary \ref{thm l1 norm consistency}]
Given positive constants $A$ and $B$, $Ap^{1-BM^{1/3}}$ and $AN^{1-BM^{1/3}}$ become arbitrarily small for large enough $M>0$. By $\frac{(\log (p\vee N))^3\sbr[2]{s_1+E\del[2]{\frac{\lambda_N}{\sqrt{N}T}}^{-\nu}}^2}{N}=o(1)$, $A(p^2+pN)\exp\del[4]{{-B\cbr[3]{ N/\sbr[2]{s_1+E\del[2]{\frac{\lambda_N}{\sqrt{N}T}}^{-\nu}}^2} ^{1/3}}}\rightarrow 0$ as $N \rightarrow \infty$. Thus the lower bound on the probability in Theorem  \ref{thm probabilistic oracle inequality} goes to one as $N \rightarrow \infty$ for large enough $M>0$ and the conclusion follows from Theorem  \ref{thm probabilistic oracle inequality}.
\end{proof}

\subsection{Proof of Lemma \ref{lemma step stone to final theorem}}

The following lemma gives the rates of the uniform prediction and estimation errors for nodewise regression. 

\begin{lemma}
\label{lemma big O p oracle inequalities for nodewise regression}
Let Assumptions \ref{assu panel data}, \ref{assu subgaussian} and \ref{assu more on eigen values} hold. Let $\lambda_{node}=\sqrt{16M(\log p)^3/N}$ for some $M>0$. For $M$ sufficiently large, we have
\begin{align}
\label{eqn nodewise regression oracle inequality prediction}
\max_{j\in H_1}\frac{1}{NT}\|Z_{-j}(\hat{\phi}_j-\phi_j)\|^2&=O_p\left( \bar{G}\lambda_{node}^{2-\vartheta}\right)\\
\label{eqn nodewise regression oracle inequality estimation}
\max_{j\in H_1}\|\hat{\phi}_{j}-\phi_{j}\|_1&=O_p\left( \bar{G}\lambda_{node}^{1-\vartheta}\right) \\
\max_{j\in H_1}\frac{1}{NT}\|Z_{-j}'\zeta_j\|_{\infty}&=O_p(\lambda_{node})\label{eqn order for event D}.
\end{align}
\end{lemma}

\begin{proof}
We say that a $(p-1)\times (p-1) $ matrix $A$ satisfies the \textit{compatibility condition} $CC(r)$ for some integer $r \in\{1,\ldots, p-1\}$ if
\begin{equation*}
\kappa^2\left(A, r\right) :=\min_{\substack{R \subseteq \{1,\ldots, p-1\}\\ |R|\leq r}} \min_{\substack{\delta\in \mathbb{R}^{p-1}\setminus  \{0\}\\\|\delta_{R^c}\|_1\leq 3\|\delta_R\|_1}}\frac{\delta'A \delta}{\frac{1}{r}\|\delta_R\|_1^2}>0.
\end{equation*}
Define a $(p-1)\times 1$ vector $\phi_j^*$ such that
\[\phi_{j,k}^*:=\phi_{j,k}1\{|\phi_{j,k}|\geq \lambda_{node}\}\qquad k=1,\ldots,p-1\]
and its active set $J_j^*$ as well as its sparsity index $s_j^*$
\[J_j^*:=\{k: \phi_{j,k}^*\neq 0, k=1,\ldots,p-1\}\qquad 1\leq s_j^*:=|J_j^*|\leq p-1.\]
Consider the events
\[\mathcal{D}_N=\cbr[2]{ \max_{j\in H_1}\frac{1}{NT}\|Z_{-j}'\zeta_j\|_{\infty}\leq \frac{\lambda_{node}}{4}},\]
\[ \mathcal{E}_{N,j}=\cbr[3]{ \kappa^2\del[2]{ \frac{1}{NT}Z_{-j}'Z_{-j}, s_j^*}\geq \frac{\kappa^2( \Psi_{Z,-j,-j}, s_j^*)}{2}},\]
and
\[\mathcal{F}_N=\cbr[3]{\enVert[3]{\frac{1}{NT}Z'Z-\Psi_{Z}}_{\infty}\leq \lambda_{node}}.\]
Using the same technique as in Section 6.2.3 of \cite{vandegeer2011}, we arrive at the following oracle inequality, which is almost the same as the one on the top of p111 of \cite{vandegeer2011}: for each $j\in H_1$, on $\mathcal{D}_N\cap \mathcal{E}_{N,j}$
\begin{align}
\frac{1}{NT}\|Z_{-j}(\hat{\phi}_j-\phi_j)\|^2+\lambda_{node}\|\hat{\phi}_j-\phi_j\|_1&\leq \frac{3}{NT}\|Z_{-j}(\phi_j^*-\phi_j)\|^2+\frac{48\lambda_{node}^2s_j^*}{\kappa^2(\frac{1}{NT}Z_{-j}'Z_{-j}, s_j^*)}+\lambda_{node}\|\phi_j^*-\phi_j\|_1\notag\\
&\leq \frac{3}{NT}\|Z_{-j}(\phi_j^*-\phi_j)\|^2+\frac{96\lambda_{node}^2s_j^*}{\kappa^2(\Psi_{Z}, s_j^*)}+\lambda_{node}\|\phi_j^*-\phi_j\|_1\label{align three terms p111 van de geer 2011}
\end{align}
where the second inequality is due to event  $\mathcal{E}_{N,j}$ and that $\kappa^2\left( \Psi_{Z,-j,-j}, r\right)\geq\kappa^2\left( \Psi_{Z}, r\right)$ for all $j=1,...,p$ and $r=1,...,p-1$. 

We now bound the three terms on the right hand side of (\ref{align three terms p111 van de geer 2011}). Let $b_j:=\phi_j^*-\phi_j$.
\begin{align*}
&\frac{1}{NT}\|Z_{-j}(\phi_j^*-\phi_j)\|^2=b_j'\Psi_{Z,-j,-j}b_j+b_j'\del[3]{\frac{1}{NT}Z_{-j}'Z_{-j}-\Psi_{Z,-j,-j}}b_j\\
&\leq \text{maxeval}(\Psi_{Z,-j,-j})\|b_j\|^2+\enVert[3]{\frac{1}{NT}Z_{-j}'Z_{-j}-\Psi_{Z,-j,-j}}_{\infty}\|b_j\|_1^2\leq \text{maxeval}(\Psi_{Z})\|b_j\|^2+\lambda_{node}\|b_j\|_1^2
\end{align*}
where the last inequality holds on event $\mathcal{F}_N$. Note that
\[\|b_j\|^2=\sum_{k=1}^{p-1}|\phi_{j,k}|^21\{|\phi_{j,k}|< \lambda_{node}\}\leq \lambda_{node}^{2-\vartheta}\sum_{k=1}^{p-1}|\phi_{j,k}|^{\vartheta}1\{|\phi_{j,k}|< \lambda_{node}\}\leq G_j\lambda_{node}^{2-\vartheta}.\]
\begin{equation}
\label{eqn bj l1 norm}
\|b_j\|_1=\sum_{k=1}^{p-1}|\phi_{j,k}|1\{|\phi_{j,k}|< \lambda_{node}\}\leq \lambda_{node}^{1-\vartheta}\sum_{k=1}^{p-1}|\phi_{j,k}|^{\vartheta}1\{|\phi_{j,k}|< \lambda_{node}\}\leq G_j\lambda_{node}^{1-\vartheta}.
\end{equation}
\begin{equation}
\label{eqn sj star rate}
1\leq s_j^*=\sum_{k=1}^{p-1}1\{|\phi_{j,k}|\geq \lambda_{node}\}=\sum_{k=1}^{p-1}1\{|\phi_{j,k}|^{\vartheta}\geq \lambda_{node}^{\vartheta}\}\leq G_j\lambda_{node}^{-\vartheta}.
\end{equation}
Thus, for each $j\in H_1$, on $ \mathcal{D}_N\cap \mathcal{E}_{N,j}\cap\mathcal{F}_N$
\begin{align*}
&\frac{1}{NT}\|Z_{-j}(\hat{\phi}_j-\phi_j)\|^2+\lambda_{node}\|\hat{\phi}_j-\phi_j\|_1\\
&\leq \text{maxeval}(\Psi_{Z})G_j\lambda_{node}^{2-\vartheta}+G_j^2\lambda_{node}^{3-2\vartheta}+\frac{96}{\kappa^2(\Psi_{Z}, s_j^*)}G_j\lambda_{node}^{2-\vartheta}+G_j\lambda_{node}^{2-\vartheta}\\
&=\del[3]{\text{maxeval}(\Psi_{Z})+\frac{96}{\kappa^2(\Psi_{Z}, s_j^*)}+1}G_j\lambda_{node}^{2-\vartheta}+G_j^2\lambda_{node}^{3-2\vartheta}
\end{align*}
from where we can extract two oracle inequalities
\[\frac{1}{NT}\|Z_{-j}(\hat{\phi}_j-\phi_j)\|^2\leq \del[3]{\text{maxeval}(\Psi_{Z})+\frac{96}{\kappa^2(\Psi_{Z}, s_j^*)}+1}G_j\lambda_{node}^{2-\vartheta}+G_j^2\lambda_{node}^{3-2\vartheta},\]
\[\|\hat{\phi}_j-\phi_j\|_1\leq \del[3]{\text{maxeval}(\Psi_{Z})+\frac{96}{\kappa^2(\Psi_{Z}, s_j^*)}+1}G_j\lambda_{node}^{1-\vartheta}+G_j^2\lambda_{node}^{2-2\vartheta}.\]
As the oracle inequalities in the above display are valid simultaneously on $\mathcal{D}_N\cap (\cap_{j\in H_1} \mathcal{E}_{N,j})\cap\mathcal{F}_N$ we conclude that
\[
\max_{j\in H_1}\frac{1}{NT}\|Z_{-j}(\hat{\phi}_j-\phi_j)\|^2\leq \del[3]{\text{maxeval}(\Psi_{Z})+\frac{96}{\min_{j\in H_1}\kappa^2(\Psi_{Z}, s_j^*)}+1}\bar{G}\lambda_{node}^{2-\vartheta}+\bar{G}^2\lambda_{node}^{3-2\vartheta},\]
\begin{align}
\max_{j\in H_1}\|\hat{\phi}_j-\phi_j\|_1\leq \del[3]{\text{maxeval}(\Psi_{Z})+\frac{96}{\min_{j\in H_1}\kappa^2(\Psi_{Z}, s_j^*)}+1}\bar{G}\lambda_{node}^{1-\vartheta}+\bar{G}^2\lambda_{node}^{2-2\vartheta},\label{prerate}
\end{align}
on  $ \mathcal{D}_N\cap (\cap_{j\in H_1} \mathcal{E}_{N,j})\cap\mathcal{F}_N$. 

Next, we establish a lower bound on the probability of $\mathcal{D}_N\cap (\cap_{j\in H_1} \mathcal{E}_{N,j})\cap\mathcal{F}_N$. Consider $\mathcal{D}_N$ first. A typical element of $Z_{-j}'\zeta_j$ is of the form $\sum_{i=1}^{N}\sum_{t=1}^{T}z_{i,t,l}\zeta_{j,i,t}$ for some $l\neq j$. By (\ref{eqn regression error by construction}), one has $\frac{1}{NT}\sum_{i=1}^{N}\sum_{t=1}^{T}z_{i,t,l}\zeta_{j,i,t}=\frac{1}{NT}\sum_{i=1}^{N}\sum_{t=1}^{T}( z_{i,t,l}\zeta_{j,i,t}-\mathbb{E}[z_{i,t,l}\zeta_{j,i,t}])$ for $l\neq j$. By Assumptions \ref{assu subgaussian} and \ref{assu more on eigen values}(c), it holds for any $\epsilon>0$ that
\[\mathbb{P}(|z_{i,t,l}\zeta_{j,i,t}|>\epsilon)\leq \mathbb{P}(|z_{i,t,l}|>\sqrt{\epsilon})+\mathbb{P}(|\zeta_{j,i,t}|>\sqrt{\epsilon})\leq Ke^{-C\epsilon}. \]
such that Lemma 2.2.1 in \cite{vandervaartWellner1996} yields that $\|z_{i,t,l}\zeta_{j,i,t}\|_{\psi_1}\leq (1+K)/C$. Therefore, by Jensen's inequality and subadditivity of the Orlicz norm 
\begin{align*}
&\left\| \frac{1}{T}\sum_{t=1}^{T}\left( z_{i,t,l}\zeta_{j,i,t}-\mathbb{E}[z_{i,t,l}\zeta_{j,i,t}]\right)\right\| _{\psi_1}
\leq 
2\max_{1\leq t\leq T}\|z_{i,t,l}\zeta_{j,i,t}\|_{\psi_1}\leq \frac{2(1+K)}{C}.
\end{align*}
Using the definition of the Orlicz norm $\mathbb{E}\exp \del[1]{\frac{C}{2(1+K)}\envert[1]{ \frac{1}{T}\sum_{t=1}^{T}\del[0]{ z_{i,t,l}\zeta_{j,i,t}-\mathbb{E}[z_{i,t,l}\zeta_{j,i,t}]}}}\leq 2$. Using independence across $i$ (Assumption \ref{assu panel data}) to invoke Proposition \ref{prop adapation of Fan} in Appendix B with $D=C/(1+K)$, $\alpha=1/3$, $C_1=2$ and $\epsilon=\lambda_{node}/4\gtrsim \frac{1}{\sqrt{N}}$, we conclude (using $h_1\leq p$)
\begin{align*}
\mathbb{P}\del[2]{\max_{j\in H_1}\frac{1}{NT}\|Z_{-j}'\zeta_j\|_{\infty}> \epsilon }
&\leq 
h_1p\mathbb{P}\del[2]{ \envert[2]{ \sum_{i=1}^{N}\frac{1}{T}\sum_{t=1}^{T}(z_{i,t,l}\zeta_{j,i,t}-\mathbb{E}[z_{i,t,l}\zeta_{j,i,t}])} > \epsilon N }\\
&\leq
 Ah_1pe^{-B(\epsilon^2N)^{1/3}}
\leq
Ap^2e^{-BM^{1/3}\log p}=Ap^{2-BM^{1/3}}
\end{align*}
for positive constants $A$ and $B$. The upper bound of the preceding probability becomes arbitrarily small for $M$ sufficiently large such that
\[\max_{j\in H_1}\frac{1}{NT}\|Z_{-j}'\zeta_j\|_{\infty}=O_p(\lambda_{node}),\]
which is (\ref{eqn order for event D}).
In order to provide a lower bound on the probability of $\left( \cap_{j\in H_1} \mathcal{E}_{N,j}\right)$ define the event 
\[\tilde{\mathcal{E}}_{N,j}:=\cbr[3]{ \max_{1\leq l,k\leq p-1}\envert[3]{ \left[ \frac{1}{NT}Z_{-j}'Z_{-j}\right]_{lk} -[ \Psi_{Z,-j,-j}]_{lk}} \leq \frac{\kappa^2(\Psi_{Z,-j,-j}, s_j^*)}{32s_j^*}} \subseteq \mathcal{E}_{N,j} \]
by Proposition \ref{lemma difference between two matrices again} in Appendix B with $A=\Psi_{Z,-j,-j}$, $B=\frac{1}{NT}Z_{-j}'Z_{-j}$, $r=s_j^*$ and $\delta =\frac{\kappa^2(\Psi_{Z,-j,-j}, s_j^*)}{32s_j^*}$.
Observe that the relation
\begin{align*}
&\max_{1\leq l,k\leq p-1}\envert[3]{ \left[ \frac{1}{NT}Z_{-j}'Z_{-j}\right]_{lk} -\left[ \Psi_{Z,-j,-j}\right]_{lk} }\leq \max_{1\leq l,k\leq p}\envert[3]{ \left[ \frac{1}{NT}Z'Z\right]_{lk} -\left[ \Psi_{Z}\right]_{lk} } \\
&\leq \frac{\kappa^2(\Psi_{Z}, \max_{j\in H_1}s_j^*)}{32\bar{G}\lambda_{node}^{-\vartheta}}\leq \frac{\kappa^2(\Psi_{Z,-j,-j}, s_j^*)}{32s_j^*},
\end{align*}
implies $\mathcal{E}_N:=\cbr[2]{\max_{1\leq l,k\leq p}\envert[1]{ [ \frac{1}{NT}Z'Z]_{lk} -[ \Psi_{Z}]_{lk}}\leq \frac{\kappa^2(\Psi_{Z},  \max_{j\in H_1}s_j^*)}{32\bar{G}\lambda_{node}^{-\vartheta}}}\subseteq \tilde{\mathcal{E}}_{N,j}\subseteq \mathcal{E}_{N,j}$ for all $j\in H_1$ and hence $\mathcal{E}_N\subseteq \cap_{j\in H_1}\mathcal{E}_{N,j}$. It remains to provide a lower bound on $\mathbb{P}(\mathcal{E}_N)$. A typical element of $\frac{1}{NT}Z'Z-\Psi_{Z}$ is of the form $\frac{1}{NT}\sum_{i=1}^{N}\sum_{t=1}^{T}( z_{i,t,l}z_{i,t,k}-\mathbb{E}[z_{i,t,l}z_{i,t,k}])$ for some $l,k\in \{1,\ldots,p\}$. Invoking (\ref{eqn block comparison 1}) with $x=\frac{\kappa^2(\Psi_{Z},  \max_{j\in H_1}s_j^*)}{32\bar{G}\lambda_{node}^{-\vartheta}}\gtrsim \frac{1}{\sqrt{N}}$ (using $\bar{G}\lambda_{node}^{1-\vartheta}=O(\log^{3/2}p)$, implied by Assumption \ref{assu more on eigen values}(b))
\[\mathbb{P}\del[2]{ \envert[2]{ \frac{1}{NT}\sum_{i=1}^{N}\sum_{t=1}^{T}\del[0]{ z_{i,t,l}z_{i,t,k}-\mathbb{E}[z_{i,t,l}z_{i,t,k}]}}\geq x }\leq Ae^{-B(x^2N)^{1/3}},\]
for positive constants $A$ and $B$. Therefore,
\begin{align*}
\mathbb{P}( \mathcal{E}_N^c)=\mathbb{P}\left(\max_{1\leq l,k\leq p} \envert[3]{\left[  \frac{1}{NT}Z'Z\right] _{lk} -[ \Psi_{Z}]_{lk} }\geq x \right)
\leq
 p^2Ae^{-B(x^2N)^{1/3}}.
\end{align*}
The upper bound of the preceding probability becomes arbitrarily small for $M$ sufficiently large (using $\bar{G}\lambda_{node}^{1-\vartheta}=O(1)$, implied by Assumption \ref{assu more on eigen values}(b)).
In a similar manner, invoke (\ref{eqn block comparison 1}) with $x=\lambda_{node}=\sqrt{\frac{16M(\log p)^3}{N}}\gtrsim\frac{1}{\sqrt{N}}$ $(M>0)$,
\[\mathbb{P}(\mathcal{F}_N^c)=\mathbb{P}\del[2]{\max_{1\leq l,k\leq p}\envert[2]{ \left[ \frac{1}{NT}Z'Z\right]_{lk} -\left[ \Psi_{Z}\right]_{lk}} \geq x } \leq  Ap^2e^{-B(x^2N)^{1/3}}=Ap^{2-BM^{1/3}},\]
for positive constants $A$ and $B$, letting $B$ absorb the extra constants. The upper bound of the preceding probability becomes arbitrarily small for sufficiently large $N$ and $M$. We also have
\begin{align}
\left\| \frac{Z'Z}{NT}-\Psi_Z\right\|_{\infty}=O_p(\lambda_{node})=O_p\del[2]{ \sqrt{\frac{(\log p)^3}{N}}}\label{Fcal}.
\end{align}

Lastly, use Assumption \ref{assu more on eigen values}(b) in the display (\ref{prerate})to get the claimed orders.
\end{proof}

\vspace{0.5cm}

\begin{proof}[Proof of Lemma \ref{lemma step stone to final theorem}]
Recall (\ref{align tauj2 obtained via KKT}) and use $z_j=Z_{-j}\phi_j+\zeta_j$:
\begin{align*}
\hat{\tau}_j^2
&=\frac{1}{NT}\zeta_j'\zeta_j+\frac{1}{NT}\zeta_j'Z_{-j}\phi_j-\frac{1}{NT}(\hat{\phi}_j-\phi_j)'Z_{-j}'\zeta_j-\frac{1}{NT}(\hat{\phi}_j-\phi_j)'Z_{-j}'Z_{-j}\phi_j.
\end{align*}
Thus,
\begin{align}
\max_{j\in H_1}|\hat{\tau}_j^2-\tau_j^2|&\leq \max_{j\in H_1}\envert[2]{ \frac{1}{NT}\zeta_j'\zeta_j-\tau_j^2} +\max_{j\in H_1}\envert[2]{ \frac{1}{NT}\zeta_j'Z_{-j}\phi_j} \nonumber\\
&\quad+\max_{j\in H_1}\envert[2]{ \frac{1}{NT}(\hat{\phi}_j-\phi_j)'Z_{-j}'\zeta_j} +\max_{j\in H_1}\envert[2]{ \frac{1}{NT}(\hat{\phi}_j-\phi_j)'Z_{-j}'Z_{-j}\phi_j}. \label{align four terms to bound tauhat2 and tau2}
\end{align}
Consider the first term on the right of the inequality in (\ref{align four terms to bound tauhat2 and tau2}). By Assumption \ref{assu more on eigen values}(c), we have for all $\epsilon>0$, $\mathbb{P}(|\zeta_{j,i,t}^2|\geq\epsilon)= \mathbb{P}(|\zeta_{j,i,t}|\geq\sqrt{\epsilon})\leq \frac{1}{2}Ke^{-C\epsilon}$. It follows from Lemma 2.2.1 in \cite{vandervaartWellner1996} that $\|\zeta_{j,i,t}^2\|_{\psi_1}\leq (1+K/2)/C$. Therefore, by Jensen's inequality and subadditivity of the Orlicz norm 
\begin{align*}
&\left\| \frac{1}{T}\sum_{t=1}^{T}\left( \zeta_{j,i,t}^2-\mathbb{E}[\zeta_{j,i,t}^2]\right)\right\| _{\psi_1}\leq 2\max_{1\leq t\leq T}\|\zeta_{j,i,t}^2\|_{\psi_1}\leq \frac{2+K}{C}.
\end{align*}
Using the definition of the Orlicz norm, $\mathbb{E} \exp \del[1]{\frac{C}{2+K}\envert[1]{ \frac{1}{T}\sum_{t=1}^{T}\del[0]{ \zeta_{j,i,t}^2-\mathbb{E}[\zeta_{j,i,t}^2]}}} \leq 2$. Using independence across $i=1,\ldots, N$ (Assumption \ref{assu panel data}) to invoke Proposition \ref{prop adapation of Fan} in Appendix B with $D=C/(2+K)$, $\alpha=1/3$ and $C_1=2$ for $x\gtrsim \frac{1}{\sqrt{N}}$,
\[\mathbb{P}\del[2]{ \envert[2]{\frac{1}{N} \sum_{i=1}^{N}\frac{1}{T}\sum_{t=1}^{T}\del[0]{ \zeta_{j,i,t}^2-\mathbb{E}[\zeta_{j,i,t}^2]}}\geq x } \leq Ae^{-B(x^2N)^{1/3}},\]
for positive constants $A$ and $B$. Setting $x=\sqrt{\frac{M(\log h_1)^3}{N}}$ for some $M>0$, we have
\begin{align*}
&\mathbb{P}\del[3]{ \max_{j\in H_1}\envert[2]{ \frac{1}{N}\sum_{i=1}^{N}\frac{1}{T}\sum_{t=1}^{T}(\zeta_{j,i,t}^2-\mathbb{E}[\zeta_{j,i,t}^2])} \geq \sqrt{\frac{M(\log h_1)^3}{N}}} \\
\leq &\sum_{j\in H_1}\mathbb{P}\del[3]{ \envert[2]{ \frac{1}{N}\sum_{i=1}^{N}\frac{1}{T}\sum_{t=1}^{T}(\zeta_{j,i,t}^2-\mathbb{E}[\zeta_{j,i,t}^2])} \geq \sqrt{\frac{M(\log h_1)^3}{N}}}\leq Ah_1^{1-BM^{1/3}}.
\end{align*}
Recognising that the upper bound of the preceding probability becomes arbitrarily small for sufficiently large $N$ and $M$, we have 
\[\max_{j\in H_1}\envert[2]{ \frac{1}{NT}\zeta_j'\zeta_j-\tau_j^2}=O_p\del[2]{ \sqrt{\frac{(\log h_1)^3}{N}}}= O_p(\lambda_{node}). \]
Now consider the second term on the right of the inequality in (\ref{align four terms to bound tauhat2 and tau2}). Recall that
\[C=\left( \begin{array}{cccc}
1 & -\phi_{1,2} &\cdots  &-\phi_{1,p}\\
-\phi_{2,1} & 1  &\cdots  &-\phi_{2,p}\\
\vdots&\vdots&\ddots&\vdots\\
-\phi_{p,1} & -\phi_{p,2} &\cdots  &1  \\
\end{array}\right) \]
such that $C_j$ is the $j$th row of $C$ but written as a $p\times 1$ vector. Then 
\begin{align}
&\max_{j\in H_1}\|\phi_j\|_1=\max_{j\in H_1}\|\phi_j^*-\phi_j-\phi_j^*\|_1\leq \max_{j\in H_1}\|\phi_j^*-\phi_j\|_1+\max_{j\in H_1}\|\phi_j^*\|_1\leq \bar{G}\lambda_{node}^{1-\vartheta}+\max_{j\in H_1}\|\phi_j^*\|_1\notag\\
&\leq \bar{G}\lambda_{node}^{1-\vartheta}+\max_{j\in H_1}\sqrt{s_j^*}\|\phi_j^*\|\leq \bar{G}\lambda_{node}^{1-\vartheta}+\max_{j\in H_1}\sqrt{s_j^*}\|\phi_j\|\leq \bar{G}\lambda_{node}^{1-\vartheta}+\max_{j\in H_1}\sqrt{s_j^*}\|C_j\|\notag \\
&\leq \bar{G}\lambda_{node}^{1-\vartheta}+\max_{j\in H_1}\sqrt{s_j^*}\sqrt{\frac{C_j'\Psi_Z C_j}{\text{mineval}(\Psi_Z)}}=\bar{G}\lambda_{node}^{1-\vartheta}+\max_{j\in H_1}\frac{\sqrt{s_j^*}\sqrt{ \Psi_{Z,j,j}-\Psi_{Z,j,-j}\Psi_{Z,-j,-j}^{-1}\Psi_{Z,-j,j}}}{\sqrt{\text{mineval}(\Psi_Z)}}\notag\\
&\leq \bar{G}\lambda_{node}^{1-\vartheta}+\max_{j\in H_1}\frac{\sqrt{s_j^*}\sqrt{ \Psi_{Z,j,j}}}{\sqrt{\text{mineval}(\Psi_Z)}}\leq \bar{G}\lambda_{node}^{1-\vartheta}+\max_{j\in H_1}\frac{\sqrt{s_j^*}\sqrt{ \text{maxeval}(\Psi_Z)}}{\sqrt{\text{mineval}(\Psi_Z)}}=O(\bar{G}^{1/2}\lambda_{node}^{-\vartheta/2})\label{align l1 phij}
\end{align}
where the second inequality is due to (\ref{eqn bj l1 norm}), the second equality is due to (\ref{eqn definition of phij}), the seventh inequality is due to that Assumption \ref{assu more on eigen values}(a) implies that $\Psi_{Z,-j,-j}^{-1}$ is positive definite for all $j\in H_1$, and the last equality is due to (\ref{eqn sj star rate}) and Assumption \ref{assu more on eigen values}(b)). Now, 
\[\max_{j\in H_1}\left| \frac{1}{NT}\zeta_j'Z_{-j}\phi_j\right|\leq \max_{j\in H_1}\del[3]{\left\| \frac{1}{NT}\zeta_j'Z_{-j}\right\|_{\infty} \left\| \phi_j\right\|_1}=O_p(\lambda_{node})O(\bar{G}^{1/2}\lambda_{node}^{-\vartheta/2})=O_p(\bar{G}^{1/2}\lambda_{node}^{1-\vartheta/2}), \]
where the first equality is due to (\ref{eqn order for event D}).

The third term in (\ref{align four terms to bound tauhat2 and tau2}) is bounded as
\begin{align*}
\max_{j\in H_1}\envert[2]{\frac{1}{NT}(\hat{\phi}_j-\phi_j)'Z_{-j}'\zeta_j}&\leq \max_{j\in H_1}\del[2]{ \left\| \hat{\phi}_j-\phi_j\right\|_1 \left\| \frac{1}{NT}Z_{-j} '\zeta_j\right\|_{\infty}}=O_p(\bar{G}\lambda_{node}^{2-\vartheta}),
\end{align*}
where the equality is due to (\ref{eqn nodewise regression oracle inequality estimation}) and (\ref{eqn order for event D}).

To bound the fourth term on the right of the inequality in (\ref{align four terms to bound tauhat2 and tau2}), recall (\ref{eqn nodewise regression KKT conditions}) and manipulate to get $\frac{1}{NT}Z_{-j}'Z_{-j}(\hat{\phi}_j-\phi_j)=\frac{1}{NT}Z_{-j}'\zeta_j-\lambda_{node}w_j$. Thus,
\begin{align*}
\left\| \frac{1}{NT}(\hat{\phi}_j-\phi_j)'Z_{-j}'Z_{-j}\right\|_{\infty}&\leq \left\| \frac{1}{NT}Z_{-j}'\zeta_j\right\|_{\infty}+\lambda_{node}\|w_j\|_{\infty}=O_p(\lambda_{node}),
\end{align*}
where the equality is due to (\ref{eqn order for event D}). Thus,
\begin{align*}
\max_{j\in H_1}\left| \frac{1}{NT}(\hat{\phi}_j-\phi_j)'Z_{-j}'Z_{-j}\phi_j\right| &\leq \max_{j\in H_1}\left\| \frac{1}{NT}(\hat{\phi}_j-\phi_j)'Z_{-j}'Z_{-j}\right\|_{\infty}\max_{j\in H_1}\|\phi_j\|_1= O_p(\bar{G}^{1/2}\lambda_{node}^{1-\vartheta/2}),
\end{align*}
where the last equality is due to (\ref{align l1 phij}). Summing up all four terms on the right of the inequality in (\ref{align four terms to bound tauhat2 and tau2}), we get
\begin{align}
\max_{j\in H_1}|\hat{\tau}_j^2-\tau_j^2|&\leq O_p\left( \lambda_{node}\right)+O_p(\bar{G}^{1/2}\lambda_{node}^{1-\vartheta/2})+O_p(\bar{G}\lambda_{node}^{2-\vartheta})\nonumber=O_p(\bar{G}^{1/2}\lambda_{node}^{1-\vartheta/2})=o_p(1),\nonumber
\end{align}
where the first equality is due to that $O_p(\bar{G}^{1/2}\lambda_{node}^{1-\vartheta/2})$ dominates $O_p(\bar{G}\lambda_{node}^{2-\vartheta})$ by Assumption \ref{assu more on eigen values}(b), and the second equality is also due to Assumption \ref{assu more on eigen values}(b). This establishes (\ref{align lemma 2 1}).

We now prove (\ref{align lemma 2 2}). We first recall
\begin{align}
\tau^2_j&=\mathbb{E}\sbr[2]{\frac{1}{NT}\sum_{i=1}^{N}\sum_{t=1}^{T}(z_{i,t,j}-z_{i,t,-j}'\phi_j)^2}=\Psi_{Z,j,j}-\Psi_{Z,j,-j}\Psi_{Z,-j,-j}^{-1}\Psi_{Z,-j,j}=\frac{1}{\Theta_{Z,j,j}},\nonumber
\end{align}
Furthermore,
\[\Theta_{Z,j,j}\equiv \frac{e_j'\Theta_Z e_j}{\|e_j\|^2}\leq \max_{\delta\in \mathbb{R}^p\setminus\{0\}}\frac{\delta'\Theta_Z \delta}{\|\delta\|^2}=\text{maxeval}(\Theta_Z)=\frac{1}{\text{mineval}(\Psi_Z)}.\]
The preceding inequality is uniform in $j$. Thus, $\min_{j\in H_1}\tau^2_j\geq \text{mineval}(\Psi_Z)$, which is uniformly bounded away from zero by Assumption \ref{assu more on eigen values}(a). Therefore,
\[\min_{j\in H_1}\hat{\tau}_j^2=\min_{j\in H_1}(\hat{\tau}_j^2-\tau_j^2+\tau_j^2)\geq \min_{j\in H_1}\tau_j^2-\max_{j\in H_1}|\hat{\tau}_j^2-\tau_j^2|\geq \text{mineval}(\Psi_Z)-o_p(1).\]
Hence, we conclude that $\min_{j\in H_1}\hat{\tau}_j^2$ is bounded away from zero for $N$ large enough and $\max_{j\in H_1}\frac{1}{\hat{\tau}_j^2}=O_p(1)$ which establishes (\ref{align lemma 2 2}).

Hence,
\begin{align}
&\max_{j\in H_1}\envert[2]{\frac{1}{\hat{\tau}_j^2}-\frac{1}{\tau_j^2}}\leq \frac{\max_{j\in H_1}|\tau_j^2-\hat{\tau}_j^2|}{\min_{j\in H_1}\tau_j^2}\cdot\max_{j\in H_1}\frac{1}{\hat{\tau}_j^2}=\max_{j\in H_1}|\tau_j^2-\hat{\tau}_j^2|O(1)O_p(1)=O_p(\bar{G}^{1/2}\lambda_{node}^{1-\vartheta/2}),\nonumber
\end{align}
which establishes (\ref{align lemma 2 3}).

We can now bound $\max_{j\in H_1}\|\hat{\Theta}_{Z,j}-\Theta_{Z,j}\|_1$. Use the definition of $C_j$ and (\ref{align inverse of the partitioned matrix}) to recognise that $\Theta_{Z,j}=C_j\Theta_{Z,j,j}=C_j/\tau_j^2$.
\begin{align*}
&\max_{j\in H_1}\left\| \hat{\Theta}_{Z,j}-\Theta_{Z,j}\right\|_1= \max_{j\in H_1}\left\| \frac{\hat{C}_j}{\hat{\tau}_j^2}-\frac{C_j}{\tau_j^2}\right\|_1= \max_{j\in H_1}\left| \frac{1}{\hat{\tau}_j^2}-\frac{1}{\tau_j^2}\right|+\max_{j\in H_1}\left\| \frac{\hat{\phi}_j}{\hat{\tau}_j^2}-\frac{\phi_j}{\tau_j^2}\right\|_1\\
&\leq \max_{j\in H_1}\left| \frac{1}{\hat{\tau}_j^2}-\frac{1}{\tau_j^2}\right|+\max_{j\in H_1}\left\| \frac{\hat{\phi}_j}{\hat{\tau}_j^2}-\frac{\phi_j}{\hat{\tau}_j^2}\right\|_1 +\max_{j\in H_1}\left\| \frac{\phi_j}{\hat{\tau}_j^2}-\frac{\phi_j}{\tau_j^2}\right\|_1\\
&= \max_{j\in H_1}\envert[2]{ \frac{1}{\hat{\tau}_j^2}-\frac{1}{\tau_j^2}}+\max_{j\in H_1}\frac{1}{\hat{\tau}_j^2}\left\|\hat{\phi}_j -\phi_j\right\|_1 +\max_{j\in H_1}\|\phi_j\|_1\envert[2]{ \frac{1}{\hat{\tau}_j^2}-\frac{1}{\tau_j^2}}=O_p(\bar{G}\lambda_{node}^{1-\vartheta}),
\end{align*}
which establishes (\ref{align lemma 2 4}). Next, we bound $\max_{j\in H_1}\|\hat{\Theta}_{Z,j}-\Theta_{Z,j}\|$. Since
\[\left| (\hat{C}_j-C_j)'\frac{Z'Z}{NT} (\hat{C}_j-C_j) - (\hat{C}_j-C_j)'\Psi_Z (\hat{C}_j-C_j)\right|\leq \left\| \frac{Z'Z}{NT}-\Psi_Z\right\|_{\infty}\left\| \hat{C}_j-C_j\right\|_1^2, \]
\begin{equation}
\label{eqn AB term in from l1 to l2}
\max_{j\in H_1} \left| (\hat{C}_j-C_j)'\Psi_Z (\hat{C}_j-C_j)\right|\leq \max_{j\in H_1}\left| (\hat{C}_j-C_j)'\frac{Z'Z}{NT} (\hat{C}_j-C_j) \right| +\left\| \frac{Z'Z}{NT}-\Psi_Z\right\|_{\infty}\max_{j\in H_1}\left\| \hat{C}_j-C_j\right\|_1^2.
\end{equation}
Consider the first term on the right hand side of (\ref{eqn AB term in from l1 to l2}).
\begin{align*}
\max_{j\in H_1}\envert[2]{ (\hat{C}_j-C_j)'\frac{Z'Z}{NT} (\hat{C}_j-C_j) }&=\max_{j\in H_1}\frac{1}{NT}\left\| Z(\hat{C}_j-C_j)\right\|^2=\max_{j\in H_1}\frac{1}{NT}\left\| Z_{-j}(\hat{\phi}_j-\phi_j)\right\|^2=O_p(\bar{G}\lambda_{node}^{2-\vartheta}), 
\end{align*}
where the last equality is due to (\ref{eqn nodewise regression oracle inequality prediction}). Next, consider the second term on the right of the inequality (\ref{eqn AB term in from l1 to l2}).  We have
\[\left\| \frac{Z'Z}{NT}-\Psi_Z\right\|_{\infty}\max_{j\in H_1}\left\| \hat{C}_j-C_j\right\|_1^2=\left\| \frac{Z'Z}{NT}-\Psi_Z\right\|_{\infty}\max_{j\in H_1}\left\| \hat{\phi}_j-\phi_j\right\|_1^2=O_p\del[2]{\bar{G}^2\lambda_{node}^{3-2\vartheta}} ,\]
where the first equality is due to the definitions of $\hat{C}_j$ and $C_j$, and the second equality is due to (\ref{eqn nodewise regression oracle inequality estimation}) and (\ref{Fcal}). Adding up the two terms, we have
\[\max_{j\in H_1} \left| (\hat{C}_j-C_j)'\Psi_Z (\hat{C}_j-C_j)\right|\leq O_p(\bar{G}\lambda_{node}^{2-\vartheta})+O_p\del[2]{\bar{G}^2\lambda_{node}^{3-2\vartheta}} =O_p(\bar{G}\lambda_{node}^{2-\vartheta}),\]
where the last equality is due to Assumption \ref{assu more on eigen values}(b). Since $\max_{j\in H_1}\envert[0]{(\hat{C}_j-C_j)'\Psi_Z(\hat{C}_j-C_j) }\geq \text{mineval}(\Psi_Z)\max_{j\in H_1}\|\hat{C}_j-C_j\|^2$ and $\text{mineval}(\Psi_Z)$ is uniformly bounded away from zero we have $\max_{j\in H_1}\|\hat{\phi}_j-\phi_j\|= \max_{j\in H_1}\|\hat{C}_j-C_j\|= O_p(\bar{G}^{1/2}\lambda_{node}^{1-\vartheta/2})$. Then,
\begin{align*}
&\max_{j\in H_1}\enVert[2]{ \hat{\Theta}_{Z,j}-\Theta_{Z,j}}= \max_{j\in H_1}\enVert[2]{ \frac{\hat{C}_j}{\hat{\tau}_j^2}-\frac{C_j}{\tau_j^2}}\leq  \max_{j\in H_1}\envert[2]{ \frac{1}{\hat{\tau}_j^2}-\frac{1}{\tau_j^2}}+\max_{j\in H_1}\enVert[2]{ \frac{\hat{\phi}_j}{\hat{\tau}_j^2}-\frac{\phi_j}{\tau_j^2}}\\
&\leq \max_{j\in H_1}\envert[2]{\frac{1}{\hat{\tau}_j^2}-\frac{1}{\tau_j^2}}+\max_{j\in H_1}\frac{1}{\hat{\tau}_j^2}\left\|\hat{\phi}_j -\phi_j\right\| +\max_{j\in H_1}\|\phi_j\|\envert[2]{ \frac{1}{\hat{\tau}_j^2}-\frac{1}{\tau_j^2}}=O_p(\bar{G}^{1/2}\lambda_{node}^{1-\vartheta/2}),
\end{align*}
where in the last equality we have used that $\max_{j\in H_1}\|\phi_j\|=O(1)$, which follows from inspecting the arguments in (\ref{align l1 phij}). We have hence established (\ref{align lemma 2 5}). Finally, recall that $\Theta_{Z,j}=C_j\Theta_{Z,j,j}=C_j/\tau_j^2$. Thus,
\begin{align}
\max_{j\in H_1}\|\Theta_{Z,j}\|_1&=\max_{j\in H_1}|1/\tau_j^2|+\max_{j\in H_1}\|\phi_j\|_1\max_{j\in H_1}1/\tau_j^2=O(\bar{G}^{1/2}\lambda_{node}^{-\vartheta/2}).\label{align rates for max ThetaZj l1 norm}
\end{align}
Therefore,
\[\max_{j\in H_1}\left\| \hat{\Theta}_{Z,j}\right\|_1\leq \max_{j\in H_1}\left\| \hat{\Theta}_{Z,j}-\Theta_{Z,j}\right\|_1+\max_{j\in H_1}\|\Theta_{Z,j}\|_1=O_p(\bar{G}\lambda_{node}^{1-\vartheta})+O(\bar{G}^{1/2}\lambda_{node}^{-\vartheta/2})=O_p(\bar{G}^{1/2}\lambda_{node}^{-\vartheta/2}), \]
where the last equality is due to Assumption \ref{assu more on eigen values}(b).
\end{proof}

\subsection{Proof of Theorem \ref{thm Delta is small op 1}}

\begin{proof}[Proof of Theorem \ref{thm Delta is small op 1}]

The following assumption is implied by Assumption \ref{assu rates sufficient}.\footnote{To be precise, Assumption \ref{assu rates sufficient}(a) implies Assumption \ref{assu rates}(a) by recognising that $h_1\geq 1$ if $h_1\neq 0$, and $\bar{G}\sbr[2]{\frac{(\log p)^3}{N}}^{-\vartheta/2}\geq s_j^*\geq 1$. Assumption \ref{assu rates sufficient}(b) implies Assumption \ref{assu rates}(b) by recognising that $\sqrt{\frac{(\log(p\vee T))^7}{N}}=o(1)$ and $ \sqrt{\frac{(\log(N\vee T))^3}{T}}=o(1)$, implied by Assumption \ref{assu rates sufficient}(a) provided $h_1\neq 0$ and $h_2\neq 0$, respectively. Last, Assumption \ref{assu rates sufficient}(c) implies Assumption \ref{assu rates}(c).} However, as Assumption \ref{assu rates sufficient} is much simpler, we have chosen to use the latter in the main text even though it is slightly less general than the following assumption. Note again how the assumptions simplifies when either $h_1$ or $h_2$ equals 0.

\begin{assumption}
\label{assu rates}
\item \begin{enumerate}[(a)]
\item \begin{enumerate}[(i)]
\item $\frac{h_1^2\bar{G}^2\sbr[2]{\frac{(\log p)^3}{N}}^{-\vartheta}(\log(p\vee T))^5}{N}=o(1)$;
\item $\frac{h_1h_2\bar{G}\sbr[2]{\frac{(\log p)^3}{N}}^{-\vartheta/2}(\log (p\vee N\vee T))^3}{N}=o(1)$;
\item $\frac{h_1\bar{G}^2\sbr[2]{\frac{(\log p)^3}{N}}^{-\vartheta}(\log p)^3(\log (p\vee N))^3}{N}=o(1)$;
\item $\frac{h_1\bar{G}\sbr[2]{\frac{(\log p)^3}{N}}^{-\vartheta/2}(\log(N\vee T))^2(\log p)^2}{NT}=o(1)$;
\item $ \frac{(\log(N\vee T))^21\{h_2\neq 0\}}{T}=o(1)$.
\end{enumerate}
\item Let
\[a:=\frac{\left[ s_1\vee E\del[2]{\frac{(\log (p\vee N))^3}{T}}^{-\nu/2}\right] (\log(p\vee N))^3}{NT}.\]
\begin{enumerate}[(i)]
\item $h_1^2\bar{G}^2\sbr[2]{\frac{(\log p)^3}{N}}^{-\vartheta}\del[3]{1\vee\sqrt{\frac{(\log(p\vee T))^7}{N}}}a=o(1)$;
\item $h_1\bar{G}\sbr[2]{\frac{(\log p)^3}{N}}^{-\vartheta/2}\log(p\vee N\vee T)a=o(1)$;
\item $h_1h_2\bar{G}\sbr[2]{\frac{(\log p)^3}{N}}^{-\vartheta/2}(\log(p\vee N\vee T))^2a=o(1)$; 
\item $\sqrt{h_1h_2\bar{G}\sbr[2]{\frac{(\log p)^3}{N}}^{-\vartheta/2}N\log(p\vee N\vee T)}a=o(1)$;
\item $Nh_2^2\del[2]{1\vee \sqrt{\frac{(\log N)^3}{T}}}a=o(1)$.
\end{enumerate}
\item 
\begin{align*}
& \frac{(h_1\vee h_2)(\log(p\vee N))^3\sbr[2]{s_1^2\vee E^2\del[2]{\frac{(\log (p\vee N))^3}{T}}^{-\nu}}b}{N}=o(1),
\end{align*}
where $b:=\sbr[2]{\del[2]{\bar{G}\sbr[2]{\frac{(\log p)^3}{N}}^{-\vartheta/2}\log(p\vee N)\vee (\log p)^3}1\{h_1\neq 0\}}\vee \sbr[2]{\log(p\vee N)1\{h_2\neq 0\}}$.
\item $\text{mineval}(\Sigma_{\Pi\varepsilon})$ is uniformly bounded away from zero and $\text{maxeval}(\Sigma_{1,N})$ is uniformly bounded from above.
\end{enumerate}
\end{assumption}

We show that 
\begin{align*}
t=\frac{\rho'S\left( \tilde{\gamma}-\gamma\right)}{\sqrt{\rho'\hat{\Theta}\hat{\Sigma}_{\Pi\varepsilon}\hat{\Theta}'\rho}}\xrightarrow{d}N(0,1).
\end{align*}
To this end, note that by (\ref{eqn univariate inference equnation}) one may write $t=t_1+t_2$, where
\begin{align*}
t_1=\frac{\rho'\hat{\Theta}S^{-1}\Pi'\varepsilon}{\sqrt{\rho'\hat{\Theta}\hat{\Sigma}_{\Pi\varepsilon}\hat{\Theta}'\rho}}\ \text{ and }\ t_2=\frac{-\rho'\Delta}{\sqrt{\rho'\hat{\Theta}\hat{\Sigma}_{\Pi\varepsilon}\hat{\Theta}'\rho}}.
\end{align*}
Defining 
\begin{align*}
t_1'=\frac{\rho'\Theta S^{-1}\Pi'\varepsilon}{\sqrt{\rho'\Theta\Sigma_{\Pi\varepsilon}\Theta'\rho}}
\end{align*}
it suffices to show that $t_1'\xrightarrow{d}N(0,1)$, $t_1'-t_1=o_p(1)$, and $t_2=o_p(1)$. In the sequel we first show that $t_1-t_1'=o_p(1)$, then $t_1'\xrightarrow{d}N(0,1)$ and finally $t_2=o_p(1)$. To show that $t_1-t_1'=o_p(1)$, it suffices to show that the denominators as well as the numerators of $t_1$ and $t_1'$ are asymptotically equivalent since 
\begin{align}
&\rho'\Theta \Sigma_{\Pi\varepsilon}\Theta'\rho \geq \text{mineval}(\Sigma_{\Pi\varepsilon})\left( \text{mineval}(\Theta)\right)^2= \frac{\text{mineval}(\Sigma_{\Pi\varepsilon})}{\left( \text{maxeval}(\Psi)\right)^2}\label{align denominator of t1 prime bounded away from zero}
\end{align}
which is uniformly bounded away from zero by Assumptions \ref{assu more on eigen values}(a) and \ref{assu rates}(d).

\subsubsection{Denominators of $t_1$ and $t_1'$}

We first show that the denominators of $t_1$ and $t_1'$ are asymptotically equivalent, i.e.,
\begin{equation}
\label{eqn consistency covariance matrix denominator equivalence}
|\rho'\hat{\Theta}\hat{\Sigma}_{\Pi\varepsilon}\hat{\Theta}'\rho-\rho' \Theta \Sigma_{\Pi\varepsilon} \Theta'\rho|=o_p(1).
\end{equation}
Write
\begin{align}
&\left| (\rho_1', \rho_2')\left( \begin{array}{cc}
\hat{\Theta}_Z\hat{\Sigma}_{1,N}\hat{\Theta}_Z' & \hat{\Theta}_Z\hat{\Sigma}_{2,N}\\
\hat{\Sigma}_{2,N}'\hat{\Theta}_Z' & \hat{\Sigma}_{3,N}
\end{array}\right)\left( \begin{array}{c}
\rho_1\\
\rho_2
\end{array}\right)-  (\rho_1', \rho_2')\left( \begin{array}{cc}
\Theta_Z\Sigma_{1,N}\Theta_Z' & \Theta_Z\Sigma_{2,N}\\
\Sigma_{2,N}'\Theta_Z' & \Sigma_{3,N}
\end{array}\right)\left( \begin{array}{c}
\rho_1\\
\rho_2
\end{array}\right) \right| \nonumber\\
&\leq |\rho_1'\hat{\Theta}_Z\hat{\Sigma}_{1,N}\hat{\Theta}_Z'\rho_1-\rho_1'\Theta_Z\Sigma_{1,N}\Theta_Z'\rho_1|\label{align consistency covariance denominator 1}\\
&\quad+2|\rho_1'\hat{\Theta}_Z\hat{\Sigma}_{2,N}\rho_2-\rho_1'\Theta_Z\Sigma_{2,N}\rho_2|\label{align consistency covariance denominator 2}\\
&\quad+|\rho_2'\hat{\Sigma}_{3,N}\rho_2-\rho_2'\Sigma_{3,N}\rho_2|.\label{align consistency covariance denominator 3}
\end{align}
To establish (\ref{eqn consistency covariance matrix denominator equivalence}), we show that (\ref{align consistency covariance denominator 1}), (\ref{align consistency covariance denominator 2}) and (\ref{align consistency covariance denominator 3}) are $o_p(1)$, respectively. 

\subsubsection*{(\ref{align consistency covariance denominator 1}) is $o_p(1)$:} 

Define $\tilde{\Sigma}_{1,N}:=\frac{1}{NT}\sum_{i=1}^{N}\sum_{t=1}^{T}\varepsilon_{i,t}^2z_{i,t}z_{i,t}'$. To show that  (\ref{align consistency covariance denominator 1}) is $o_p(1)$, it suffices to show that
\begin{align}
|\rho_1'\hat{\Theta}_Z\hat{\Sigma}_{1,N}\hat{\Theta}_Z'\rho_1-\rho_1'\hat{\Theta}_Z\tilde{\Sigma}_{1,N}\hat{\Theta}_Z'\rho_1|&=o_p(1)\label{align consistency covariance denominator 11}\\
|\rho_1'\hat{\Theta}_Z\tilde{\Sigma}_{1,N}\hat{\Theta}_Z'\rho_1-\rho_1'\hat{\Theta}_Z\Sigma_{1,N}\hat{\Theta}_Z'\rho_1|&=o_p(1)\label{align consistency covariance denominator 12}\\
|\rho_1'\hat{\Theta}_Z\Sigma_{1,N}\hat{\Theta}_Z'\rho_1-\rho_1'\Theta_Z\Sigma_{1,N}\Theta_Z'\rho_1|&=o_p(1).\label{align consistency covariance denominator 13}
\end{align}
We prove (\ref{align consistency covariance denominator 11}) first. Note that
\[|\rho_1'\hat{\Theta}_Z\hat{\Sigma}_{1,N}\hat{\Theta}_Z'\rho_1-\rho_1'\hat{\Theta}_Z\tilde{\Sigma}_{1,N}\hat{\Theta}_Z'\rho_1|\leq \enVert[0]{ \hat{\Sigma}_{1,N}-\tilde{\Sigma}_{1,N}}_{\infty}\enVert[0]{ \hat{\Theta}_{Z}'\rho_1}_1^2.  \]
First,
\begin{align}
&\enVert[0]{ \hat{\Theta}_{Z}'\rho_1}_1= \enVert[4]{\sum_{j\in H_1}\hat{\Theta}_{Z,j}\rho_{1j}}_1 \leq  \sum_{j\in H_1}|\rho_{1j}|\left\|\hat{\Theta}_{Z,j}\right\|_1 =O_p(h_1^{1/2}\bar{G}^{1/2}\lambda_{node}^{-\vartheta/2}),\label{align ThetaZ rho1 l1 norm}
\end{align}
where the last equality is due to (\ref{align lemma 2 6}). We now bound $\left\| \hat{\Sigma}_{1,N}-\tilde{\Sigma}_{1,N}\right\|_{\infty}$. Since $\hat{\varepsilon}_{i,t}=y_{i,t}-z_{i,t}'\hat{\alpha}-\hat{\eta}_i=\varepsilon_{i,t}-z_{i,t}'(\hat{\alpha}-\alpha)-(\hat{\eta}_i-\eta_i)=:\varepsilon_{i,t}-\pi_{i,t}(\hat{\gamma}-\gamma)$, substituting for $\hat{\varepsilon}_{i,t}$, we have
\begin{align}
&\left\| \hat{\Sigma}_{1,N}-\tilde{\Sigma}_{1,N}\right\|_{\infty}=\enVert[4]{ \frac{1}{NT}\sum_{i=1}^{N}\sum_{t=1}^{T}\hat{\varepsilon}_{i,t}^2z_{i,t}z_{i,t}' -\frac{1}{NT}\sum_{i=1}^{N}\sum_{t=1}^{T}\varepsilon_{i,t}^2z_{i,t}z_{i,t}'}_{\infty}\nonumber\\
&\leq  2\enVert[4]{ \frac{1}{NT}\sum_{i=1}^{N}\sum_{t=1}^{T}z_{i,t}z_{i,t}'\varepsilon_{i,t}\pi_{i,t}'(\hat{\gamma}-\gamma)}_{\infty}+\enVert[4]{\frac{1}{NT}\sum_{i=1}^{N}\sum_{t=1}^{T}z_{i,t}z_{i,t}'[ \pi_{i,t}'(\hat{\gamma}-\gamma)]^2}_{\infty}.  \label{align consistency covariance denominator 11a}
\end{align} 
Consider the first term of (\ref{align consistency covariance denominator 11a}). A typical element of $\frac{1}{NT}\sum_{i=1}^{N}\sum_{t=1}^{T}z_{i,t}z_{i,t}'\varepsilon_{i,t}\pi_{i,t}'(\hat{\gamma}-\gamma)$ is
\begin{align}
&\frac{1}{NT}\sum_{j=1}^{NT}z_{j,l}z_{j,k}\varepsilon_{j}\pi_{j}'(\hat{\gamma}-\gamma) \leq \frac{1}{NT}\del[3]{ \sum_{j=1}^{NT}z_{j,l}^2z_{j,k}^2\varepsilon_{j}^2} ^{1/2}\del[3]{ \sum_{j=1}^{NT} [ \pi_{j}'(\hat{\gamma}-\gamma)]^2 }^{1/2} \nonumber\\
&= \del[3]{ \frac{1}{NT} \sum_{i=1}^{N}\sum_{t=1}^{T}z_{i,t,l}^2z_{i,t,k} ^2\varepsilon_{i,t}^2}^{1/2}\del[2]{ \frac{1}{NT}\left\| \Pi(\hat{\gamma}-\gamma) \right\|^2}^{1/2}\label{align 6 subgaussian times prediction}
\end{align}
for some $l,k\in \{1,\ldots,p\}$, where the inequality is due to Cauchy-Schwarz inequality. Use independence across $i$ (Assumption \ref{assu panel data}) and subgaussianity (Assumption \ref{assu subgaussian}) to invoke Proposition \ref{prop product of L subgaussian} in Appendix B, such that
\begin{equation*}
\max_{1\leq l\leq p}\max_{1\leq k\leq p}\envert[2]{ \frac{1}{NT}\sum_{i=1}^{N}\sum_{t=1}^{T}(z_{i,t,l}^2z_{i,t,k}^2\varepsilon_{i,t}^2-\mathbb{E}[z_{i,t,l}^2z_{i,t,k}^2\varepsilon_{i,t}^2] )}=O_p\del[2]{ \sqrt{\frac{(\log(p^2T))^7}{N}}}
\end{equation*}
and
\begin{equation*}
\max_{1\leq l\leq p}\max_{1\leq k\leq p}\max_{1\leq i\leq N}\max_{1\leq t\leq T}\mathbb{E}[ z_{i,t,l}^2z_{i,t,k}^2\varepsilon_{i,t}^2]\leq A=O(1)
\end{equation*}
for some positive constant $A$. Then, by the triangle inequality,
\begin{align}
&\max_{1\leq l\leq p}\max_{1\leq k\leq p}\envert[2]{ \frac{1}{NT}\sum_{i=1}^{N}\sum_{t=1}^{T}z_{i,t,l}^2z_{i,t,k}^2\varepsilon_{i,t}^2}=O_p\del[3]{\sqrt{\frac{(\log(p\vee T))^7}{N}}}+O(1).\label{align log p vee T 7 rate}
\end{align}
Combining (\ref{align 6 subgaussian times prediction}) and (\ref{align log p vee T 7 rate}), we have
\begin{align}
&\enVert[4]{ \frac{1}{NT}\sum_{i=1}^{N}\sum_{t=1}^{T}z_{i,t}z_{i,t}'\varepsilon_{i,t}\pi_{i,t}'(\hat{\gamma}-\gamma)}_{\infty}=O_p\del[3]{\frac{(\log(p\vee T))^{7/4}}{N^{1/4}}\vee 1}\del[2]{ \frac{1}{NT}\left\| \Pi(\hat{\gamma}-\gamma) \right\|^2}^{1/2}.\label{align consistency covariance denominator 11a actual result}
\end{align}

We now consider the second term of (\ref{align consistency covariance denominator 11a}). A typical element of $\frac{1}{NT}\sum_{i=1}^{N}\sum_{t=1}^{T}z_{i,t}z_{i,t}'[ \pi_{i,t}'(\hat{\gamma}-\gamma)]^2$ is $\frac{1}{NT}\sum_{i=1}^{N}\sum_{t=1}^{T}z_{i,t,l}z_{i,t,k}[ \pi_{i,t}'(\hat{\gamma}-\gamma)]^2\leq \max_{1\leq i\leq N}\max_{1\leq t\leq T}|z_{i,t,l}z_{i,t,k}|\frac{1}{NT}\|\Pi(\hat{\gamma}-\gamma)\|^2$ for some $l,k\in \{1,\ldots,p\}$. Recall that we have proved in the proof of Lemma \ref{lemma lower bound on tildeBnt} that $\|z_{i,t,l}z_{i,t,k}\|_{\psi_1}\leq (1+K)/C$. Using the definition of the Orlicz norm, we have $\mathbb{E} e^{\frac{C}{1+K}|z_{i,t,l}z_{i,t,k}|}\leq 2$. Using Markov's inequality, we have for any $\epsilon>0$ 
\begin{align*}
&\mathbb{P}\left( \max_{1\leq l\leq p}\max_{1\leq k\leq p}\max_{1\leq i\leq N}\max_{1\leq t\leq T}|z_{i,t,l}z_{i,t,k}|\geq \epsilon\right) \leq \sum_{l=1}^{p}\sum_{k=1}^{p}\sum_{i=1}^{N}\sum_{t=1}^{T}\frac{\mathbb{E} e^{\frac{C}{1+K}|z_{i,t,l}z_{i,t,k}|} }{e^{\frac{C}{1+K}\epsilon}}\leq 2NTp^2e^{-\frac{C}{1+K}\epsilon}.
\end{align*}
Set $\epsilon=M\log(p^2NT)$ for some $M>0$ and note that the upper bound of the preceding probability becomes arbitrarily small for $N$ and $M$ sufficiently large. Thus,
\[\max_{1\leq l\leq p}\max_{1\leq k\leq p}\max_{1\leq i\leq N}\max_{1\leq t\leq T}|z_{i,t,l}z_{i,t,k}|=O_p(\log(p^2NT))\]
and we get
\begin{align}
&\left\| \frac{1}{NT}\sum_{i=1}^{N}\sum_{t=1}^{T}z_{i,t}z_{i,t}'[ \pi_{i,t}'(\hat{\gamma}-\gamma)]^2\right\|_{\infty}=O_p(\log(p\vee N\vee T))\frac{1}{NT}\|\Pi(\hat{\gamma}-\gamma)\|^2. \label{align consistency covariance denominator 11b actual result}
\end{align}
Combining (\ref{align consistency covariance denominator 11a actual result}) and (\ref{align consistency covariance denominator 11b actual result}), conclude
\begin{align*}
&\left\| \hat{\Sigma}_{1,N}-\tilde{\Sigma}_{1,N}\right\|_{\infty}\\
&=O_p\del[3]{\frac{(\log(p\vee T))^{7/4}}{N^{1/4}}\vee 1}\del[2]{ \frac{1}{NT}\left\| \Pi(\hat{\gamma}-\gamma) \right\|^2}^{1/2}+O_p(\log(p\vee N\vee T))\frac{1}{NT}\|\Pi(\hat{\gamma}-\gamma)\|^2.
\end{align*}
Therefore, combining the preceding rates with (\ref{align ThetaZ rho1 l1 norm}) one gets
\begin{align*}
&|\rho_1'\hat{\Theta}_Z\hat{\Sigma}_{1,N}\hat{\Theta}_Z'\rho_1-\rho_1'\hat{\Theta}_Z\tilde{\Sigma}_{1,N}\hat{\Theta}_Z'\rho_1|\\
&= O_p(h_1\bar{G}\lambda_{node}^{-\vartheta})O_p\del[3]{\frac{(\log(p\vee T))^{7/4}}{N^{1/4}}\vee 1}\left[ \frac{1}{NT}\left\| \Pi(\hat{\gamma}-\gamma) \right\|^2\right]^{1/2}\\
&\qquad+O_p(h_1\bar{G}\lambda_{node}^{-\vartheta})O_p(\log(p\vee N\vee T))\frac{1}{NT}\|\Pi(\hat{\gamma}-\gamma)\|^2\\
&=o_p(1),
\end{align*}
where the last equality is also due to Assumption \ref{assu rates}(b)(i)-(ii), which establishes (\ref{align consistency covariance denominator 11}).

Next, turn to (\ref{align consistency covariance denominator 12}). Note that
\[|\rho_1'\hat{\Theta}_Z\tilde{\Sigma}_{1,N}\hat{\Theta}_Z'\rho_1-\rho_1'\hat{\Theta}_Z\Sigma_{1,N}\hat{\Theta}_Z'\rho_1|\leq \enVert[1]{ \tilde{\Sigma}_{1,N}-\Sigma_{1,N}}_{\infty}\enVert[1]{ \hat{\Theta}_{Z}'\rho_1}_1^2.  \]
Given (\ref{align ThetaZ rho1 l1 norm}), we only need to consider $\enVert[0]{ \tilde{\Sigma}_{1,N}-\Sigma_{1,N}}_{\infty}$. Using independence across $i$ (Assumption \ref{assu panel data}) and subgaussianity (Assumption \ref{assu subgaussian}) to invoke Proposition \ref{prop product of L subgaussian} in Appendix B such that
\begin{align}
\enVert[1]{\tilde{\Sigma}_{1,N}-\Sigma_{1,N}}_{\infty}&=
\max_{1\leq l\leq p}\max_{1\leq k\leq p}\envert[2]{ \frac{1}{NT}\sum_{i=1}^{N}\sum_{t=1}^{T}(z_{i,t,l}z_{i,t,k}\varepsilon_{it}^2-\mathbb{E}[z_{i,t,l}z_{i,t,k}\varepsilon_{i,t}^2] )}=O_p\del[3]{ \sqrt{\frac{(\log(p^2T))^5}{N}}}. \label{align Sigma1Ntilde -Sigma1N sup}
\end{align}
Thus,
\begin{align*}
&|\rho_1'\hat{\Theta}_Z\tilde{\Sigma}_{1,N}\hat{\Theta}_Z'\rho_1-\rho_1'\hat{\Theta}_Z\Sigma_{1,N}\hat{\Theta}_Z'\rho_1|=O_p\left( \sqrt{\frac{(\log(p\vee T))^5}{N}}h_1\bar{G}\lambda_{node}^{-\vartheta}\right)=o_p(1),
\end{align*}
where the last equality is due to Assumption \ref{assu rates}(a)(i), establishing (\ref{align consistency covariance denominator 12}).

To prove (\ref{align consistency covariance denominator 13}) invoke Lemma \ref{lemma vandegeer 6.1} in Appendix B:
\begin{align*}
&|\rho_1'\hat{\Theta}_Z\Sigma_{1,N}\hat{\Theta}_Z'\rho_1-\rho_1'\Theta_Z\Sigma_{1,N}\Theta_Z'\rho_1|\leq \|\Sigma_{1,N}\|_{\infty}\|(\hat{\Theta}_{Z}'-\Theta_Z')\rho_1\|_1^2+2\|\Sigma_{1,N}\Theta_Z'\rho_1\|\|(\hat{\Theta}_{Z}'-\Theta_Z')\rho_1\|\\
&\leq \|\Sigma_{1,N}\|_{\infty}\|(\hat{\Theta}_{Z}'-\Theta_Z')\rho_1\|_1^2+2\text{maxeval}(\Sigma_{1,N})\|\Theta_Z'\rho_1\|\|(\hat{\Theta}_{Z}'-\Theta_Z')\rho_1\|.
\end{align*}
First, note that $\|\Sigma_{1,N}\|_{\infty}$ is uniformly bounded as every entry is an average of uniformly bounded population moments (see Proposition \ref{prop product of L subgaussian} in Appendix B). 
\begin{align}
&\|(\hat{\Theta}_{Z}'-\Theta_Z')\rho_1\|_1\leq \sum_{j\in H_1} \enVert[1]{ \hat{\Theta}_{Z,j}-\Theta_{Z,j}}_1|\rho_{1j}|\leq \max_{j\in H_1}\enVert[1]{ \hat{\Theta}_{Z,j}-\Theta_{Z,j}}_1\sqrt{h_1}\notag\\
&=O_p\del[3]{\bar{G}\sbr[3]{\frac{(\log p)^3}{N}}^{\frac{1-\vartheta}{2}}\sqrt{h_1}}=o_p(1),\label{align for the use of comparing numerators of t1 and t1 prime}
\end{align}
where the first equality is due to (\ref{align lemma 2 4}), and the last equality is due to Assumption \ref{assu rates}(a)(i). Next, $\|\Theta_Z'\rho_1\|\leq \text{maxeval}(\Theta_Z)\|\rho_1\|\leq \text{maxeval}(\Theta_Z)=1/\text{mineval}(\Psi_Z)$, which is uniformly bounded from above by Assumption \ref{assu more on eigen values}(a). Furthermore,
\begin{align*}
&\|(\hat{\Theta}_{Z}'-\Theta_Z')\rho_1\|=\enVert[2]{\sum_{j\in H_1}(\hat{\Theta}_{Z,j}-\Theta_{Z,j})\rho_{1j}}\leq \sum_{j\in H_1} \left\| \hat{\Theta}_{Z,j}-\Theta_{Z,j}\right\||\rho_{1j}|\\
&\leq  \max_{j\in H_1}\enVert[1]{ \hat{\Theta}_{Z,j}-\Theta_{Z,j}}\sqrt{h_1}=O_p\del[3]{\bar{G}^{1/2}\sbr[3]{\frac{(\log p)^3}{N}}^{\frac{2-\vartheta}{4}}\sqrt{h_1}}=o_p(1),
\end{align*}
where the second last equality is due to (\ref{align lemma 2 5}), and the last equality is due to (\ref{align for the use of comparing numerators of t1 and t1 prime}). Thus, we have established (\ref{align consistency covariance denominator 13}) concluding the proof of (\ref{align consistency covariance denominator 1}) is $o_p(1)$.

\subsubsection*{(\ref{align consistency covariance denominator 2}) is $o_p(1)$:}

Define $\tilde{\Sigma}_{2,N}:=\frac{1}{\sqrt{N}T}\sum_{i=1}^{N}\sum_{t=1}^{T}\varepsilon_{i,t}^2z_{i,t}d_{i,t}'$. It suffices to show
\begin{align}
|\rho_1'\hat{\Theta}_Z\hat{\Sigma}_{2,N}\rho_2-\rho_1'\hat{\Theta}_Z\tilde{\Sigma}_{2,N}\rho_2|&=o_p(1)\label{align consistency covariance denominator 21}\\
|\rho_1'\hat{\Theta}_Z\tilde{\Sigma}_{2,N}\rho_2-\rho_1'\hat{\Theta}_Z\Sigma_{2,N}\rho_2|&=o_p(1)\label{align consistency covariance denominator 22}\\
|\rho_1'\hat{\Theta}_Z\Sigma_{2,N}\rho_2-\rho_1'\Theta_Z\Sigma_{2,N}\rho_2|&=o_p(1).\label{align consistency covariance denominator 23}
\end{align}

Consider (\ref{align consistency covariance denominator 21}) first. Note that
\begin{align*}
&|\rho_1'\hat{\Theta}_Z\hat{\Sigma}_{2,N}\rho_2-\rho_1'\hat{\Theta}_Z\tilde{\Sigma}_{2,N}\rho_2|\leq \enVert[2]{ \rho_1'\hat{\Theta}_Z\left( \hat{\Sigma}_{2,N}-\tilde{\Sigma}_{2,N}\right)}_{\infty}\left\|   \rho_2\right\|_1\\
&\leq \enVert[1]{ \rho_1'\hat{\Theta}_Z}_1 \enVert[1]{  \hat{\Sigma}_{2,N}-\tilde{\Sigma}_{2,N}}_{\infty}\sqrt{h_2}=O_p\del[1]{ \sqrt{h_1h_2\bar{G}\lambda_{node}^{-\vartheta}}}\enVert[1]{ \hat{\Sigma}_{2,N}-\tilde{\Sigma}_{2,N}}_{\infty},
\end{align*}
where the last equality is due to (\ref{align ThetaZ rho1 l1 norm}). In addition,
\begin{align}
&\left\| \hat{\Sigma}_{2,N}-\tilde{\Sigma}_{2,N}\right\|_{\infty}=\enVert[4]{ \frac{1}{\sqrt{N}T}\sum_{i=1}^{N}\sum_{t=1}^{T}\hat{\varepsilon}_{i,t}^2z_{i,t}d_{i,t}' -\frac{1}{\sqrt{N}T}\sum_{i=1}^{N}\sum_{t=1}^{T}\varepsilon_{i,t}^2z_{i,t}d_{i,t}'}_{\infty}\nonumber\\
&\leq  2\enVert[4]{ \frac{1}{\sqrt{N}T}\sum_{i=1}^{N}\sum_{t=1}^{T}z_{i,t}d_{i,t}'\varepsilon_{i,t}\pi_{i,t}'(\hat{\gamma}-\gamma)}_{\infty}+\enVert[4]{ \frac{1}{\sqrt{N}T}\sum_{i=1}^{N}\sum_{t=1}^{T}z_{i,t}d_{i,t}'[ \pi_{i,t}'(\hat{\gamma}-\gamma)]^2}_{\infty} \label{align consistency covariance denominator 21a}
\end{align} 

Consider the first term of (\ref{align consistency covariance denominator 21a}). A typical element of $\frac{1}{\sqrt{N}T}\sum_{i=1}^{N}\sum_{t=1}^{T}z_{i,t}d_{i,t}'\varepsilon_{i,t}\pi_{i,t}'(\hat{\gamma}-\gamma)$ is 
\begin{align*}
&\frac{1}{\sqrt{N}T}\sum_{j=1}^{NT}z_{j,l}d_{j,k}\varepsilon_{j}\pi_{j}'(\hat{\gamma}-\gamma) \leq \frac{1}{\sqrt{N}T}\del[4]{ \sum_{j=1}^{NT}z_{j,l}^2d_{j,k}^2\varepsilon_{j}^2}^{1/2}\del[4]{ \sum_{j=1}^{NT} [ \pi_{j}'(\hat{\gamma}-\gamma) ]^2 }^{1/2} \\
&= \del[4]{\frac{1}{T} \sum_{i=1}^{N}\sum_{t=1}^{T}z_{i,t,l}^2d_{i,t,k}^2\varepsilon_{i,t}^2} ^{1/2}\frac{1}{\sqrt{NT}}\left\| \Pi(\hat{\gamma}-\gamma) \right\|= \del[4]{\frac{1}{T} \sum_{t=1}^{T}z_{k,t,l}^2\varepsilon_{k,t}^2}^{1/2}\frac{1}{\sqrt{NT}}\left\| \Pi(\hat{\gamma}-\gamma) \right\|
\end{align*}
for some $l\in \{1,\ldots,p\}$ and $k\in \{1,\ldots,N\}$ where the inequality is due to Cauchy-Schwarz inequality. By subgaussianity, Assumption \ref{assu subgaussian}, we can use the same technique as in (\ref{align bound exponential moments}) in Proposition \ref{prop product of L subgaussian} in Appendix B to prove $\mathbb{E} e^{D\left| \frac{1}{T}\sum_{t=1}^{T}z_{i,t,l}^2\varepsilon_{i,t}^2\right| ^{1/2}}\leq BT$ for positive constants $D, B$. Using Markov's inequality, we have for $\epsilon>0$
\begin{align*}
&\mathbb{P}\del[3]{ \max_{1\leq l\leq p}\max_{1\leq k\leq N}\envert[2]{ \frac{1}{T}\sum_{t=1}^{T}z_{k,t,l}^2\varepsilon_{k,t}^2}\geq \epsilon} \leq
 \sum_{l=1}^{p}\sum_{k=1}^{N}\frac{\mathbb{E} e^{D\left| \frac{1}{T}\sum_{t=1}^{T}z_{k,t,l}^2\varepsilon_{k,t}^2\right|^{1/2}} }{e^{D\epsilon^{1/2}}}\leq BpNTe^{-D\epsilon^{1/2}}.
\end{align*}
Set $\epsilon=M(\log(pNT))^2$ for some $M>0$ and note that the upper bound of the preceding probability becomes arbitrarily small for $N$ and $M$ sufficiently large. Thus, $\max_{1\leq l\leq p}\max_{1\leq k\leq N}\left| \frac{1}{T}\sum_{t=1}^{T}z_{k,t,l}^2\varepsilon_{k,t}^2\right|=O_p((\log(pNT))^2)$. Therefore,
\begin{align}
&\enVert[4]{ \frac{1}{\sqrt{N}T}\sum_{i=1}^{N}\sum_{t=1}^{T}z_{i,t}d_{i,t}'\varepsilon_{i,t}\pi_{i,t}'(\hat{\gamma}-\gamma)}_{\infty}\nonumber\leq \del[4]{\max_{1\leq l\leq p}\max_{1\leq k\leq N}\frac{1}{T}\sum_{t=1}^{T}z_{k,t,l}^2\varepsilon_{k,t}^2} ^{1/2}\frac{1}{\sqrt{NT}}\left\| \Pi(\hat{\gamma}-\gamma) \right\|\nonumber\\
&\leq O_p( \log(pNT)) \frac{1}{\sqrt{NT}}\left\| \Pi(\hat{\gamma}-\gamma) \right\|. \label{align consistency covariance denominator 21a actual result}
\end{align}
Now consider the second term of (\ref{align consistency covariance denominator 21a}). A typical element of $\frac{1}{\sqrt{N}T}\sum_{i=1}^{N}\sum_{t=1}^{T}z_{i,t}d_{i,t}'[ \pi_{i,t}'(\hat{\gamma}-\gamma)]^2$ is
\begin{align*}
\frac{1}{\sqrt{N}T}\sum_{i=1}^{N}\sum_{t=1}^{T}z_{i,t,l}d_{i,t,k}[ \pi_{i,t}'(\hat{\gamma}-\gamma)]^2&\leq \max_{1\leq i\leq N}\max_{1\leq t\leq T}\sqrt{N}|z_{i,t,l}d_{i,t,k}|\frac{1}{NT}\|\Pi(\hat{\gamma}-\gamma)\|^2 \\
\leq \max_{1\leq t\leq T}\sqrt{N}|z_{k,t,l}|\frac{1}{NT}\|\Pi(\hat{\gamma}-\gamma)\|^2
\end{align*}
for some $l\in \{1,\ldots,p\}$, $k\in \{1,\ldots,N\}$. Using Markov's inequality, we have for any $\epsilon>0$
\begin{align*}
&\mathbb{P}\del[2]{\max_{1\leq l\leq p}\max_{1\leq k\leq N}\max_{1\leq t\leq T}|z_{k,t,l}|\geq \epsilon} \leq  \sum_{l=1}^{p}\sum_{k=1}^{N}\sum_{t=1}^{T}\mathbb{P}\left(|z_{k,t,l}|\geq \epsilon\right)\leq pNT\frac{K}{2}e^{-C\epsilon^2}.
\end{align*}
Set $\epsilon=\sqrt{M\log(pNT)}$ for some $M>0$ to see that the upper bound of the preceding probability becomes arbitrarily small for $N$ and $M$ sufficiently large. Thus, $\max_{1\leq l\leq p}\max_{1\leq k\leq N}\max_{1\leq t\leq T}|z_{k,t,l}|=O_p(\sqrt{\log(pNT)})$. In total,
\begin{align}
&\enVert[4]{ \frac{1}{\sqrt{N}T}\sum_{i=1}^{N}\sum_{t=1}^{T}z_{i,t}d_{i,t}'[ \pi_{i,t}'(\hat{\gamma}-\gamma)]^2}_{\infty}\leq \max_{1\leq l\leq p}\max_{1\leq k\leq N}\max_{1\leq t\leq T}\sqrt{N}|z_{k,t,l}|\frac{1}{NT}\|\Pi(\hat{\gamma}-\gamma)\|^2\nonumber\\
&=O_p(\sqrt{N\log(pNT)})\frac{1}{NT}\|\Pi(\hat{\gamma}-\gamma)\|^2.\label{align consistency covariance denominator 21b actual result}
\end{align}
Therefore, combining (\ref{align consistency covariance denominator 21a actual result}) and (\ref{align consistency covariance denominator 21b actual result})
\begin{align*}
&|\rho_1'\hat{\Theta}_Z\hat{\Sigma}_{2,N}\rho_2-\rho_1'\hat{\Theta}_Z\tilde{\Sigma}_{2,N}\rho_2|\leq \enVert[1]{ \hat{\Sigma}_{2,N}-\tilde{\Sigma}_{2,N}}_{\infty} O_p(\sqrt{h_1h_2\bar{G}\lambda_{node}^{-\vartheta}})\\
& =O_p\del[1]{\sqrt{h_1h_2\bar{G}\lambda_{node}^{-\vartheta}} \log(pNT)} \frac{1}{\sqrt{NT}}\left\| \Pi(\hat{\gamma}-\gamma) \right\|+O_p\del[1]{ \sqrt{h_1h_2\bar{G}\lambda_{node}^{-\vartheta}N\log(pNT)}} \frac{1}{NT}\|\Pi(\hat{\gamma}-\gamma)\|^2\\
&=o_p(1),
\end{align*}
where the last equality is due to Assumption \ref{assu rates}(b)(iii)-(iv), which establishes (\ref{align consistency covariance denominator 21}).

Next, turn to (\ref{align consistency covariance denominator 22}). Note that
\[|\rho_1'\hat{\Theta}_Z\tilde{\Sigma}_{2,N}\rho_2-\rho_1'\hat{\Theta}_Z\Sigma_{2,N}\rho_2|\leq \enVert[1]{ \tilde{\Sigma}_{2,N}-\Sigma_{2,N}}_{\infty}\enVert[1]{ \hat{\Theta}_{Z}'\rho_1}_1\sqrt{h_2}.  \]
Given (\ref{align ThetaZ rho1 l1 norm}), it suffices to consider
\begin{align*}
\enVert[1]{ \tilde{\Sigma}_{2,N}-\Sigma_{2,N}}_{\infty}&=\max_{1\leq l\leq p}\max_{1\leq k\leq N}\envert[3]{ \frac{1}{\sqrt{N}T}\sum_{i=1}^{N}\sum_{t=1}^{T}(z_{i,t,l}d_{i,t,k}\varepsilon_{i,t}^2-\mathbb{E}[z_{i,t,l}d_{i,t,k}\varepsilon_{i,t}^2] )}\\
&=\max_{1\leq l\leq p}\max_{1\leq k\leq N}\envert[3]{ \frac{1}{\sqrt{N}T}\sum_{t=1}^{T}(z_{k,t,l}\varepsilon_{k,t}^2-\mathbb{E}[z_{k,t,l}\varepsilon_{k,t}^2] )}.
\end{align*}
By subgaussianity, Assumption \ref{assu subgaussian}, we can use the same technique as in (\ref{align bound exponential moments}) in Proposition \ref{prop product of L subgaussian} in Appendix B to prove $\mathbb{E} e^{D| \frac{1}{T}\sum_{t=1}^{T}(z_{k,t,l}\varepsilon_{k,t}^2-\mathbb{E}[z_{k,t,l}\varepsilon_{k,t}^2] )| ^{2/3}}\leq BT$ for some positive constant $B$. Using Markov's inequality, we have for any $\epsilon>0$
\begin{align*}
&\mathbb{P}\del[3]{ \max_{1\leq l\leq p}\max_{1\leq k\leq N}\envert[2]{ \frac{1}{T}\sum_{t=1}^{T}(z_{k,t,l}\varepsilon_{k,t}^2-\mathbb{E}[z_{k,t,l}\varepsilon_{k,t}^2] )} \geq \epsilon} \leq  \sum_{l=1}^{p}\sum_{k=1}^{N}\mathbb{P}\del[3]{\envert[2]{ \frac{1}{T}\sum_{t=1}^{T}(z_{k,t,l}\varepsilon_{k,t}^2-\mathbb{E}[z_{k,t,l}\varepsilon_{k,t}^2] )} \geq \epsilon}\\
&\leq \sum_{l=1}^{p}\sum_{k=1}^{N}\frac{\mathbb{E} e^{D\left| \frac{1}{T}\sum_{t=1}^{T}(z_{k,t,l}\varepsilon_{k,t}^2-\mathbb{E}[z_{k,t,l}\varepsilon_{k,t}^2] )\right|^{2/3} }}{e^{D\epsilon^{2/3}}} \leq BpNTe^{-D\epsilon^{2/3}}.
\end{align*}
Set $\epsilon=\sqrt{M(\log(pNT))^3}$ for some $M>0$ and note that the upper bound of the preceding probability becomes arbitrarily small for $N$ and $M$ sufficiently large. Thus,
\[\max_{1\leq l\leq p}\max_{1\leq k\leq N}\envert[2]{\frac{1}{T}\sum_{t=1}^{T} (z_{k,t,l}\varepsilon_{k,t}^2-\mathbb{E}[z_{k,t,l}\varepsilon_{k,t}^2] )}=O_p\del[1]{ \sqrt{(\log(pNT))^3}} \]
and so
\begin{align}
\enVert[1]{ \tilde{\Sigma}_{2,N}-\Sigma_{2,N}}_{\infty}&=
\frac{1}{\sqrt{N}}\max_{1\leq l\leq p}\max_{1\leq k\leq N}\envert[2]{\frac{1}{T}\sum_{t=1}^{T} (z_{k,t,l}\varepsilon_{k,t}^2-\mathbb{E}[z_{k,t,l}\varepsilon_{k,t}^2] )}=O_p\del[2]{ \sqrt{\frac{(\log(pNT))^3}{N}}} . \label{align Sigma2Ntilde -Sigma2N sup}
\end{align}
In total,
\begin{align*}
&|\rho_1'\hat{\Theta}_Z\tilde{\Sigma}_{2,N}\rho_2-\rho_1'\hat{\Theta}_Z\Sigma_{2,N}\rho_2|=O_p\del[2]{  \sqrt{\frac{(\log(p\vee N\vee T))^3h_1h_2\bar{G}\lambda_{node}^{-\vartheta}}{N}}}=o_p(1),
\end{align*}
where the last equality is due to Assumption \ref{assu rates}(a)(ii), establishing (\ref{align consistency covariance denominator 22}).

We now establish (\ref{align consistency covariance denominator 23}).
\begin{align*}
&|\rho_1'\hat{\Theta}_Z\Sigma_{2,N}\rho_2-\rho_1'\Theta_Z\Sigma_{2,N}\rho_2| \leq \|\Sigma_{2,N}\|_{\infty}\|(\hat{\Theta}_{Z}'-\Theta_Z')\rho_1\|_1\sqrt{h_2}\\
&=\|\Sigma_{2,N}\|_{\infty}O_p(\bar{G}\lambda_{node}^{1-\vartheta}\sqrt{h_1h_2})=O\del[0]{1/\sqrt{N}} O_p(\bar{G}\lambda_{node}^{1-\vartheta}\sqrt{h_1h_2})=o_p(1),
\end{align*}
where the first equality is due to (\ref{align for the use of comparing numerators of t1 and t1 prime}), the second equality is due to the definition of $\Sigma_{2,N}$ and (\ref{eqn product of L subgaussian bounded uniform expectation}), and the last equality is due to Assumption \ref{assu rates}(a)(ii) and \ref{assu more on eigen values}b). Thus, we have established (\ref{align consistency covariance denominator 23}), concluding the proof of that (\ref{align consistency covariance denominator 2}) is $o_p(1)$.

\subsubsection*{(\ref{align consistency covariance denominator 3}) is $o_p(1)$:}

We now prove that (\ref{align consistency covariance denominator 3}) is $o_p(1)$. First,
\begin{align*}
&|\rho_2'\hat{\Sigma}_{3,N}\rho_2-\rho_2'\Sigma_{3,N}\rho_2| \leq \enVert[1]{ \hat{\Sigma}_{3,N}-\Sigma_{3,N}}_{\infty}h_2\leq h_2\del[1]{\enVert[1]{ \hat{\Sigma}_{3,N}-\tilde{\Sigma}_{3,N}}_{\infty}+\enVert[1]{ \tilde{\Sigma}_{3,N}-\Sigma_{3,N}}_{\infty}} ,
\end{align*}
where $\tilde{\Sigma}_{3,N}:=\frac{1}{T}\sum_{i=1}^{N}\sum_{t=1}^{T}\varepsilon_{i,t}^2d_{i,t}d_{i,t}'$. We consider $\enVert[1]{\hat{\Sigma}_{3,N}-\tilde{\Sigma}_{3,N}}_{\infty}$ first.
\begin{align}
&\left\| \hat{\Sigma}_{3,N}-\tilde{\Sigma}_{3,N}\right\|_{\infty}=\enVert[4]{ \frac{1}{T}\sum_{i=1}^{N}\sum_{t=1}^{T}\hat{\varepsilon}_{i,t}^2d_{i,t}d_{i,t}' -\frac{1}{T}\sum_{i=1}^{N}\sum_{t=1}^{T}\varepsilon_{i,t}^2d_{i,t}d_{i,t}'}_{\infty}\nonumber\\
&\leq  2\enVert[4]{ \frac{1}{T}\sum_{i=1}^{N}\sum_{t=1}^{T}d_{i,t}d_{i,t}'\varepsilon_{i,t}\pi_{i,t}'(\hat{\gamma}-\gamma)}_{\infty}+\enVert[4]{ \frac{1}{T}\sum_{i=1}^{N}\sum_{t=1}^{T}d_{i,t}d_{i,t}'[ \pi_{i,t}'(\hat{\gamma}-\gamma)]^2}_{\infty}. \label{align consistency covariance denominator 31a}
\end{align} 
Consider the first term of (\ref{align consistency covariance denominator 31a}). A typical element of $\frac{1}{T}\sum_{i=1}^{N}\sum_{t=1}^{T}d_{i,t}d_{i,t}'\varepsilon_{i,t}\pi_{i,t}'(\hat{\gamma}-\gamma)$ is
\begin{align*}
& \frac{1}{T}\sum_{j=1}^{NT}d_{j,l}d_{j,k}\varepsilon_{j}\pi_{j}'(\hat{\gamma}-\gamma) \leq \frac{1}{T}\del[3]{ \sum_{j=1}^{NT}d_{j,l}^2d_{j,k}^2\varepsilon_{j}^2} ^{1/2}\del[3]{\sum_{j=1}^{NT} [ \pi_{j}'(\hat{\gamma}-\gamma) ]^2}^{1/2} \\
&= \del[3]{\frac{1}{T} \sum_{i=1}^{N}\sum_{t=1}^{T}d_{i,t,l}^2d_{i,t,k}^2\varepsilon_{i,t}^2} ^{1/2}\frac{1}{\sqrt{T}}\left\| \Pi(\hat{\gamma}-\gamma) \right\|= \del[3]{\frac{1}{T} \sum_{t=1}^{T}\varepsilon_{k,t}^2} ^{1/2}\frac{1}{\sqrt{T}}\left\| \Pi(\hat{\gamma}-\gamma) \right\|
\end{align*}
for some $l, k\in \{1,\ldots,N\}$, where the inequality is due to Cauchy-Schwarz inequality. By Assumption \ref{assu subgaussian} we have $\mathbb{P}( |\varepsilon_{i,t}^2|\geq\epsilon)\leq \mathbb{P}( |\varepsilon_{i,t}|\geq \epsilon^{1/2})\leq \frac{1}{2}Ke^{-C\epsilon}$ for every $\epsilon> 0$. It follows from Lemma 2.2.1 in \cite{vandervaartWellner1996} that $\|\varepsilon_{i,t}^2\|_{\psi_1}\leq (1+K/2)/C$ for all $i$ and $t$. Hence, by subadditivity of the Orlicz norm and Jensen's inequality, $\| \varepsilon_{i,t}^2-\mathbb{E}[\varepsilon_{i,t}^2]\| _{\psi_1}\leq 2\|\varepsilon_{i,t}^2\|_{\psi_1}\leq (2+K)/C$. Using the definition of the Orlicz norm, we have $\mathbb{E}\exp (\frac{C}{2+K}|  \varepsilon_{i,t}^2-\mathbb{E}[\varepsilon_{i,t}^2]|) \leq 2$. Use independence of $\varepsilon_{i,t}$ across $t$ to invoke Proposition \ref{prop adapation of Fan} in Appendix B for $D=\frac{C}{2+K}$ and $\alpha=1/3$ to conclude
\begin{align*}
\mathbb{P}\del[3]{\envert[2]{ \sum_{t=1}^{T}(\varepsilon_{i,t}^2-\mathbb{E}[\varepsilon_{i,t}^2] )} \geq T\epsilon }\leq Ae^{- B(\epsilon^{2}T)^{1/3}},
\end{align*}
for positive constants $A$ and $B$. Setting $\epsilon=\sqrt{\frac{M(\log N)^3}{T}}$ for some $M>0$ $\left( \epsilon \gtrsim \frac{1}{\sqrt{T}}\right) $, one has
\begin{align*}
&\mathbb{P}\del[3]{\max_{1\leq k\leq N}\envert[2]{ \sum_{t=1}^{T}(\varepsilon_{k,t}^2-\mathbb{E}[\varepsilon_{k,t}^2] )} \geq T\epsilon }  \leq \sum_{k=1}^{N}\mathbb{P}\del[3]{ \envert[2]{ \sum_{t=1}^{T}(\varepsilon_{k,t}^2-\mathbb{E}[\varepsilon_{k,t}^2] )} \geq T\epsilon }\leq AN^{1-BM^{1/3}}.
\end{align*}
The upper bound of the preceding probability becomes arbitrarily small for $N$ and $M$ sufficiently large. Hence,
\begin{equation}
\label{eqn rates for demeaned e squared}
\max_{1\leq k\leq N}\envert[2]{ \frac{1}{T}\sum_{t=1}^{T}(\varepsilon_{k,t}^2-\mathbb{E}[\varepsilon_{k,t}^2] )}=O_p\del[3]{\sqrt{\frac{(\log N)^3}{T}}}. 
\end{equation}
Furthermore, since $\max_{1\leq k\leq N}\max_{1\leq t\leq T}\mathbb{E}[ \varepsilon_{k,t}^2]\leq\max_{1\leq k\leq N}\max_{1\leq t\leq T}\| \varepsilon_{k,t}^2\|_{\psi_1}\leq (1+K/2)/C=O(1)$
\begin{align}
\max_{1\leq k\leq N}\envert[2]{ \frac{1}{T}\sum_{t=1}^{T}\varepsilon_{k,t}^2}
&\leq \max_{1\leq k\leq N}\envert[2]{ \frac{1}{T}\sum_{t=1}^{T}(\varepsilon_{k,t}^2-\mathbb{E}[\varepsilon_{k,t}^2])}+\max_{1\leq k\leq N} \max_{1\leq t\leq T}\mathbb{E}[\varepsilon_{k,t}^2]=O_p\del[2]{\sqrt{\frac{(\log N)^3}{T}}}+O(1).\label{align T average of varepsilon square}
\end{align}
Therefore,
\begin{align}
&\enVert[4]{ \frac{1}{T}\sum_{i=1}^{N}\sum_{t=1}^{T}d_{i,t}d_{i,t}'\varepsilon_{i,t}\pi_{i,t}'(\hat{\gamma}-\gamma)}_{\infty}= O_p\del[2]{\frac{(\log N)^{3/4}}{T^{1/4}}\vee 1}\frac{1}{\sqrt{T}}\| \Pi(\hat{\gamma}-\gamma) \|. \label{align consistency covariance denominator 31a actual result}
\end{align}
Now consider the second term of (\ref{align consistency covariance denominator 31a}). A typical element of $\frac{1}{T}\sum_{i=1}^{N}\sum_{t=1}^{T}d_{i,t}d_{i,t}'[ \pi_{i,t}'(\hat{\gamma}-\gamma)]^2$ is
\begin{align}
&\frac{1}{T}\sum_{i=1}^{N}\sum_{t=1}^{T}d_{i,t,l}d_{i,t,k}[ \pi_{i,t}'(\hat{\gamma}-\gamma)]^2\leq \max_{1\leq i\leq N}\max_{1\leq t\leq T}|d_{i,t,l}d_{i,t,k}|\frac{1}{T}\|\Pi(\hat{\gamma}-\gamma)\|^2=\frac{1}{T}\|\Pi(\hat{\gamma}-\gamma)\|^2, \label{align consistency covariance denominator 31b actual result}
\end{align}
uniformly over $l, k\in \{1,\ldots,N\}$. Combining  (\ref{align consistency covariance denominator 31a actual result}) and (\ref{align consistency covariance denominator 31b actual result}), we have
\begin{align}
\label{eqn consistency covariance denominator 31}
&\left\| \hat{\Sigma}_{3,N}-\tilde{\Sigma}_{3,N}\right\|_{\infty}=O_p\del[2]{\frac{(\log N)^{3/4}}{T^{1/4}}\vee 1}\frac{1}{\sqrt{T}}\| \Pi(\hat{\gamma}-\gamma) \|+\frac{1}{T}\|\Pi(\hat{\gamma}-\gamma)\|^2.
\end{align}

Next, consider $\left\| \tilde{\Sigma}_{3,N}-\Sigma_{3,N}\right\|_{\infty}$.
\begin{align}
&\left\| \tilde{\Sigma}_{3,N}-\Sigma_{3,N}\right\|_{\infty}=\max_{1\leq l\leq N}\max_{1\leq k\leq N}\envert[2]{ \frac{1}{T}\sum_{i=1}^{N}\sum_{t=1}^{T}(\varepsilon_{i,t}^2d_{i,t,l}d_{i,t,k} -\mathbb{E}[\varepsilon_{i,t}^2d_{i,t,l}d_{i,t,k}])}\nonumber\\
&=\max_{1\leq k\leq N}\envert[2]{ \frac{1}{T}\sum_{t=1}^{T}(\varepsilon_{k,t}^2 -\mathbb{E}[\varepsilon_{k,t}^2])}=O_p\del[3]{ \sqrt{\frac{(\log N)^3}{T}}},\label{align consistency covariance denominator 32}
\end{align} 
where the last equality is due to (\ref{eqn rates for demeaned e squared}). Summing up (\ref{eqn consistency covariance denominator 31}) and (\ref{align consistency covariance denominator 32}) yields
\begin{align*}
&|\rho_2'\hat{\Sigma}_{3,N}\rho_2-\rho_2'\Sigma_{3,N}\rho_2|\\
&= h_2O_p\del[2]{\frac{(\log N)^{3/4}}{T^{1/4}}\vee 1}\frac{1}{\sqrt{T}}\| \Pi(\hat{\gamma}-\gamma) \|+h_2\frac{1}{T}\|\Pi(\hat{\gamma}-\gamma)\|^2+O_p\del[2]{ h_2\sqrt{\frac{(\log N)^3}{T}}}\\
&=o_p(1),
\end{align*}
where the last equality is due to Assumptions \ref{assu rates}(b)(v), which, in turns, implies that (\ref{align consistency covariance denominator 3}) is $o_p(1)$. 

\bigskip

Thus, we have proved (\ref{eqn consistency covariance matrix denominator equivalence}). (\ref{eqn uniform denominator}) then follows trivially since the conclusions of Theorem \ref{thm probabilistic oracle inequality} and Corollary \ref{thm l1 norm consistency} are uniform over the set $\mathcal{F}(s_1,\nu,E)$ and the true parameter vector only entered the above arguments when these results were used.

\subsubsection{Numerators of $t_1$ and $t_1'$}

We now show that the numerators of $t_1$ and $t_1'$ are asymptotically equivalent, i.e.,
\begin{equation}
\label{eqn consistency covariance matrix numerator equivalence}
|\rho'\hat{\Theta}S^{-1}\Pi'\varepsilon-\rho'\Theta S^{-1}\Pi'\varepsilon|=o_p(1).
\end{equation}
Note that 
\begin{align*}
&|\rho'\hat{\Theta}S^{-1}\Pi'\varepsilon-\rho'\Theta S^{-1}\Pi'\varepsilon|  \leq \|\rho'(\hat{\Theta}-\Theta)\|_1\|S^{-1}\Pi'\varepsilon\|_{\infty}=\|\rho_1'(\hat{\Theta}_Z-\Theta_Z)\|_1\|S^{-1}\Pi'\varepsilon\|_{\infty}\\
&=  O_p(\bar{G}\lambda_{node}^{1-\vartheta}\sqrt{h_1}) \del[2]{ \frac{1}{\sqrt{NT}}\left\| Z'\varepsilon\right\|_{\infty}\vee  \frac{1}{\sqrt{T}}\left\| D'\varepsilon\right\|_{\infty}} = O_p(\bar{G}\lambda_{node}^{1-\vartheta}\sqrt{h_1}) O_p\del[0]{ \sqrt{(\log (p\vee N))^3}}=o_p(1),
\end{align*}
where the second equality is due to (\ref{align for the use of comparing numerators of t1 and t1 prime}), and the third equality is due to (\ref{eqn rates for Zepsilon infinity}) and (\ref{eqn rates for Depsilon infinity}), and the last equality is due to Assumption \ref{assu rates}(a)(iii). 

\subsubsection{$t_1'\xrightarrow{d}N(0,1)$}

We now prove that $t_1'$ is asymptotically distributed as a standard normal by verifying (i)-(iii) of Theorem \ref{thm mcleish clt} in Appendix B. Note that 
\[t_1':=\frac{\rho'\Theta S^{-1}\Pi'\varepsilon}{\sqrt{\rho' \Theta \Sigma_{\Pi\varepsilon} \Theta'\rho}}=\frac{\rho'\Theta S^{-1}\sum_{i=1}^{N}\sum_{t=1}^{T}\del[2]{ \begin{array}{c}
z_{i,t}\varepsilon_{i,t}\\
d_{i,t}\varepsilon_{i,t}\\
\end{array}} }{\sqrt{\rho' \Theta \Sigma_{\Pi\varepsilon} \Theta'\rho}}=\frac{\rho'\Theta S^{-1}\sum_{j=1}^{k}\del[2]{ \begin{array}{c}
z_{j}\varepsilon_{j}\\
d_{j}\varepsilon_{j}\\
\end{array}} }{\sqrt{\rho' \Theta \Sigma_{\Pi\varepsilon} \Theta'\rho}},\]
where $k:=NT$. In the proof of Lemma \ref{lemma lower bound on Ant}, we have shown that $t_1'$ is a martingale difference array with variance 
\begin{align*}
\text{var}(t_1')=\mathbb{E}[t_1'^2]=
\frac{\rho'\Theta S^{-1}\mathbb{E}[\Pi'\varepsilon\varepsilon'\Pi]S^{-1}\Theta'\rho}{\rho'\Theta\Sigma_{\Pi\varepsilon}\Theta'\rho}=1
\end{align*}
where we have used the definition of $\Sigma_{\Pi\varepsilon}$. We have already shown in (\ref{align denominator of t1 prime bounded away from zero}) that the denominator of $t_1'$ is uniformly bounded away from zero. Thus, verifying that $t_1'$ satisfies (i) and (ii) of Theorem \ref{thm mcleish clt} in Appendix B is equivalent to verifying that the numerator of $t_1'$ satisfies (i) and (ii) of Theorem \ref{thm mcleish clt}. First, note that
\begin{align}
&\left\| \rho_1'\Theta_{Z}\right\|_1= \enVert[4]{\sum_{j\in H_1}\rho_{1j}\Theta_{Z,j}'}_1 \leq  \sum_{j\in H_1}|\rho_{1j}|\left\|\Theta_{Z,j}'\right\|_1 =O(\sqrt{h_1\bar{G}\lambda_{node}^{-\vartheta}}), \label{align ThetaZ rho1 nonhat l1 norm}
\end{align}
where the last equality is due to (\ref{align rates for max ThetaZj l1 norm}). Next,
\begin{align*}
&\envert[2]{\rho'\Theta S^{-1}\del[2]{ \begin{array}{c}
z_{i,t}\varepsilon_{i,t}\\
d_{i,t}\varepsilon_{i,t}\\
\end{array}}}
\leq \envert[2]{ \rho_1'\Theta_Z\frac{z_{i,t}\varepsilon_{i,t}}{\sqrt{NT}}}
+\envert[2]{ \frac{\rho_{2,i}\varepsilon_{i,t}}{\sqrt{T}}}\leq \left\| \rho_1'\Theta_Z\right\|_1\max_{1\leq l\leq p}\envert[2]{ \frac{z_{i,t,l}\varepsilon_{i,t}}{\sqrt{NT}}} +\frac{\|\rho_2\|_\infty|\varepsilon_{i,t}|}{\sqrt{T}}\\
&\lesssim \sqrt{h_1\bar{G}\lambda_{node}^{-\vartheta}}\max_{1\leq l\leq p}\envert[2]{ \frac{z_{i,t,l}\varepsilon_{i,t}}{\sqrt{NT}}} +\frac{\|\rho_2\|_\infty|\varepsilon_{i,t}|}{\sqrt{T}},
\end{align*}
where the last inequality due to (\ref{align ThetaZ rho1 nonhat l1 norm}). We have already shown in the proof of Lemma \ref{lemma lower bound on Ant} that $z_{i,t,l}\varepsilon_{i,t}$ has uniformly bounded $\psi_1$-Orlicz norm. The same is the case for $\varepsilon_{i,t}$. Hence,
\begin{align*}
&\enVert[2]{ \sqrt{h_1\bar{G}\lambda_{node}^{-\vartheta}}\max_{1\leq l\leq p}\envert[2]{ \frac{z_{i,t,l}\varepsilon_{i,t}}{\sqrt{NT}}} +\frac{\|\rho_2\|_\infty|\varepsilon_{i,t}|}{\sqrt{T}}}_{\psi_1}\leq \sqrt{\frac{h_1\bar{G}\lambda_{node}^{-\vartheta}}{NT}}\enVert[2]{ \max_{1\leq l\leq p}z_{i,t,l}\varepsilon_{i,t} }_{\psi_1}+\frac{\|\rho_2\|_\infty}{\sqrt{T}}\left\| \varepsilon_{i,t}\right\|_{\psi_1}\\
&\lesssim\sqrt{\frac{h_1\bar{G}\lambda_{node}^{-\vartheta}}{NT}} \log(1+p) \max_{1\leq l\leq p}\left\|z_{i,t,l}\varepsilon_{i,t} \right\|_{\psi_1}+\frac{\|\rho_2\|_\infty}{\sqrt{T}}\left\| \varepsilon_{i,t}\right\|_{\psi_1}\lesssim\sqrt{\frac{h_1\bar{G}\lambda_{node}^{-\vartheta}}{NT}} \log(1+ p) +\frac{\|\rho_2\|_\infty}{\sqrt{T}},
\end{align*}
for all $i$ and $T$, where the first rate inequality is due to Lemma 2.2.2 in \cite{vandervaartWellner1996}. Using Lemma 2.2.2 in \cite{vandervaartWellner1996} one more time,
\begin{align*}
&\enVert[2]{\max_{1\leq i\leq N}\max_{1\leq t\leq T}\envert[2]{ \rho'\Theta S^{-1}\del[2]{ \begin{array}{c}
z_{i,t}\varepsilon_{i,t}\\
d_{i,t}\varepsilon_{i,t}\\
\end{array}}}}_{\psi_1}\lesssim\log(1+NT)\sbr[4]{ \sqrt{\frac{h_1\bar{G}\lambda_{node}^{-\vartheta}}{NT}} \log(1+p) +\frac{\|\rho_2\|_\infty}{\sqrt{T}}}=o(1),
\end{align*}
where the last equality is due to Assumption \ref{assu rates}(a)(iv)-(v). Since $\|U\|_{L_r}\leq r!\|U\|_{\psi_1}$ for any random variable $U$ (\cite{vandervaartWellner1996}, p95), we conclude that (i) and (ii) of Theorem \ref{thm mcleish clt} are satisfied.  

We now verify (iii) of Theorem \ref{thm mcleish clt}. That is, 
\begin{align*}
\frac{\sum_{j=1}^{k_N}\sbr[2]{\rho'\Theta S^{-1}\del[2]{ \begin{array}{c}
z_{j}\varepsilon_{j}\\
d_{j}\varepsilon_{j}\\
\end{array}}}^2 }{\rho' \Theta \Sigma_{\Pi\varepsilon} \Theta'\rho}
=
\frac{\rho'\Theta \del[2]{\begin{array}{cc}
\tilde{\Sigma}_{1,N}& \tilde{\Sigma}_{2,N}\\
\tilde{\Sigma}_{2,N}'& \tilde{\Sigma}_{3,N}\\
\end{array}}\Theta'\rho }{\rho' \Theta \Sigma_{\Pi\varepsilon} \Theta'\rho}\xrightarrow{p}1.
\end{align*}
Since we have already shown in (\ref{align denominator of t1 prime bounded away from zero}) that the denominator of $t_1'$ is uniformly bounded away from zero, it suffices to show
\begin{equation}
\label{eqn verify iii of mclesh clt}
\envert[2]{\rho'\Theta \del[2]{ \begin{array}{cc}
\tilde{\Sigma}_{1,N}& \tilde{\Sigma}_{2,N}\\
\tilde{\Sigma}_{2,N}'& \tilde{\Sigma}_{3,N}\\
\end{array}}\Theta'\rho - \rho' \Theta \Sigma_{\Pi\varepsilon} \Theta'\rho} =o_p(1).
\end{equation}
The left-hand side of (\ref{eqn verify iii of mclesh clt}) can be bounded by
\begin{align}
&\envert[2]{\rho'\Theta \del[2]{ \begin{array}{cc}
\tilde{\Sigma}_{1,N}& \tilde{\Sigma}_{2,N}\\
\tilde{\Sigma}_{2,N}'& \tilde{\Sigma}_{3,N}\\
\end{array}}\Theta'\rho - \rho' \Theta \Sigma_{\Pi\varepsilon} \Theta'\rho}\nonumber\\
&\leq |\rho_1'\Theta_Z\tilde{\Sigma}_{1,N}\Theta_Z'\rho_1-\rho_1'\Theta_Z\Sigma_{1,N}\Theta_Z'\rho_1|\label{align asymptotic normality 1}\\
&\quad+2|\rho_1'\Theta_Z\tilde{\Sigma}_{2,N}\rho_2-\rho_1'\Theta_Z\Sigma_{2,N}\rho_2|\label{align asymptotic normality 2}\\
&\quad+|\rho_2'\tilde{\Sigma}_{3,N}\rho_2-\rho_2'\Sigma_{3,N}\rho_2|.\label{align asymptotic normality 3}
\end{align}
Thus, we establish that (\ref{align asymptotic normality 1}),  (\ref{align asymptotic normality 2}) and  (\ref{align asymptotic normality 3}) are $o_p(1)$. Consider (\ref{align asymptotic normality 1}) first. 
\begin{align*}
&|\rho_1'\Theta_Z\tilde{\Sigma}_{1,N}\Theta_Z'\rho_1-\rho_1'\Theta_Z\Sigma_{1,N}\Theta_Z'\rho_1|\leq \enVert[0]{ \tilde{\Sigma}_{1,N}-\Sigma_{1,N}}_{\infty}\enVert[0]{ \Theta_Z'\rho_1}_1^2\\
&= O_p\del[3]{ \sqrt{\frac{(\log(p^2T))^5}{N}}}O(h_1\bar{G}\lambda_{node}^{-\vartheta})=o_p(1)
\end{align*}
where the first equality is due to (\ref{align ThetaZ rho1 nonhat l1 norm}) and (\ref{align Sigma1Ntilde -Sigma1N sup}), and the last equality is due to Assumption \ref{assu rates}(a)(i). Now consider (\ref{align asymptotic normality 2}). 
\begin{align*}
&|\rho_1'\Theta_Z\tilde{\Sigma}_{2,N}\rho_2-\rho_1'\Theta_Z\Sigma_{2,N}\rho_2|\leq \enVert[0]{ \tilde{\Sigma}_{2,N}-\Sigma_{2,N}}_{\infty}\enVert[0]{ \Theta_Z'\rho_1}_1\|\rho_2\|_1\\
&= O_p\del[4]{ \sqrt{\frac{(\log(pNT))^3 h_1h_2\bar{G}\lambda_{node}^{-\vartheta}}{N}}} =o_p(1), 
\end{align*}
where the first equality is due to (\ref{align Sigma2Ntilde -Sigma2N sup}), and the last equality is due to Assumption \ref{assu rates}(a)(ii). Finally, consider (\ref{align asymptotic normality 3}). 
\begin{align*}
&|\rho_2'\tilde{\Sigma}_{3,N}\rho_2-\rho_2'\Sigma_{3,N}\rho_2|\leq \enVert[0]{ \tilde{\Sigma}_{3,N}-\Sigma_{3,N}}_{\infty}\| \rho_2\|_1^2= O_p\del[2]{ \sqrt{\frac{(\log N)^3}{T}}}O(h_2)=o_p(1), 
\end{align*}
where the first equality is due to (\ref{align consistency covariance denominator 32}), and the last equality is due to Assumption \ref{assu rates}(b)(v). Therefore, we have established (\ref{eqn verify iii of mclesh clt}) and $t_1'$ is asymptotically standard gaussian.

\subsubsection{$t_2=o_p(1)$}

Last, we prove that $t_2=o_p(1)$. Since the denominator of $t_2$ is bounded away from zero by a positive constant with probability approaching one by (\ref{align denominator of t1 prime bounded away from zero}) and (\ref{eqn consistency covariance matrix denominator equivalence}), it suffices to show $\rho'\Delta=o_p(1)$. 
\begin{align*}
& |\rho'\Delta|=\envert[2]{ \sum_{j\in H}\rho_j\Delta_j} \leq \sqrt{h}\max_{j\in H}|\Delta_j|\leq \sqrt{h}\| S( \hat{\gamma}-\gamma)\|_1\max_{j\in H} \enVert[1]{ \hat{\Theta}_j'\Psi_N-\text{I}_{p+N,j}'}_{\infty} \\
& = \sqrt{h}\| S\left( \hat{\gamma}-\gamma\right)\|_1\del[3]{\max_{j\in H_1}\enVert[3]{ \del[3]{\begin{array}{c}
\frac{1}{NT}Z'Z\hat{\Theta}_{Z,j}-e_j\\
\frac{1}{T\sqrt{N}}D'Z\hat{\Theta}_{Z,j} 
\end{array}} }_{\infty} \vee \max_{i\in H_2}\enVert[3]{\del[3]{\begin{array}{c}
\frac{1}{T\sqrt{N}}Z'De_i\\
0
\end{array}}}_{\infty}} \\
& = \sqrt{h}\| S( \hat{\gamma}-\gamma)\|_1\del[3]{\max_{j\in H_1}\del[3]{ \enVert[2]{  \frac{1}{NT}Z'Z\hat{\Theta}_{Z,j}-e_j}_{\infty}\vee \enVert[2]{ \frac{1}{T\sqrt{N}}D'Z \hat{\Theta}_{Z,j}}_{\infty}}\vee \max_{i\in H_2}\enVert[3]{
\frac{1}{T\sqrt{N}}Z'D}_{\infty}}\\
& \leq \sqrt{h}\| S( \hat{\gamma}-\gamma)\|_1\del[3]{\max_{j\in H_1}\del[3]{ \enVert[2]{  \frac{1}{NT}Z'Z\hat{\Theta}_{Z,j}-e_j}_{\infty}\vee \enVert[1]{\hat{\Theta}_{Z,j}}_1\enVert[2]{ \frac{1}{T\sqrt{N}}D'Z }_{\infty}}\vee \max_{i\in H_2}\enVert[3]{
\frac{1}{T\sqrt{N}}Z'D}_{\infty}}
\end{align*}
where $\hat{\Theta}_j$ is the $j$th row of $\hat{\Theta}$ but written as a $(p+N)\times 1$ vector, and $\text{I}_{p+N,j}$ is the $j$th row of $\text{I}_{p+N}$ but written as a $(p+N)\times 1$ vector. Note that
\[\max_{j\in H_1} \enVert[2]{\frac{1}{NT}Z'Z\hat{\Theta}_{Z,j}-e_j}_{\infty} \leq \max_{j\in H_1}\frac{\lambda_{node}}{\hat{\tau}_j^2}=O_p(\lambda_{node}),\]
where the inequality is due to the extended KKT conditions (\ref{eqn extended KKT conditions}), and the equality is due to (\ref{align lemma 2 2}). Recall that by (\ref{align block comparison 2}) we have that for every $\epsilon>0$
\[\mathbb{P}\del[2]{ \max_{1\leq i\leq N}\max_{1\leq l\leq p}\envert[2]{ \frac{1}{\sqrt{N}T}\sum_{t=1}^{T}z_{i,t,l}} \geq \epsilon} \leq \sum_{i=1}^{N}\sum_{l=1}^{p}\mathbb{P}\del[2]{\envert[2]{ \frac{1}{\sqrt{N}T}\sum_{t=1}^{T}z_{i,t,l}} \geq \epsilon } \leq ApNe^{-B\epsilon^2N},\]
for positive constants $A,B$. Setting $\epsilon=\sqrt{\frac{M\log (pN)}{N}}$ ($M>0$) makes the upper bound of the preceding inequality arbitrarily small for sufficiently large $N$ and $M$, such that
\[\enVert[1]{\hat{\Theta}_{Z,j}}_1\enVert[2]{ \frac{1}{T\sqrt{N}}D'Z}_{\infty}=O_p\del[3]{ \sqrt{\frac{\bar{G}\lambda_{node}^{-\vartheta}\log (pN)}{N}}} .\]
Thus, $|\rho'\Delta|=o_p(1)$ by Assumption \ref{assu rates}(c). For later reference,
\begin{equation}
\label{eqn uniform rho Delta rate}
\sup_{\gamma\in \mathcal{F}(s_1,\nu,E)}|\rho'\Delta|=o_p(1)
\end{equation}
by the same reasoning leading to the uniform validity of (\ref{eqn uniform denominator}). 
\end{proof}

\subsection{Proof of Theorem \ref{thm uniform convergence}}

\begin{proof}[Proof of Theorem \ref{thm uniform convergence}]
For every $\epsilon>0$, define
\begin{align*}
A_{1,N}&:=\cbr[3]{ \sup_{\gamma\in \mathcal{F}(s_1,\nu,E)}|\rho'\Delta|<\epsilon}\qquad A_{2,N}:=\cbr[4]{ \sup_{\gamma\in \mathcal{F}(s_1,\nu,E)}  \envert[4]{ \frac{\sqrt{\rho'\hat{\Theta}\hat{\Sigma}_{\Pi\varepsilon}\hat{\Theta}'\rho}}{\sqrt{\rho'\Theta\Sigma_{\Pi\varepsilon}\Theta'\rho}}-1}<\epsilon } \\
A_{3,N}&:=\cbr[1]{ |\rho'\hat{\Theta}S^{-1}\Pi'\varepsilon-\rho'\Theta S^{-1}\Pi'\varepsilon|<\epsilon}. 
\end{align*}
By (\ref{eqn uniform rho Delta rate}), (\ref{eqn uniform denominator}), (\ref{align denominator of t1 prime bounded away from zero}) and (\ref{eqn consistency covariance matrix numerator equivalence}), the probabilities of the preceding three events all tend to one. Thus, for every $t\in \mathbb{R}$,
\begin{align*}
&\envert[4]{ \mathbb{P}\del[4]{ \frac{\rho'S\left( \tilde{\gamma}-\gamma\right)}{\sqrt{\rho'\hat{\Theta}\hat{\Sigma}_{\Pi\varepsilon}\hat{\Theta}'\rho}}\leq t} - \Phi(t) }\\
&\leq \envert[4]{ \mathbb{P}\del[4]{ \frac{\rho'\hat{\Theta}S^{-1}\Pi'\varepsilon}{\sqrt{\rho'\hat{\Theta}\hat{\Sigma}_{\Pi\varepsilon}\hat{\Theta}'\rho}}-\frac{\rho'\Delta}{\sqrt{\rho'\hat{\Theta}\hat{\Sigma}_{\Pi\varepsilon}\hat{\Theta}'\rho}}\leq t, A_{1,N},A_{2,N},A_{3,N} } - \Phi(t) }+\mathbb{P}\left(\cup_{i=1}^3 A_{i,N}^c \right). 
\end{align*}

We consider $\mathbb{P}\del[2]{ \frac{\rho'\hat{\Theta}S^{-1}\Pi'\varepsilon}{\sqrt{\rho'\hat{\Theta}\hat{\Sigma}_{\Pi\varepsilon}\hat{\Theta}'\rho}}-\frac{\rho'\Delta}{\sqrt{\rho'\hat{\Theta}\hat{\Sigma}_{\Pi\varepsilon}\hat{\Theta}'\rho}}\leq t, A_{1,N},A_{2,N},A_{3,N} }$ first.
\begin{align*}
&\mathbb{P}\del[4]{ \frac{\rho'\hat{\Theta}S^{-1}\Pi'\varepsilon}{\sqrt{\rho'\hat{\Theta}\hat{\Sigma}_{\Pi\varepsilon}\hat{\Theta}'\rho}}-\frac{\rho'\Delta}{\sqrt{\rho'\hat{\Theta}\hat{\Sigma}_{\Pi\varepsilon}\hat{\Theta}'\rho}}\leq t, A_{1,N},A_{2,N},A_{3,N} }\\
&\leq \mathbb{P}\del[3]{ \frac{\rho'\Theta S^{-1}\Pi'\varepsilon}{\sqrt{\rho'\Theta\Sigma_{\Pi\varepsilon}\Theta'\rho}}\leq t(1+\epsilon)+\frac{\epsilon+\epsilon}{\sqrt{\rho'\Theta\Sigma_{\Pi\varepsilon}\Theta'\rho}}}\leq \mathbb{P}\del[3]{ \frac{\rho'\Theta S^{-1}\Pi'\varepsilon}{\sqrt{\rho'\Theta\Sigma_{\Pi\varepsilon}\Theta'\rho}}\leq t(1+\epsilon)+2D\epsilon}
\end{align*}
for some positive constant $D$, where the first and second inequalities are due to the fact that $\rho'\Theta\Sigma_{\Pi\varepsilon}\Theta'\rho$ is uniformly bounded away from zero, see (\ref{align denominator of t1 prime bounded away from zero}). Since the last inequality in the above does not depend on $\gamma$,
\begin{align*}
&\sup_{\gamma\in \mathcal{F}(s_1,\nu,E)}\mathbb{P}\del[4]{ \frac{\rho'\hat{\Theta}S^{-1}\Pi'\varepsilon}{\sqrt{\rho'\hat{\Theta}\hat{\Sigma}_{\Pi\varepsilon}\hat{\Theta}'\rho}}-\frac{\rho'\Delta}{\sqrt{\rho'\hat{\Theta}\hat{\Sigma}_{\Pi\varepsilon}\hat{\Theta}'\rho}}\leq t, A_{1,N},A_{2,N},A_{3,N} }\\
&\leq \mathbb{P}\del[3]{ \frac{\rho'\Theta S^{-1}\Pi'\varepsilon}{\sqrt{\rho'\Theta\Sigma_{\Pi\varepsilon}\Theta'\rho}}\leq t(1+\epsilon)+2D\epsilon}.
\end{align*}
By the asymptotic normality of $t_1'$, for $N$ sufficiently large,
\begin{align*}
&\sup_{\gamma\in \mathcal{F}(s_1,\nu,E)}\mathbb{P}\del[4]{ \frac{\rho'\hat{\Theta}S^{-1}\Pi'\varepsilon}{\sqrt{\rho'\hat{\Theta}\hat{\Sigma}_{\Pi\varepsilon}\hat{\Theta}'\rho}}-\frac{\rho'\Delta}{\sqrt{\rho'\hat{\Theta}\hat{\Sigma}_{\Pi\varepsilon}\hat{\Theta}'\rho}}\leq t, A_{1,N},A_{2,N},A_{3,N} }\leq \Phi(t(1+\epsilon)+2D\epsilon)+\epsilon.
\end{align*}
As the above arguments are valid for every $\epsilon>0$, we can use the continuity of $q\mapsto \Phi(q)$ to conclude that for every $\delta>0$, one can choose $\epsilon$ sufficiently small such that
\begin{equation}
\label{eqn sup 3A probability upper bound}
\sup_{\gamma\in \mathcal{F}(s_1,\nu,E)}\mathbb{P}\del[4]{ \frac{\rho'\hat{\Theta}S^{-1}\Pi'\varepsilon}{\sqrt{\rho'\hat{\Theta}\hat{\Sigma}_{\Pi\varepsilon}\hat{\Theta}'\rho}}-\frac{\rho'\Delta}{\sqrt{\rho'\hat{\Theta}\hat{\Sigma}_{\Pi\varepsilon}\hat{\Theta}'\rho}}\leq t, A_{1,N},A_{2,N},A_{3,N} }\leq \Phi(t)+\delta+\epsilon.
\end{equation}
We next find a lower bound for $\mathbb{P}\del[2]{ \frac{\rho'\hat{\Theta}S^{-1}\Pi'\varepsilon}{\sqrt{\rho'\hat{\Theta}\hat{\Sigma}_{\Pi\varepsilon}\hat{\Theta}'\rho}}-\frac{\rho'\Delta}{\sqrt{\rho'\hat{\Theta}\hat{\Sigma}_{\Pi\varepsilon}\hat{\Theta}'\rho}}\leq t, A_{1,N},A_{2,N},A_{3,N} }$.
\begin{align*}
&\mathbb{P}\del[4]{ \frac{\rho'\hat{\Theta}S^{-1}\Pi'\varepsilon}{\sqrt{\rho'\hat{\Theta}\hat{\Sigma}_{\Pi\varepsilon}\hat{\Theta}'\rho}}-\frac{\rho'\Delta}{\sqrt{\rho'\hat{\Theta}\hat{\Sigma}_{\Pi\varepsilon}\hat{\Theta}'\rho}}\leq t, A_{1,N},A_{2,N},A_{3,N}}\\
&\geq \mathbb{P}\del[3]{ \frac{\rho'\Theta S^{-1}\Pi'\varepsilon}{\sqrt{\rho'\Theta\Sigma_{\Pi\varepsilon}\Theta'\rho}}\leq t(1-\epsilon)-\frac{\epsilon+\epsilon}{\sqrt{\rho'\Theta\Sigma_{\Pi\varepsilon}\Theta'\rho}},A_{1,N},A_{2,N},A_{3,N}}\\
&\geq \mathbb{P}\del[3]{\frac{\rho'\Theta S^{-1}\Pi'\varepsilon}{\sqrt{\rho'\Theta\Sigma_{\Pi\varepsilon}\Theta'\rho}}\leq t(1-\epsilon)-2D\epsilon,A_{1,N},A_{2,N},A_{3,N}}\\
&\geq \mathbb{P}\del[3]{ \frac{\rho'\Theta S^{-1}\Pi'\varepsilon}{\sqrt{\rho'\Theta\Sigma_{\Pi\varepsilon}\Theta'\rho}}\leq t(1-\epsilon)-2D\epsilon} +\mathbb{P}( \cap_{i=1}^3A_{i,N})-1
\end{align*}
for some positive constant $D$, where the first and second inequalities are due to the fact that $\rho'\Theta\Sigma_{\Pi\varepsilon}\Theta'\rho$ is uniformly bounded away from zero, see (\ref{align denominator of t1 prime bounded away from zero}). Since the last inequality in the above display does not depend on $\gamma$, and $\mathbb{P}( \cap_{i=1}^3A_{i,N})$ can be made arbitrarily close to one for sufficiently large $N$,
\begin{align*}
&\inf_{\gamma\in \mathcal{F}(s_1,\nu,E)}\mathbb{P}\del[4]{ \frac{\rho'\hat{\Theta}S^{-1}\Pi'\varepsilon}{\sqrt{\rho'\hat{\Theta}\hat{\Sigma}_{\Pi\varepsilon}\hat{\Theta}'\rho}}-\frac{\rho'\Delta}{\sqrt{\rho'\hat{\Theta}\hat{\Sigma}_{\Pi\varepsilon}\hat{\Theta}'\rho}}\leq t, A_{1,N},A_{2,N},A_{3,N} }\\
&\geq \mathbb{P}\del[3]{ \frac{\rho'\Theta S^{-1}\Pi'\varepsilon}{\sqrt{\rho'\Theta\Sigma_{\Pi\varepsilon}\Theta'\rho}}\leq t(1-\epsilon)-2D\epsilon}-\epsilon.
\end{align*}
By the asymptotic normality of $t_1'$, for $N$ sufficiently large,
\begin{align*}
&\inf_{\gamma\in \mathcal{F}(s_1,\nu,E)}\mathbb{P}\del[4]{ \frac{\rho'\hat{\Theta}S^{-1}\Pi'\varepsilon}{\sqrt{\rho'\hat{\Theta}\hat{\Sigma}_{\Pi\varepsilon}\hat{\Theta}'\rho}}-\frac{\rho'\Delta}{\sqrt{\rho'\hat{\Theta}\hat{\Sigma}_{\Pi\varepsilon}\hat{\Theta}'\rho}}\leq t, A_{1,N},A_{2,N},A_{3,N} }\geq \Phi\left(  t(1-\epsilon)-2D\epsilon\right) -2\epsilon.
\end{align*}
As the above arguments are valid for every $\epsilon>0$, we can use the continuity of $q\mapsto \Phi(q)$ to conclude that for every $\delta>0$, one can choose $\epsilon$ sufficiently small such that
\begin{equation}
\label{eqn inf 3A probability lower bound}
\inf_{\gamma\in \mathcal{F}(s_1,\nu,E)}\mathbb{P}\del[4]{ \frac{\rho'\hat{\Theta}S^{-1}\Pi'\varepsilon}{\sqrt{\rho'\hat{\Theta}\hat{\Sigma}_{\Pi\varepsilon}\hat{\Theta}'\rho}}-\frac{\rho'\Delta}{\sqrt{\rho'\hat{\Theta}\hat{\Sigma}_{\Pi\varepsilon}\hat{\Theta}'\rho}}\leq t, A_{1,N},A_{2,N},A_{3,N} }\geq \Phi\left(  t\right) -\delta-2\epsilon.
\end{equation}
Thus, by (\ref{eqn sup 3A probability upper bound}), (\ref{eqn inf 3A probability lower bound}) and the fact that $\sup_{\gamma\in \mathcal{F}(s_1,\nu,E)}\mathbb{P}(\cup_{i=1}^3A_{i,N}^c)=\mathbb{P}(\cup_{i=1}^3A_{i,N}^c)=o(1)$, we have proved (\ref{eqn major uniform convergence}) (the uniformity over $t\in\mathbb{R}$ follows from the fact that $\Phi(t)$ is continuous). To see (\ref{align honest confidence interval alpha}), note that
\begin{align*}
&\mathbb{P}\del[2]{\alpha_j \notin \sbr[2]{\tilde{\alpha}_j-z_{1-\delta/2}\frac{\tilde{\sigma}_{\alpha,j}}{\sqrt{NT}},\tilde{\alpha}_j+z_{1-\delta/2}\frac{\tilde{\sigma}_{\alpha,j}}{\sqrt{NT}}}}=\mathbb{P}\del[3]{\envert[3]{ \frac{\sqrt{NT}(\tilde{\alpha}_j-\alpha_j)}{\tilde{\sigma}_{z,j}}}> z_{1-\delta/2}} \\
&\leq1-\mathbb{P}\del[3]{ \frac{\sqrt{NT}(\tilde{\alpha}_j-\alpha_j)}{\tilde{\sigma}_{z,j}}\leq z_{1-\delta/2}}+\mathbb{P}\del[3]{ \frac{\sqrt{NT}(\tilde{\alpha}_j-\alpha_j)}{\tilde{\sigma}_{z,j}}\leq -z_{1-\delta/2}}.
\end{align*}
Thus, taking the supremum over $\gamma\in \mathcal{F}(s_1,\nu,E)$ and letting $N$ tend to infinity yields (\ref{align honest confidence interval alpha}) via (\ref{eqn major uniform convergence}). The proof is the same for (\ref{align honest confidence interval eta}). Next, we turn to (\ref{align confidence interval contract at right rate alpha}).
\begin{align*}
&\sqrt{NT}\sup_{\gamma\in \mathcal{F}(s_1,\nu,E)}\text{diam}\del[2]{ \sbr[2]{\tilde{\alpha}_j-z_{1-\delta/2}\frac{\tilde{\sigma}_{\alpha,j}}{\sqrt{NT}}, \tilde{\alpha}_j+z_{1-\delta/2}\frac{\tilde{\sigma}_{\alpha,j}}{\sqrt{NT}} } }\\
& =2z_{1-\delta/2}\del[2]{ \sqrt{[ \Theta_{Z}\Sigma_{1,N}\Theta_{Z}] _{jj}}+o_p(1)}\leq 2z_{1-\delta/2}\del[3]{ \frac{\sqrt{\text{maxeval}(\Sigma_{1,N})}}{\text{mineval}(\Psi_Z)}+o_p(1)}=O_p(1),
\end{align*}
where the first equality is due to (\ref{eqn uniform denominator}), and the last equality is due to Assumptions \ref{assu more on eigen values}(a) and \ref{assu rates}(d). Similarly, we can prove (\ref{align confidence interval contract at right rate eta}):
\begin{align*}
&\sqrt{T}\sup_{\gamma\in \mathcal{F}(s_1,\nu,E)}\text{diam}\del[2]{ \sbr[2]{\tilde{\eta}_i-z_{1-\delta/2}\frac{\tilde{\sigma}_{\eta,i}}{\sqrt{T}}, \tilde{\eta}_i+z_{1-\delta/2}\frac{\tilde{\sigma}_{\eta,i}}{\sqrt{T}} }} =2z_{1-\delta/2}\del[2]{\sqrt{[ \Sigma_{3,N}] _{ii}}+o_p(1)}\\
&=2z_{1-\delta/2}\del[3]{ \sbr[3]{\frac{1}{T}\sum_{t=1}^{T}\mathbb{E}[\varepsilon_{i,t}^2]}^{1/2}+o_p(1)}=O_p(1),
\end{align*}
where the third equality follows from the arguments above (\ref{align T average of varepsilon square}).
\end{proof}

\section{Appendix B}

\begin{proposition}
\label{lemma difference between two matrices again}
Let $A$ and $B$ be two positive semidefinite $(p-1)\times(p-1)$ matrices and $\delta:=\max_{1\leq l,k \leq p-1}|A_{lk}-B_{lk}|$. For any integer $r\in\{1,\ldots,p-1\}$, one has
\[\kappa^2(B, r)\geq \kappa^2(A, r) -\delta16r.\]
\end{proposition}

\begin{proof}
The proof is exactly the same as that of Lemma \ref{lemma difference between two matrices}.
\end{proof}

\begin{thm}[\cite{fangramaliu2012}]
Let $\alpha\in (0,1)$. Assume that $(X_i, \mathcal{F}_i)_{i= 1}^n$ is a sequence of supermartingale differences satisfying $\sup_i\mathbb{E}[e^{|X_i|^{\frac{2\alpha}{1-\alpha}}}]\leq C_1$ for some constant $C_1\in (0,\infty)$. Define $S_k:=\sum_{i=1}^{k}X_i$. Then, for all $\epsilon>0$,
\[\mathbb{P}\del[2]{\max_{1\leq k\leq n}S_k\geq n\epsilon }\leq C(\alpha,n,\epsilon)e^{-\left( \epsilon/4\right) ^{2\alpha}n^{\alpha}}, \]
where 
\[C(\alpha,n,\epsilon):=2+35C_1\sbr[4]{ \frac{1}{16^{1-\alpha}(n\epsilon^2)^{\alpha}}+\frac{1}{n\epsilon^2}\del[3]{\frac{3(1-\alpha)}{2\alpha} }^{\frac{1-\alpha}{\alpha}} }. \]
\end{thm}

The preceding theorem is not exactly the same as Theorem 2.1 in \cite{fangramaliu2012}, but taken from the proof of Theorem 2.1 in \cite{fangramaliu2012}. This theorem generalises Theorem 3.2 in \cite{lesignevolny2001}.

\begin{proposition}
\label{prop adapation of Fan}
Let $\alpha\in (0,1)$. Assume that $(X_i, \mathcal{F}_i)_{i= 1}^n$ is a sequence of martingale differences satisfying satisfying $\sup_i\mathbb{E}[e^{D|X_i|^{\frac{2\alpha}{1-\alpha}}}]\leq C_1$ for some positive constant $D$. ($C_1$ could change with the sample size $n$.) Then, for all $\epsilon\gtrsim\frac{1}{\sqrt{n}}$,
\[\mathbb{P}\del[3]{\envert[2]{ \sum_{i=1}^{n}X_i} \geq n\epsilon }\leq AC_1e^{- K(\epsilon ^2n)^{\alpha}}, \]
for positive constants $A$ and $K$.
\end{proposition}

\begin{proof}
This proposition is a simple adaptation of preceding theorem. Note that for some positive constant $D$,
\[\mathbb{P}\del[2]{ \sum_{i=1}^{n}X_i\geq n\epsilon}= \mathbb{P}\del[2]{ \sum_{i=1}^{n}D^{\frac{1-\alpha}{2\alpha}}X_i\geq nD^{\frac{1-\alpha}{2\alpha}}\epsilon}=\mathbb{P}\del[2]{ \sum_{i=1}^{n}Y_i\geq n\delta},\]
where $Y_i:= D^{\frac{1-\alpha}{2\alpha}}X_i$ and $\delta:= D^{\frac{1-\alpha}{2\alpha}}\epsilon$. Now $(Y_i)_{i= 1}^n$ is a sequence of martingale differences satisfying $\sup_i\mathbb{E}[e^{|Y_i|^{\frac{2\alpha}{1-\alpha}}}]\leq C_1$. Invoking the preceding theorem, we have
\[\mathbb{P}\del[2]{\sum_{i=1}^{n}Y_i\geq n\delta }\leq C(\alpha,n,\delta)e^{-\left( \delta/4\right) ^{2\alpha}n^{\alpha}}. \]
$(-Y_i)_{i=1}^n$ is also a sequence of martingale differences satisfying the same exponential moment condition. Thus,
\begin{align*}
&\mathbb{P}\del[2]{\envert[2]{ \sum_{i=1}^{n}X_i} \geq n\epsilon}=\mathbb{P}\del[2]{\envert[2]{ \sum_{i=1}^{n}Y_i} \geq n\delta }\leq 2C(\alpha,n,\delta)e^{-\left( \delta/4\right) ^{2\alpha}n^{\alpha}}\\
&=2C(\alpha,n,D^{\frac{1-\alpha}{2\alpha}}\epsilon)e^{-( D^{\frac{1-\alpha}{2\alpha}}\epsilon/4) ^{2\alpha}n^{\alpha}}\leq AC_1e^{-K\epsilon ^{2\alpha}n^{\alpha}},
\end{align*}
for positive constants $A, K$, where the last inequality used that if $\epsilon\gtrsim
\frac{1}{\sqrt{n}}$ then $2C(\alpha,n,D^{\frac{1-\alpha}{2\alpha}}\epsilon)\leq AC_1$ for some positive constant $A$.
\end{proof}

\begin{proposition}
\label{prop product of L subgaussian}
Suppose we have random variables $Z_{l,i,t,j}$ uniformly subgaussian for $l=1,\ldots,L$ ($L\geq 3$ fixed), $i=1,\ldots,N$, $t=1,\ldots,T$ and $j=1,\ldots,p$. Both $p$ and $T$ increase with $N$ (functions of $N$). $Z_{l_1,i_1,t_1,j_1}$ and $Z_{l_2,i_2,t_2,j_2}$ are independent as long as $i_1\neq i_2$ regardless of the values of other subscripts. Then,
\begin{equation}
\label{eqn product of L subgaussian bounded uniform expectation}
\max_{1\leq j\leq p}\max_{1\leq t\leq T}\max_{1\leq i\leq N}\mathbb{E}\sbr[2]{ \prod_{l=1}^LZ_{l,i,t,j}}\leq A=O(1),
\end{equation}
for some positive constant $A$ and
\begin{equation}
\label{eqn product of L subgaussian rates}
\max_{1\leq j \leq p}\envert[2]{ \frac{1}{NT}\sum_{i=1}^{N}\sum_{t=1}^{T}\del[2]{ \prod_{l=1}^LZ_{l,i,t,j}-\mathbb{E}\sbr[2]{ \prod_{l=1}^LZ_{l,i,t,j}} } } =O_p\del[2]{\sqrt{\frac{(\log (pT))^{L+1}}{N}}}.
\end{equation}
\end{proposition}

\begin{proof}
For every $\epsilon\geq 0$, $\mathbb{P}\del[1]{ \envert[0]{\prod_{l=1}^LZ_{l,i,t,j} } \geq\epsilon}\leq  \sum_{l=1}^{L}\mathbb{P}\del[1]{ \envert[0]{ Z_{l,i,t,j}} \geq\epsilon^{1/L}}\leq L\frac{K}{2}e^{-C\epsilon^{2/L}}$ for positive constants $K, C$. Next, using Hölder's inequaliy, we have
\begin{align*}
&\max_{1\leq j\leq p}\max_{1\leq t\leq T}\max_{1\leq i\leq N}\mathbb{E}\sbr[2]{ \prod_{l=1}^LZ_{l,i,t,j}}
\leq
\max_{1\leq j\leq p}\max_{1\leq t\leq T}\max_{1\leq i\leq N}\prod_{l=1}^{L}\del[2]{ \mathbb{E}\left[ |Z_{l,i,t,j}|^{L}\right]}^{\frac{1}{L}}  .
\end{align*}
Uniform subgaussianity implies that $\del[2]{ \mathbb{E}\left[ |Z_{l,i,t,j}|^{L}\right]}^{\frac{1}{L}}$ is uniformly bounded. That is, $\del[2]{ \mathbb{E}\left[ |Z_{l,i,t,j}|^{L}\right]}^{\frac{1}{L}}\leq L!\|Z_{l,i,t,j}\|_{\psi_1}\leq L!(\log 2)^{-1/2}\|Z_{l,i,t,j}\|_{\psi_2}\leq L!(\log 2)^{-1/2}\del[1]{ \frac{1+K/2}{C}}^{1/2}$, where the first two inequalities are taken from p95 of \cite{vandervaartWellner1996}, and the third inequality is due to Lemma 2.2.1 in \cite{vandervaartWellner1996}. (\ref{eqn product of L subgaussian bounded uniform expectation}) then follows. 

For every $\epsilon\geq 0$,
\begin{align*}
&\mathbb{P}\del[2]{ \envert[2]{ \frac{1}{T}\sum_{t=1}^{T}\del[2]{ \prod_{l=1}^LZ_{l,i,t,j}-\mathbb{E}\sbr[2]{ \prod_{l=1}^LZ_{l,i,t,j}} }} \geq \epsilon} \leq \mathbb{P}\del[2]{ \max_{1\leq t\leq T}\envert[2]{ \prod_{l=1}^LZ_{l,i,t,j}}\geq\epsilon-A}\\
& \leq \sum_{t=1}^{T}\mathbb{P}\del[2]{ \envert[2]{ \prod_{l=1}^LZ_{l,i,t,j}}\geq\epsilon-A\wedge \epsilon } \leq\frac{L}{2}TKe^{-C(\epsilon-A\wedge\epsilon)^{2/L}}\leq \frac{L}{2}TKe^{-C[\epsilon^{2/L}-(A\wedge\epsilon)^{2/L}]}\leq TK'e^{-C\epsilon^{2/L}},
\end{align*}
for $K'=\frac{L}{2}Ke^{CA^{2/L}}$ and where the second last inequality is due to subadditivity of the concave function: $(x+y)^{2/L}\leq x^{2/L}+y^{2/L}$ for $x,y\geq 0$, $L\geq 3$. Let $X_{i,j}$ denote $\frac{1}{T}\sum_{t=1}^{T}\del[1]{ \prod_{l=1}^LZ_{l,i,t,j}-\mathbb{E}[ \prod_{l=1}^LZ_{l,i,t,j}] }$. Consider some positive constant $D<C$. 
\begin{align}
&\mathbb{E}\left[ e^{D|X_{i,j}|^{2/L}}\right]=\int_{x\in\mathbb{R}}\int_{0}^{|x|^{2/L}}De^{Ds}ds P(dx)+1=\int_{0}^{\infty}De^{Ds}\mathbb{P}(|X_{i,j}|>s^{L/2})ds+1\nonumber\\
&\leq \int_{0}^{\infty}TK'De^{(D-C)s}ds+1=\frac{TK'D}{C-D}+1\leq BT,\label{align bound exponential moments}
\end{align}
for some positive constant $B$, where the second equality is by Fubini's theorem. Then we can use independence across $i$ to invoke Proposition \ref{prop adapation of Fan} in Appendix B with $\alpha=\frac{1}{L+1}$ and $C_1=BT$, for $\epsilon\gtrsim \frac{1}{\sqrt{N}}$,
\begin{align*}
\mathbb{P}\del[2]{\envert[2]{ \sum_{i=1}^{N}\frac{1}{T}\sum_{t=1}^{T}\del[2]{ \prod_{l=1}^LZ_{l,i,t,j}-\mathbb{E}\sbr[2]{\prod_{l=1}^LZ_{l,i,t,j}} }} \geq N\epsilon }\leq A'Te^{- K\left( \epsilon ^2N\right) ^{\frac{1}{L+1}}}
\end{align*}
for positive constants $A'$ and $K$. Setting $\epsilon=\sqrt{\frac{M(\log (pT))^{L+1}}{N}}\left(\gtrsim \frac{1}{\sqrt{N}}\right)$ for some $M>0$, we have
\begin{align*}
&\mathbb{P}\del[2]{\max_{1\leq l\leq p}\envert[2]{ \sum_{i=1}^{N}\frac{1}{T}\sum_{t=1}^{T}\del[2]{  \prod_{l=1}^LZ_{l,i,t,j}-\mathbb{E}\sbr[2]{ \prod_{l=1}^LZ_{l,i,t,j}}} } \geq N\epsilon }\leq pA'Te^{- K\left( \epsilon ^2N\right) ^{\frac{1}{L+1}}}=A'(pT)^{1-KM^{\frac{1}{L+1}}}.
\end{align*}
The upper bound of the preceding probability becomes arbitrarily small for $N$ and $M$ sufficiently large. Hence (\ref{eqn product of L subgaussian rates}) follows.
\end{proof}

\begin{lemma}
\label{lemma vandegeer 6.1}
Let $A$ be a symmetric $p\times p$ matrix, and $\hat{v}$ and $v\in \mathbb{R}^p$. Then
\[|\hat{v}'A\hat{v}-v'Av|\leq \|A\|_{\infty}\|\hat{v}-v\|_1^2+2\|Av\|\|\hat{v}-v\|.\]
\end{lemma}

\begin{proof}
See Lemma 6.1 in the working-paper version of \cite{vandegeerbuhlmannritovdezeure2014}.
\end{proof}

\begin{thm}[\cite{mcleish1974}]
\label{thm mcleish clt}
Let $\{X_{n,i}, i=1,...,k_n\}$ be a martingale difference array with respect to the triangular array of $\sigma$-algebras $\{\mathcal{F}_{n,i}, i=0,...,k_n\}$ (i.e., $X_{n,i}$ is $\mathcal{F}_{n,i}$-measurable and $\mathbb{E}[X_{n,i}|\mathcal{F}_{n,i-1}]=0$ almost surely for all $n$ and $i$) satisfying $\mathcal{F}_{n,i-1}\subseteq \mathcal{F}_{n,i}$ for all $n\geq 1$. Assume,
\begin{enumerate}[(i)]
\item $\max_{i\leq k_n}|X_{n,i}|$ is uniformly bounded in $L_2$ norm,
\item $\max_{i\leq k_n}|X_{n,i}|\xrightarrow{p}0$, and
\item $\sum_{i=1}^{k_n}X_{n,i}^2\xrightarrow{p} 1$.
\end{enumerate}
Then, $S_n=\sum_{i=1}^{k_n}X_{n,i}\xrightarrow{d} N(0,1)$ as $n\to \infty$.
\end{thm}

\bibliographystyle{chicagoa}
\bibliography{DynPanel_Biblio}

\end{document}